\documentclass{amsart}
\usepackage[ps2pdf]{hyperref}

\usepackage{amsmath,amssymb}
\usepackage{bbold}

\usepackage{psfrag}
\usepackage{graphicx}
\usepackage{subfigure}
\usepackage{leftidx}

 \usepackage[all,poly]{xy}
\usepackage[latin1]{inputenc}

\usepackage[dvipsnames]{xcolor}

\newtheorem{thm}{Theorem}[section]

\newtheorem{lem}[thm]{Lemma}

\theoremstyle{definition}

\def\1{\hbox{\rm\rlap {1}\hskip .03in{\rm I}}}

\newcommand{\ds}{\displaystyle}

\newcommand{\co}{\colon}
\newcommand{\id}{\mathrm{id}}

\newcommand{\opp}{\mathrm{op}}
\newcommand{\rv}{{\otimes\mathrm{op}}}

\newcommand{\cc}{\mathcal{C}}
\newcommand{\cch}{\mathcal{C}_{\mathrm{hom}}}

\newcommand{\dd}{\mathcal{D}}

\newcommand{\aaa}{\mathcal{A}}
\newcommand{\ee}{\mathcal{E}}

\newcommand{\uu}{\mathcal{U}}

\newcommand{\ff}{\mathcal{F}}

\newcommand{\zz}{\mathcal{Z}}

                     \newcommand{\RR}{\mathbb{R}}

\newcommand{\Ker}{\mathrm{Ker}}

\newcommand{\iso}{\stackrel{\sim}{\longrightarrow}}

\newcommand{\kk}{\Bbbk}
\newcommand{\kt}{$\Bbbk$\nobreakdash-\hspace{0pt}}

\newcommand{\ti}{\mbox{-}\,}

\newcommand{\un}{\mathbb{1}}

\newcommand{\Aut}{\mathrm{Aut}}

\newcommand{\Endcm}{\mathrm{End}_{\mathrm{co}\otimes}}

\newcommand{\End}{\mathrm{End}}
\newcommand{\Hom}{\mathrm{Hom}}

\newcommand{\tr}{\mathrm{tr}}

\newcommand{\lev}{\mathrm{ev}}

\newcommand{\rev}{\widetilde{\mathrm{ev}}}
\newcommand{\lcoev}{\mathrm{coev}}
\newcommand{\rcoev}{\widetilde{\mathrm{coev}}}
\newcommand{\ldual}[1]{#1^{*}}

\newcommand{\scaledraw}[1]{A}
\newcommand{\scaleraisedraw}[2]{A}
\newcommand{\rsdraw}[3]{\raisebox{-#1\height}{\scalebox{#2}{\includegraphics{#3.eps}}}}
\newcommand{\labela}{\renewcommand{\labelenumi}{{\rm (\alph{enumi})}}}

\providecommand{\bysame}{\leavevmode\hbox to3em{\hrulefill}\thinspace}

\begin{document}

\title{On the graded center of graded categories}
\author[V. Turaev]{Vladimir Turaev}
 \address{%
 Vladimir Turaev\newline
  \indent            Department of Mathematics, \newline
\indent  Indiana University \newline
                     \indent Bloomington IN47405 \newline
                     \indent USA \newline
\indent e-mail: vtouraev@indiana.edu}
\author[A. Virelizier]{Alexis Virelizier}
\address{%
 Alexis Virelizier\newline
     \indent         Department of Mathematics, \newline
\indent    University of Montpellier 2\newline
                     \indent 34095 Montpellier Cedex 5r \newline
                     \indent France \newline
\indent e-mail:  virelizi@math.univ-montp2.fr} 
\date{\today}

\begin{abstract}
We study the $G$-centers of $G$-graded monoidal categories
where $G$ is an arbitrary group. We prove that for any spherical
$G$-fusion category $\cc$ over an algebraically closed field such
that the dimension of the neutral component of $\cc$ is non-zero,
the $G$-center of $\cc$ is a $G$-modular category. This generalizes
a theorem of M. M\"uger corresponding to   $G=\{1\}$. We also
  exhibit interesting objects of  the $G$-center.
\end{abstract}
\maketitle

\setcounter{tocdepth}{1} \tableofcontents

\section{Introduction}\label{sec-Intro}

 The study of  group-graded categories was initiated by the first author     \cite{Tu1} with the view towards constructing 3-dimensional
 Homotopy Quantum Field Theory (HQFT) generalizing the   3-dimensional Topological Quantum  Field Theory (TQFT)
 introduced by    E. Witten  and M. Atiyah. An HQFT applies to manifolds and cobordisms equipped with maps to a fixed target space.  HQFTs with target   the Eilenberg-MacLane space $K(G,1)$, where $G$ is a
  group, naturally arise from $G$-graded  categories via  two fundamental constructions based on state-sums  on triangulations
   and  on   surgery, see  \cite{Tu1},
  \cite{TVi2} and references therein.
 The present  paper is a part   of the authors' work on the   following claim: for any  group $G$,  the  state sum
 HQFT associated with a spherical $G$-fusion  category   is isomorphic to the surgery   HQFT associated with the $G$-center of  that category.
We provide here the algebraic background for this claim and
specifically   study the  $G$-centers.

A $G$-graded category is a  monoidal category whose   objects
are equipped with  a  multiplicative grading by elements of
$G$. The objects of $\cc$ graded by $\alpha\in G$ form a full
subcategory, $\cc_\alpha$, of $\cc$  called the $\alpha$-component
of $\cc$. The multiplicativity of the grading  means  that $X\otimes
Y\in \cc_{\alpha\beta}$ for any $X\in \cc_\alpha$, $Y\in \cc_\beta$
with $\alpha,\beta\in G$. The category $\cc_1$ corresponding to
$\alpha=1$ is called the neutral component of $\cc$.

A number of standard notions of the theory of monoidal categories
(corresponding to   $G=\{1\}$)   naturally generalize to this
setting. This leads, in particular, to a notion of a $G$-fusion
category. On the other hand, to define $G$-braidings   in a
$G$-graded category $\cc$, one needs an additional ingredient: an
action of $G$ on $\cc$ by  strong  monoidal auto-equivalences
$\{\varphi_\alpha:\cc\to \cc\}_{\alpha\in G}$ such that
$\varphi_\alpha(\cc_\beta)\subset \cc_{\alpha^{-1}\beta \alpha}$ for
all $\alpha, \beta \in \cc$.  Using this action, called a crossing,
we define $G$-braided and $G$-ribbon     categories. Under the
simplifying assumption that the   crossing is strict, these notions
were first introduced  in  \cite{Tu1}.

The   Drinfeld-Joyal-Street center construction   applies to  any monoidal category $\cc$  and produces a braided monoidal category $\zz(\cc)$, the {\it center} of $\cc$.   An analogue of the center in the setting of  $G$-graded categories was   considered  by Gelaki, Naidu, and Nikshych \cite{GNN}: for   a finite group $G$, they associate to any $G$-fusion category $\cc$  a $G$-braided category $\zz_G(\cc)$.
The construction of   $ \zz_G(\cc)$ as a monoidal category is rather straightforward and applies to an arbitrary group $G$ and any $G$-graded category $\cc$. The objects of $ \zz_G(\cc)$ are pairs $(A,\sigma)$ where $A$ is an object of $\cc$ and $\sigma$ is a  half-braiding in $\cc$ relative to $\cc_1$, that is a   system of isomorphisms $\sigma_Y:A\otimes Y\to Y\otimes A$ in $\cc$   permuting $A$ with  arbitrary objects $Y$ of $\cc_1$.
The morphisms in $ \zz_G(\cc)$ are the morphisms in $\cc$ commuting with the half-braidings. The monoidal product in
$ \zz_G(\cc)$ is essentially the composition of half-braidings.
The difficult part in the  construction of $ \zz_G(\cc)$, requiring additional assumptions on $\cc$,
 concerns the crossing and the $G$-braiding. We define a  crossing and  a $G$-braiding in $ \zz_G(\cc)$ for
 so-called   non-singular     $G$-graded categories $\cc$
  generalizing the Gelaki-Naidu-Nikshych construction (see
  Theorem~\ref{thm-G-center}).

For topological applications, it is   important to study   so-called
$G$-modular categories. A $G$-modular category is a  $G$-fusion
$G$-ribbon $G$-graded category  whose  neutral component
(which is a fusion ribbon   category in the usual sense) has an
invertible $S$-matrix. Our first   main result is the following
modularity theorem.

\begin{thm}\label{thm-G-center-ribbonIntro}
If $\cc $ is a spherical $G$-fusion category over an algebraically closed field and the dimension of the neutral component of $\cc$ is non-zero, then    $\zz_G(\cc)$   is a
$G$-modular  category.
\end{thm}

This theorem is highly non-trivial already in the case  $G=\{1\}$
where it was first proved by   M. M\"uger, see \cite[Theorem
1.2]{Mu}. Our proof of the modularity theorem  heavily uses the
technique of Hopf monads introduced in \cite{BV2}.

In general, it is not easy to exhibit objects of  the
$G$-center  of a
   $G$-graded category~$\cc$.  When $\cc$ is non-singular and    $\cc_1$-centralizable (in the sense defined in this paper),
  the  forgetful functor   $   \zz_G(\cc) \to \cc, (A,\sigma) \mapsto  A$,  has a left adjoint functor $\ff\co\cc \to \zz_G(\cc)$. The objects of $ \zz_G(\cc)$ isomorphic to objects in the image of $\ff $  are said to be free.     Our second group of results yields  explicit computations of the  free objects. In particular, if
    $\cc$ is a $G$-fusion category  over a field, then $\cc$ is non-singular and $\cc_1$-centralizable, and  for any object $X$ of $\cc$,
 $$
 \ff(X) \simeq \bigl( A=\bigoplus_{i \in I_1} i^* \otimes X \otimes i, \,\,\,  \sigma=\{\sigma_Y\co A \otimes Y\to Y\otimes A\}_{Y\in\cc_1} \bigr)
 $$
where  $I_1$ is a representative set of isomorphism classes of simple objects of $\cc_1$ and $Y $ runs over objects of $ \cc_1$. The restriction of $\sigma_Y$ to the direct summand $i^* \otimes X \otimes i \otimes Y$ of $A\otimes Y$ with $i\in I_1$ is computed by
$$
\sigma_{Y}= \sum_{ \lambda }\;
 \psfrag{e}[Bc][Bc]{\scalebox{.9}{$i_\lambda$}}
 \psfrag{i}[Br][Bc]{\scalebox{.9}{$i$}}
 \psfrag{X}[Bc][Bc]{\scalebox{.9}{$X$}}
 \psfrag{Y}[Bc][Bc]{\scalebox{.9}{$Y$}}
 \psfrag{u}[Bc][Bc]{\scalebox{1}{$p_\lambda$}}
 \psfrag{q}[Bc][Bc]{\scalebox{1}{$q_\lambda^*$}}
\rsdraw{.6}{.9}{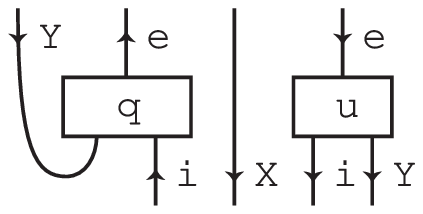}
$$
where $(p_\lambda \co i\otimes Y^* \to i_\lambda, q_\lambda \co
i_\lambda \to i\otimes Y^* )_{\lambda }$ are the  projections   and
the embeddings   determined by a splitting     of $i\otimes Y^*$ as
a direct sum of simple objects  $i_\lambda \in I_1$.  Similar
pictorial formulas compute the  crossing and the  $G$-braiding in $
\zz_G(\cc)$ on the free objects. For example,  for any $\alpha\in
G$,  the object $\varphi_\alpha (\ff(X))$ is computed by the formula
above with $I_1$   replaced everywhere by $I_\alpha$, a
representative set of isomorphism classes of simple objects of
$\cc_\alpha$. For   precise statements, see   Theorems~\ref{lemu}
and~\ref{cor-calc-tau}.

 Every $G$-modular category  gives rise to  a 3-dimensional
HQFT with target
  $K(G,1)$.   The modularity theorem above   allows us to derive such an HQFT from
   $\zz_G(\cc)$ under the conditions of this theorem.    Our computations with free objects  lead to a computation of
the vector spaces assigned by this HQFT to   surfaces equipped with maps to $K(G,1)$.
     These results are crucially used  in the proof (given elsewhere) of the claim stated at the beginning of the introduction.
     In the present paper   we   focus  on  the algebraic side of the theory and do not study HQFTs.

 The organization of the paper is as follows. In Section  \ref{sectprelim} we recall necessary notions  from the theory of monoidal categories.  In Section~\ref{sect-bigdef-G-cat} we introduce  the  key  notions of the theory of $G$-graded categories. In Section \ref{sect-center-Gcat} we construct  the   $G$-center. In Section \ref{sect-main}
 we introduce  spherical $G$-fusion categories, state the modularity theorem, and start its proof.
 In Sections~\ref{sect-ribonness-ZG-new}, \ref{sect-big-HM---}, and \ref{sect-big-HM} we discuss  $G$-ribbonness,
 Hopf monads,  and coends, respectively. In   Section~\ref{sect-bigproof-semiFULL} we finish the proof of the modularity  theorem.
  In Section~\ref{sect-free-functor-big} we   compute the crossing and the  $G$-braiding
   on the free objects and also discuss the $G$-fusion case.  The appendix is devoted to the computation
    of   certain objects of  $\zz_G(\cc)$ which will be instrumental in the study of   the associated HQFT.

  Throughout the paper,  we fix  a  (discrete) group $G$  and a commutative ring $\kk$.

  {\it Acknowledgements.}   The work of V.\ Turaev was partially supported
by the NSF grant DMS-1202335.

\section{Preliminaries on categories and functors}\label{sectprelim}

We recall here several standard notions and techniques of the theory
of monoidal categories referring for details to
\cite{ML1}.

\subsection{Conventions}\label{Conventions}  The   unit  object of   a monoidal category $\cc$ is denoted by~$\un=\un_\cc$.
Notation  $X\in \cc$   means    that $X$ is an object of $\cc$.
To simplify the formulas, we  will always pretend that  the monoidal categories at hand are
strict. Consequently, we  omit brackets in the tenor products and
suppress the associativity constraints $(X\otimes Y)\otimes Z\cong
X\otimes (Y\otimes Z)$ and the unitality constraints $X\otimes \un
\cong X\cong \un \otimes X$.    By the tensor product $X_1 \otimes
X_2 \otimes \cdots \otimes X_n$ of $n\geq 2$ objects $X_1,...,
X_n $ of a monoidal category we mean $(... ((X_1\otimes X_2) \otimes X_3) \otimes
\cdots \otimes X_{n-1}) \otimes X_n$.

\subsection{Monoidal  functors}\label{sect-monofunctor}
Let $\cc$ and $\dd$ be    monoidal categories. A \emph{monoidal
functor} from $\cc$ to $\dd$ is a triple $(F,F_2,F_0)$, where $F\co
\cc \to \dd$ is a functor, $$ F_2=\{F_2(X,Y) \co F(X) \otimes F(Y)
\to F(X \otimes Y)\}_{X,Y \in \cc} $$ is a natural transformation
from $F\otimes F$ to $F \otimes$, and $F_0\co\un_\dd \to F(\un_\cc)$ is a
morphism in~$\dd$  such that the following diagrams  commute for all  $X,Y,Z \in \cc$:
\begin{equation}\label{stmonoidal1}
\begin{split}
    \xymatrix@R=1cm @C=3cm { F(X) \otimes F(Y) \otimes F(Z) \ar[r]^-{\id_{F(X)} \otimes F_2(Y,Z)}  \ar[d]_-{F_2(X,Y) \otimes \id_{F(Z)}}
    & F(X) \otimes F(Y \otimes Z) \ar[d]^-{F_2(X,Y \otimes Z)} \\
    F(X \otimes Y) \otimes F(Z) \ar[r]_-{F_2(X \otimes Y, Z)}
    & F(X \otimes Y \otimes Z),
    }
\end{split}
\end{equation}
\begin{equation}\label{stmonoidal2}
\begin{split}
\xymatrix@R=1cm @C=2,5cm { F(X)  \ar[rd]^-{\id_{F(X)}} \ar[r]^-{\id_{F(X)} \otimes F_0}  \ar[d]_-{F_0 \otimes \id_{F(X)}}
    & F(X) \otimes F(\un_\cc) \ar[d]^-{F_2(X,\un_\cc)} \\
    F(\un_\cc) \otimes F(X) \ar[r]_-{F_2(\un_\cc,X)}
    & F(X).
    }
\end{split}
\end{equation}

 A monoidal functor $(F,F_2,F_0)$ is   \emph{strong}   if $F_2$ and
$F_0$ are isomorphisms and \emph{strict}
if $F_2$ and $F_0$ are   identity morphisms.

 A natural transformation $\varphi=\{\varphi_X \co F(X) \to
G(X)\}_{X \in \cc}$ from a monoidal
functor $F\co\cc \to \dd$ to a monoidal
functor $G\co\cc \to \dd$ is \emph{monoidal} if $G_0=\varphi_\un F_0$ and \begin{equation*}\label{monoidalnattrans}
   \varphi_{X \otimes Y} F_2(X,Y)= G_2(X,Y) (\varphi_X \otimes
\varphi_Y)
\end{equation*}
for all   $X,Y\in \cc$.  A \emph{monoidal natural
isomorphism} between $F$ and $G$ is a monoidal natural
transformation
 $\varphi$ from $F$ to $G$ which is an isomorphism in the sense that each $\varphi_X$ is an isomorphism. The inverse $\varphi^{-1}=\{\varphi_X^{-1} \co G(X) \to F(X)\}_{X \in \cc}$ is then a monoidal natural transformation from $G$ to $F$.


  If $F\co \cc \to \dd$ and $G\co \dd \to \ee$ are   monoidal
functors between monoidal categories, then their composition $GF\co
\cc \to \ee$
is a monoidal functor with 
$$(GF)_0=G(F_0)G_0 \quad \text{and} \quad
(GF)_2=\{G(F_2(X,Y))\, G_2(F(X),F(Y))\}_{X,Y \in \cc}
.$$
%

\subsection{Rigid categories}\label{rigid categories}
Let $\cc=(\cc,\otimes,\un)$ be a monoidal category. A \emph{left dual} of an object
  $X\in \cc$ is an object $\leftidx{^\vee}{\!X}{}\in \cc$ together
with morphisms $\lev_X \co \leftidx{^\vee}{\!X}{} \otimes X \to \un$
and $\lcoev_X \co \un \to X \otimes \leftidx{^\vee}{\! X}{}$ such that
$$
(\id_X \otimes \lev_X)(\lcoev_X \otimes \id_X)=\id_X \quad \text{and} \quad (\lev_X \otimes \id_{\leftidx{^\vee}{\! X}{}})(\id_{\leftidx{^\vee}{\! X}{}} \otimes \lcoev_X)=\id_{\leftidx{^\vee}{\! X}{}}.
$$
One calls $\cc$   \emph{left rigid}   if every object of $\cc$ has a left dual.
A choice of a left  dual  for each object  of   $\cc$ defines a {\it  left  dual functor} $\leftidx{^\vee}{?}{}\co \cc^\opp \to \cc$,  where $\cc^\opp$ is the  category opposite to $\cc$ with opposite monoidal structure.
The functor $\leftidx{^\vee}{?}{}$ carries a morphism $f:X\to Y$ in $\cc$ (i.e., a morphism $Y\to X$ in $\cc^\opp$) to
$$ \leftidx{^\vee}f= (\lev_Y \otimes  {\id_{ {^\vee} X}})({\id_{ { {^\vee} Y}}}  \otimes f \otimes {\id_{  {^\vee} X}})({\id_{  {^\vee} Y}}\otimes \lcoev_X) \co  \leftidx{^\vee}Y \to  \leftidx{^\vee} X.$$
The   functor $\leftidx{^\vee}{?}{}$ is strong monoidal with $(\leftidx{^\vee}{?})_0=\lcoev_\un:\un \to {^\vee}\un$ and  with $$(\leftidx{^\vee}{?})_2(X,Y) \co  \leftidx{^\vee} X\otimes  \leftidx{^\vee} Y\to  \leftidx{^\vee} (X\otimes^\opp  Y)= \leftidx{^\vee} (Y\otimes X)$$
 defined to be  equal to
$$  (\lev_X  \otimes \id_{ {^\vee} (X\otimes Y)} ) (\id_{ {^\vee} X}\otimes \lcoev_{   Y}\otimes \id_{  X  \otimes  {^\vee} (X\otimes Y)} ) (\id_{ {^\vee} X  \otimes {^\vee} Y} \otimes \lcoev_{X\otimes Y}).$$
The isomorphisms $(\leftidx{^\vee}{?})_0$ and $(\leftidx{^\vee}{?})_2(X,Y)$ are called {\it left monoidal constraints}.

Similarly, a \emph{right dual} of $X\in \cc$ is an object $X^\vee$ of $\cc$ equipped
with morphisms $\rev_X \co X \otimes X^\vee \to \un$ and
$\rcoev_X \co \un \to X^\vee \otimes X$ such that
$$
(\rev_X \otimes \id_X)(\id_X \otimes \rcoev_X)=\id_X \quad \text{and} \quad (\id_{X^\vee} \otimes \rev_X)(\rcoev_X \otimes \id_{X^\vee})=\id_{X^\vee}.
$$
One calls $\cc$   \emph{right rigid}   if every object of $\cc$ has a right dual. Similarly to the    above, for a right rigid category $\cc$, one defines  a strong monoidal {\it  right  dual functor} $\leftidx{}{?}{^\vee}\co \cc^\opp \to \cc$ and {\it right monoidal constraints}.

A monoidal category is \emph{rigid} if it is both left rigid and right rigid.  Note that the left and right duals of an object  are unique up to an isomorphism  preserving the (co)evaluation morphisms. Different choices of left/right  dual objects lead to   monoidally isomorphic left/right  dual
functors.

\subsection{Pivotal    categories}\label{pivotall}
By a {\it  pivotal} category we mean a rigid monoidal category $\cc$ such that  the left and   right dual functors are equal as monoidal functors.  Then for each
object $X\in \cc$, we have a \emph{dual
object}~$X^*={^\vee} X=X^\vee$ and   morphisms
\begin{align*}
& \lev_X \co X^*\otimes X \to\un,  \qquad \lcoev_X\co \un  \to X \otimes X^*,\\
&   \rev_X \co X\otimes X^* \to\un, \qquad   \rcoev_X\co \un  \to X^* \otimes X,
\end{align*}
such that $(X^*,\lev_X,\lcoev_X)$ is a left
dual for $X$ and  $(X^*, \rev_X,\rcoev_X)$ is a right
dual for~$X$.   The \emph{dual} $f^* \co Y^* \to X^*$ of any morphism $f\co X \to Y$ in $\cc$   is   computed by
\begin{align*}
f^*&= (\lev_Y \otimes  \id_{X^*})(\id_{Y^*}  \otimes f \otimes \id_{X^*})(\id_{Y^*}\otimes \lcoev_X)\\
 &= (\id_{X^*} \otimes \rev_Y)(\id_{X^*} \otimes f \otimes \id_{Y^*})(\rcoev_X \otimes \id_{Y^*}).
\end{align*}

Working with   a pivotal category $\cc$, we will suppress the duality constraints
$\un^* \cong \un$ and $X^* \otimes Y^*\cong (Y\otimes X)^* $. For
example, we write $(f \otimes g)^*=g^* \otimes f^*$ for
morphisms $f,g$ in~$\cc$.

  For an  endomorphism $g$ of an object
$X$ of a pivotal category $ \cc$, one   defines the {\it left} and  {\it right traces}
$$\tr_l(g)=\lev_X(\id_{\ldual{X}} \otimes g) \rcoev_X  \quad {\text {and}}\quad \tr_r(g)=  \rev_X( g \otimes
\id_{\ldual{X}}) \lcoev_X .$$ Both traces take values in $\End_\cc(\un)$ and are symmetric: $\tr_{l/r}
(fh)=\tr_{l/r} (hf)$ for any morphisms $f:X\to Y$, $h:Y\to X$ in $\cc$. Also $\tr_{l/r}(g)=\tr_{r/l}( {g}^*) $ for any
 endomorphism $g$ of an object. The  {\it left} and  {\it right
 dimensions} of an object $X\in \cc$ are defined by
 $\dim_{l/r} (X)=\tr_{l/r} (\id_X)$. Clearly,
 $\dim_{l/r}(X)=\dim_{r/l}(X^*) $ for all $X$.


\subsection{Penrose graphical calculus}\label{sect-penrose} We will represent morphisms in a category $\cc$ by plane   diagrams to be read from the bottom to the top.
The  diagrams are made of   oriented arcs colored by objects of
$\cc$  and of boxes colored by morphisms of~$\cc$.  The arcs connect
the boxes and   have no mutual intersections or self-intersections.
The identity $\id_X$ of $X\in  \cc$, a morphism $f\co X \to Y$,
and the composition of two morphisms $f\co X \to Y$ and $g\co Y \to
Z$ are represented as follows:
\begin{center}
\psfrag{X}[Bc][Bc]{\scalebox{.7}{$X$}} \psfrag{Y}[Bc][Bc]{\scalebox{.7}{$Y$}} \psfrag{h}[Bc][Bc]{\scalebox{.8}{$f$}} \psfrag{g}[Bc][Bc]{\scalebox{.8}{$g$}}
\psfrag{Z}[Bc][Bc]{\scalebox{.7}{$Z$}} $\id_X=$ \rsdraw{.45}{.9}{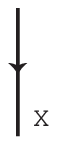}\,,\quad $f=$ \rsdraw{.45}{.9}{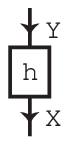} ,\quad \text{and} \quad $gf=$ \rsdraw{.45}{.9}{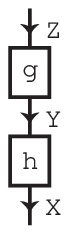}\,.
\end{center}
  If $\cc$ is
monoidal, then the monoidal product of two morphisms $f\co X \to Y$
and $g \co U \to V$ is represented by juxtaposition:
\begin{center}
\psfrag{X}[Bc][Bc]{\scalebox{.7}{$X$}} \psfrag{h}[Bc][Bc]{\scalebox{.8}{$f$}}
\psfrag{Y}[Bc][Bc]{\scalebox{.7}{$Y$}}  $f\otimes g=$ \rsdraw{.45}{.9}{morphism} \psfrag{X}[Bc][Bc]{\scalebox{.8}{$U$}} \psfrag{g}[Bc][Bc]{\scalebox{.8}{$g$}}
\psfrag{Y}[Bc][Bc]{\scalebox{.7}{$V$}} \rsdraw{.45}{.9}{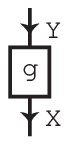}\,.
\end{center}
 Suppose that $\cc$ is      pivotal. By convention,   if an arc colored by $X\in \cc$ is oriented upwards,
then the corresponding object   in the source/target of  morphisms
is $X^*$. For example, $\id_{X^*}$  and a morphism $f\co X^* \otimes
Y \to U \otimes V^* \otimes W$  may be depicted as
\begin{center}
 $\id_{X^*}=$ \, \psfrag{X}[Bl][Bl]{\scalebox{.7}{$X$}}
\rsdraw{.45}{.9}{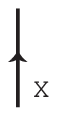} $=$  \,
\psfrag{X}[Bl][Bl]{\scalebox{.7}{$\ldual{X}$}}
\rsdraw{.45}{.9}{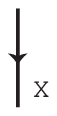}  \quad and \quad
\psfrag{X}[Bc][Bc]{\scalebox{.7}{$X$}}
\psfrag{h}[Bc][Bc]{\scalebox{.8}{$f$}}
\psfrag{Y}[Bc][Bc]{\scalebox{.7}{$Y$}}
\psfrag{U}[Bc][Bc]{\scalebox{.7}{$U$}}
\psfrag{V}[Bc][Bc]{\scalebox{.7}{$V$}}
\psfrag{W}[Bc][Bc]{\scalebox{.7}{$W$}} $f=$
\rsdraw{.45}{.9}{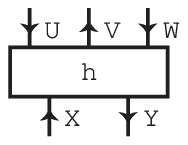} \,.
\end{center}
The duality morphisms   are depicted as follows:
\begin{center}
\psfrag{X}[Bc][Bc]{\scalebox{.7}{$X$}} $\lev_X=$ \rsdraw{.45}{.9}{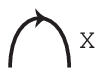}\,,\quad
 $\lcoev_X=$ \rsdraw{.45}{.9}{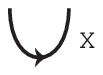}\,,\quad
$\rev_X=$ \rsdraw{.45}{.9}{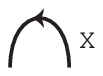}\,,\quad
\psfrag{C}[Bc][Bc]{\scalebox{.7}{$X$}} $\rcoev_X=$
\rsdraw{.45}{.9}{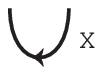}\,.
\end{center}
The dual of a morphism $f\co X \to Y$ and the
  traces of a morphism $g\co X \to X$ can be depicted as
follows:
\begin{center}
\psfrag{X}[Bc][Bc]{\scalebox{.7}{$X$}} \psfrag{h}[Bc][Bc]{\scalebox{.8}{$f$}}
\psfrag{Y}[Bc][Bc]{\scalebox{.7}{$Y$}} \psfrag{g}[Bc][Bc]{\scalebox{.8}{$g$}}
$f^*=$ \rsdraw{.45}{.9}{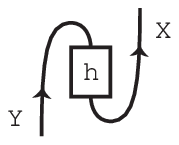}$=$ \rsdraw{.45}{.9}{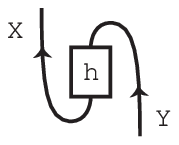}\quad \text{and} \quad
$\tr_l(g)=$ \rsdraw{.45}{.9}{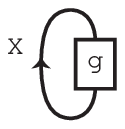}\,,\quad  $\tr_r(g)=$ \rsdraw{.45}{.9}{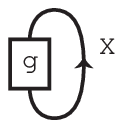}\,.
\end{center}
  It is easy to see that     the morphisms represented by such diagrams
are invariant under isotopies of the diagrams in $\RR^2$  keeping
fixed the bottom and   top endpoints.

\subsection{Pivotal   functors}\label{sect-monofunctorP}
    Given
  a strong monoidal functor $F\co \cc \to \dd$ between pivotal categories,  we define for each $X\in \cc$ a
  morphism
$$
F^l(X)=(F_0^{-1}F(\lev_X)F_2(X^*,X) \otimes \id_{F(X)^*})(\id_{F(X^*)} \otimes \lcoev_{F(X)}) : F(X^*) \to F(X)^*.
$$
It is well-known that $F^l=\{F^l(X) \co F(X^*) \to F(X)^*\}_{X \in
\cc}$ is a monoidal natural isomorphism.
Likewise, the morphisms $ \{F^r(X)
\co F(X^*) \to F(X)^*\}_{X \in \cc}$, defined by
$$
F^r(X)=(\id_{F(X)^*} \otimes F_0^{-1}F(\rev_X)F_2(X,X^*))(\rcoev_{F(X)} \otimes \id_{F(X^*)}),
$$
form a monoidal natural isomorphism $F^r$.
For all $X \in \cc$, we have
\begin{equation*}\label{FrFlcomp}
F^l(X^*)F(\phi_X)=F^r(X)^*\phi_{F(X)}
\end{equation*}
 where $ \{\phi_X \co X \to X^{**}\}_X $ is the
 {\it pivotal structure} in $\cc$ defined by
\begin{equation}\label{pivotal-struct}
\phi_X=(\rev_X \otimes \id_{X^{**}})(\id_X \otimes \lcoev_{X^*})\co X \to X^{**}.
\end{equation}
The monoidal functor $F\co \cc \to \dd$ is     \emph{pivotal} if
$F^l(X)=F^r(X)$ for any $X \in \cc$. In this case, $F^l=F^r$ is
denoted by $F^1$.

\subsection{Additive categories}\label{sect-semisimple-cat}
A  category $\cc$ is  \emph{\kt  additive} if the
$\Hom$-sets of $\cc$ are modules over the
 ring $\kk$,  the  composition of morphisms is \kt bilinear, and any finite family of objects has a direct sum. In particular, such a
   $\cc$ has a \emph{zero object},   that is, an object $\mathbf{0}\in \cc$ such
that $\End_\cc (\mathbf{0})=0$. A  monoidal category is
\emph{$\kk$-additive} if it is $\kk$-additive  as a category and the
monoidal product is \kt bilinear.

A functor $F\co \cc \to \dd$ between \kt additive
categories is \emph{\kt linear} if   the map from $ \Hom_\cc(X,Y) $ to $
\Hom_\dd(F(X),F(Y))$ induced by $F$ is \kt linear for all $X,Y \in \cc$. Such a
functor necessarily preserves direct sums.

\section{$G$-structures on monoidal categories}\label{sect-bigdef-G-cat}

 In this section we define the classes of $G$-graded, $G$-crossed, $G$-braided, and $G$-ribbon categories.

\subsection{$G$-graded categories}\label{G-cat-deb} By a \emph{$G$-graded category (over   $\kk$)}, we mean a
$\kk$-additive  monoidal category $\cc$ endowed with a system of
pairwise disjoint full $\kk$-additive subcategories $\{{\mathcal
C}_{\alpha}\}_{{\alpha}\in G}$ such that
\begin{enumerate}
  \labela
  \item
  each object $X \in \cc$ splits as a  direct sum
$\oplus_{\alpha} \, X_\alpha$ where $X_\alpha\in \cc_{\alpha}$ and
$\alpha$ runs over a finite subset of $G$;

 \item if $X\in {\mathcal C}_{\alpha}$ and  $Y\in {\mathcal C}_{\beta}$, then
$X\otimes Y\in {\mathcal C}_{\alpha\beta}$;

\item if $X\in {\mathcal C}_{\alpha}$ and $Y\in
{\mathcal C}_{\beta}$ with $\alpha\neq \beta$, then $\Hom_{\cc}
(X,Y)=0$;

\item   $\un_\cc\in  {\mathcal C}_1$.

\end{enumerate}

The category ${\mathcal C}_1$ corresponding to the
neutral element $1\in G$ is called the {\it neutral component}
of~${\mathcal C}$. Clearly, ${\mathcal C}_1$ is a $\kk$-additive
monoidal category.

%
%

\subsection{$G$-crossed categories}\label{sect-crossed-ded}

Given a monoidal \kt additive category $\cc$,   denote by
$\Aut(\cc)$ the category whose objects   are \kt linear strong monoidal functors
 $\cc\to \cc$ that are equivalences of categories. The  morphisms  in  $\Aut(\cc)$ are monoidal natural isomorphisms. The category $\Aut(\cc)$ has a canonical structure of a monoidal category, in which
the monoidal product is the composition of monoidal functors and the
monoidal unit is the identity endofunctor $1_\cc$ of~$\cc$.

Denote by $\overline{G}$ the category whose objects are elements of
$G$ and morphisms are identities. We turn $\overline{G}$ into a monoidal category with
monoidal product given by the opposite group multiplication in $G$, i.e.,
$\alpha\otimes \beta= \beta\alpha$ for all $\alpha, \beta \in G$.

 A \emph{$G$-crossed category} is a $G$-graded category $\cc$ (over $\kk$) endowed with a \emph{crossing}, that is, a
strong monoidal functor  $\varphi\co \overline{G} \to
\Aut(\cc) $ such that $ \varphi_\alpha(\cc_\beta) \subset
\cc_{\alpha^{-1} \beta \alpha} $ for all $\alpha,\beta \in G$.
Thus, for each $\alpha\in G$, the crossing $\varphi$ provides a
strong monoidal  equivalence $\varphi_\alpha:\cc\to \cc$  equipped with an isomorphism
$(\varphi_\alpha)_0 \co \un \iso \varphi_\alpha(\un)$ in $\cc$  and
with natural isomorphisms
\begin{align*}
&(\varphi_\alpha)_2=\{(\varphi_\alpha)_2(X,Y) \co \varphi_\alpha(X) \otimes \varphi_\alpha(Y) \iso \varphi_\alpha(X \otimes Y)\}_{X,Y \in \cc},\\
&\varphi_2=\bigl\{\varphi_2(\alpha,\beta)=\{\varphi_2(\alpha,\beta)_X  \co \varphi_\alpha\varphi_\beta(X)\iso \varphi_{\beta\alpha}(X)\}_{X \in \cc}\bigr\}_{\alpha,\beta \in G},\\
&\varphi_0=\{(\varphi_0)_X \co X \iso \varphi_1(X)\}_{X \in \cc}
\end{align*}
such that $(\varphi_0)_\un=(\varphi_1)_0$ and  for all $\alpha,\beta,\gamma \in G$ and all $X,Y,Z
\in \cc$, the following diagrams commute:
\begin{equation}\label{crossing1}
\begin{split}
\xymatrix@R=1cm @C=3cm {
\varphi_\alpha(X) \otimes \varphi_\alpha(Y) \otimes \varphi_\alpha(Z) \ar[r]^-{\id_{\varphi_\alpha(X)} \otimes (\varphi_\alpha)_2(Y,Z)} \ar[d]^{(\varphi_\alpha)_2(X,Y) \otimes \id_{\varphi_\alpha(Z)}} & \varphi_\alpha(X) \otimes \varphi_\alpha(Y \otimes  Z)
\ar[d]^{(\varphi_\alpha)_2(X,Y \otimes Z)}  \\
\varphi_\alpha(X \otimes Y) \otimes \varphi_\alpha(Z) \ar[r]_-{(\varphi_\alpha)_2(X \otimes Y, Z)} & \varphi_\alpha(X \otimes Y \otimes Y),
}
\end{split}
\end{equation}

\begin{equation}\label{crossing2}
\begin{split}
\xymatrix@R=1cm @C=3cm {
\varphi_\alpha(X) \ar[rd]^-{\id_{\varphi_\alpha(X)}} \ar[r]^-{\id_{\varphi_\alpha(X)} \otimes (\varphi_\alpha)_0} \ar[d]_{(\varphi_\alpha)_0
\otimes \id_{\varphi_\alpha(X)}} & \varphi_\alpha(X) \otimes \varphi_\alpha(\un) \ar[d]^{(\varphi_\alpha)_2(X,\un)}\\
\varphi_\alpha(\un) \otimes \varphi_\alpha(X) \ar[r]_-{(\varphi_\alpha)_2(\un,X)} & \varphi_\alpha(X),
}\end{split}
\end{equation}

\def\MyNode{\ifcase\xypolynode\or
      \varphi_{\beta\alpha}(X \otimes Y)
    \or
      \varphi_{\beta\alpha}(X) \otimes \varphi_{\beta\alpha} (Y)
    \or
      \varphi_\alpha\varphi_\beta(X) \otimes \varphi_\alpha\varphi_\beta(Y)
    \or
      \varphi_\alpha(\varphi_\beta(X) \otimes \varphi_\beta(Y))
    \or
      \varphi_\alpha\varphi_\beta(X \otimes Y),
    \fi
  }%
\begin{equation}\label{crossing3}
\begin{split}
  \xy/r10pc/: (0,.3)::
    \xypolygon5{~>{}\txt{\ \ \strut\ensuremath{\MyNode}}} 
    \ar "2";"1" ^-{\qquad (\varphi_{\beta\alpha})_2(X,Y)}
    \ar "3";"2" ^-{\varphi_2(\alpha,\beta)_X \otimes \varphi_2(\alpha,\beta)_Y \qquad\qquad}
    \ar "3";"4" _-{(\varphi_\alpha)_2(\varphi_\beta(X),\varphi_\beta(Y))}
    \ar "4";"5" _-{\varphi_\alpha((\varphi_\beta)_2(X,Y))}
    \ar "5";"1" _-{\varphi_2(\alpha,\beta)_{X\otimes Y}}
  \endxy
\end{split}
\end{equation}

\begin{equation}\label{crossing4}
\begin{split}
\xymatrix@R=1cm @C=2.5cm { \un \ar[d]_-{(\varphi_\alpha)_0} \ar[r]^-{(\varphi_{\beta\alpha})_0}  & \varphi_{\beta\alpha}(\un) \\
\varphi_\alpha(\un) \ar[r]_-{\varphi_\alpha\bigl((\varphi_\beta)_0\bigr) } & \varphi_\alpha\varphi_\beta(\un), \ar[u]_-{\varphi_2(\alpha,\beta)_\un}
}
\end{split}
\end{equation}

\begin{equation}\label{crossing7}
\begin{split}
\xymatrix@R=1cm @C=.5cm { X \otimes Y \ar[rd]_-{(\varphi_0)_X \otimes (\varphi_0)_Y}\ar[rr]^-{(\varphi_0)_{X \otimes Y}} && \varphi_1(X \otimes Y)\\
& \varphi_1(X) \otimes \varphi_1(Y), \ar[ru]_-{(\varphi_1)_2(X,Y)} &
}
\end{split}
\end{equation}

\begin{equation}\label{crossing5}
\begin{split}
\xymatrix@R=1cm @C=2.5cm {
\varphi_\alpha\varphi_\beta \varphi_\gamma(X) \ar[r]^-{\varphi_\alpha(\varphi_2(\beta, \gamma)_X)}\ar[d]_{\varphi_2(\alpha,\beta)_{\varphi_\gamma(X)}} & \varphi_\alpha\varphi_{\gamma\beta}(X)
\ar[d]^{\varphi_2(\alpha, \gamma \beta)_X} \\
\varphi_{\beta\alpha} \varphi_\gamma(X) \ar[r]_-{\varphi_2( \beta\alpha, \gamma)_X} & \varphi_{\gamma\beta\alpha}(X),}
\end{split}
\end{equation}

\begin{equation}\label{crossing6}
\begin{split}
\xymatrix@R=1cm @C=2.5cm {
\varphi_\alpha(X) \ar[rd]^-{\id_{\varphi_\alpha(X)}} \ar[r]^-{\varphi_\alpha((\varphi_0)_X)} \ar[d]_{(\varphi_0)_{\varphi_\alpha(X)}} & \varphi_\alpha\varphi_1(X)  \ar[d]^{\varphi_2(\alpha,1)_X}\\
\varphi_1\varphi_\alpha(X) \ar[r]_-{\varphi_2(1,\alpha)_X} & \varphi_\alpha(X).
}
\end{split}
\end{equation}
The commutativity of the diagrams  \eqref{crossing1} and
\eqref{crossing2} means that
$(\varphi_\alpha,(\varphi_\alpha)_2,(\varphi_\alpha)_0)$ is a
monoidal endofunctor of~$\cc$. The diagrams
\eqref{crossing3} and \eqref{crossing4} indicate  that the natural
transformation $\varphi_2(\alpha,\beta)$ is monoidal. The commutativity of
\eqref{crossing7} and the equality $(\varphi_0)_\un=(\varphi_1)_0$ mean that the natural
transformation $\varphi_0$ is monoidal. Finally, the diagrams
\eqref{crossing5} and \eqref{crossing6} indicate  that
$(\varphi,\varphi_2,\varphi_0)$ is a monoidal functor.

\subsection{$G$-braidings}\label{sect-braided-def}
 An object $X$ of a $G$-graded category $\cc$ is   {\it homogeneous} if   $X\in \cc_\alpha$ for some $\alpha\in
G$. Such an $\alpha$
 is then uniquely determined by $X$ and  denoted $|X|$. If
 two homogeneous objects $X, Y\in \cc $ are
 isomorphic, then either they are zero objects or $|X|=|Y|$.
 Let  $\cch=\amalg_{\alpha\in G}\,  \cc_\alpha$ denote
the full subcategory of homogeneous objects of $\cc$.

A \emph{$G$-braided category} is a $G$-crossed category $(\cc,\varphi)$ endowed with a \emph{$G$-braiding}, i.e., a family of isomorphisms
$$
\tau=\{\tau_{X,Y} \co X \otimes Y \to Y \otimes \varphi_{|Y|}(X)\}_{X \in \cc, Y \in \cch}
$$
 which is
natural in   $X$, $Y$ and satisfies the following three conditions:
\begin{enumerate}
\labela
  \item for all $X \in \cc$ and $Y,Z \in \cch$, the following diagram commutes:
\begin{equation}\label{eq-braiding1}
\begin{split}
\xymatrix@R=1cm @C=3cm {
X \otimes Y \otimes Z \ar[r]^-{\tau_{X, Y \otimes Z}}\ar[d]_{\tau_{X,Y} \otimes \id_Z} & Y \otimes Z \otimes \varphi_{|Y \otimes Z|}(X) \\
Y \otimes \varphi_{|Y|}(X) \otimes Z \ar[r]_-{\id_Y \otimes \tau_{\varphi_{|Y|}(X),Z}} & Y \otimes Z \otimes \varphi_{|Z|}\varphi_{|Y|}(X), \ar[u]_{\id_{Y \otimes Z} \otimes \varphi_2(|Z|,|Y|)_X}
}
\end{split}
\end{equation}
  \item for all $X,Y \in \cc$ and $Z \in \cch$, the following diagram commutes:
\begin{equation}\label{eq-braiding2}
\begin{split}
\xymatrix@R=1cm @C=3cm {
X \otimes Y \otimes Z \ar[r]^-{\tau_{X \otimes Y, Z}}\ar[d]_{\id_X\otimes \tau_{Y,Z}} & Z \otimes \varphi_{|Z|}(X \otimes Y) \\
X \otimes Z \otimes \varphi_{|Z|}(Y) \ar[r]_-{\tau_{X,Z} \otimes \id_{\varphi_{|Z|}(Y)}} & Z \otimes \varphi_{|Z|}(X) \otimes \varphi_{|Z|}(Y), \ar[u]_{\id_Z \otimes (\varphi_{|Z|})_2(X,Y)}
}
\end{split}
\end{equation}
   \item for all $\alpha \in G$, $X \in \cc$, and $Y \in \cch$, the following diagram commutes:
\begin{equation}\label{eq-braiding3}
\begin{split}
\xymatrix@R=1cm @C=3cm{
\varphi_\alpha(X) \otimes \varphi_\alpha(Y) \ar[r]^-{(\varphi_\alpha)_2(X,Y)}\ar[d]_{\tau_{\varphi_\alpha(X),\varphi_\alpha(Y)}} &\varphi_\alpha(X \otimes Y) \ar[d]^-{\varphi_\alpha(\tau_{X,Y})} \\
\varphi_\alpha(Y) \otimes \varphi_{\alpha^{-1}|Y|\alpha}\varphi_\alpha(X)\ar[d]^-{\id_{\varphi_\alpha(Y)} \otimes \varphi_2(\alpha^{-1}|Y|\alpha,\alpha)_X} & \varphi_\alpha(Y \otimes \varphi_{|Y|}(X)) \\
 \varphi_\alpha(Y) \otimes \varphi_{|Y|\alpha}(X) \ar[r]_-{\id_{\varphi_\alpha(Y)} \otimes \varphi_2(\alpha,|Y|)_X^{-1}} &\varphi_\alpha(Y) \otimes \varphi_{\alpha}\varphi_{|Y|}(X). \ar[u]_-{(\varphi_\alpha)_2(Y,\varphi_{|Y|}(X))} \\
}
\end{split}
\end{equation}
\end{enumerate}


For a   $G$-braided category $(\cc,\varphi,\tau)$, the category   $\cc_1$ is a braided category  (in the usual sense of the word) with braiding
\begin{equation}\label{eq-usualbraiding}
\{c_{X,Y}=(\id_Y \otimes (\varphi_0)^{-1}_X)\tau_{X,Y} \co X \otimes Y \to Y \otimes X\}_{X,Y\in \cc_1}\,. \end{equation}

\subsection{Pivotality and ribbonness}\label{sect-pivotality}
 A $G$-graded   category  $\cc$ is \emph{rigid} (resp., \emph{pivotal})   if its underlying monoidal category is rigid (resp.,  pivotal).  If $\cc$ is pivotal, then for   all $X \in \cc_\alpha$ with $\alpha\in
G$, we   always   choose $X^*$ to be in  $ \cc_{\alpha^{-1}}$.
If $\cc$ is   pivotal, then so is~$\cc_1$.

  A  crossing $\varphi\co \overline{G} \to
\Aut(\cc)$ in a pivotal  $G$-graded   category
$\cc$ is  \emph{pivotal} if the monoidal functor  $\varphi_\alpha$
  is pivotal for all $\alpha \in G$. A pivotal
crossing $\varphi$ gives  for each $\alpha \in G$    a natural
isomorphism
$$
\varphi_\alpha^1=\{\varphi_\alpha^1(X) \co \varphi_\alpha(X^*) \iso \varphi_\alpha(X)^*\}_{X \in \cc}
$$
which  preserves both left and right duality  (as in
Section~\ref{sect-monofunctorP})  and is monoidal:
$\bigl((\varphi_\alpha)_0^{-1}\bigr)^*=\varphi_\alpha^1(\un)(\varphi_\alpha)_0:\un^\ast=\un
\to \varphi_\alpha (\un )^\ast$  and for all $X,Y \in \cc$,
$$
\varphi_\alpha^1(X \otimes Y)(\varphi_\alpha)_2(Y^*,X^*)=  \bigl((\varphi_\alpha)_2(X,Y)\bigr)^*(\varphi_\alpha^1(Y) \otimes \varphi_\alpha^1(X)).
$$
A pivotal crossing $\varphi$ preserves the
trace  in the following sense:  for any $\alpha\in G$ and for any
endomorphism $g$ of an object of $\cc$,
$$\tr (\varphi_\alpha (g))= (\varphi_\alpha)_0^{-1} \varphi_\alpha( \tr (g)) (\varphi_\alpha)_0.$$
If  $   \End_\cc(\un)=\kk\, \id_\un$, then this formula implies that  $\tr (\varphi_\alpha (g))=\tr(g)$ for any   $g$.

     A pivotal $G$-braided  category $(\cc,\varphi,\tau)$
has a  \emph{twist} which is the family of morphisms
$\theta=\{\theta_X  \}_{X \in \cch}$ where
\begin{equation}\label{eq-def-twist}
\theta_X =(\lev_X \otimes \id_{\varphi_{|X|}(X)})(\id_{X^*} \otimes \tau_{X,X})(\rcoev_X \otimes \id_X) \co X\to \varphi_{|X|}(X).
\end{equation}
The naturality of $\tau$ implies that   $ \theta_X $ is natural
in $X$.

A \emph{$G$-ribbon category} is a  pivotal $G$-braided category $(\cc,\varphi,\tau)$  whose crossing  $\varphi$ is pivotal and whose  twist   is
\emph{self-dual} in the sense that for all $X \in \cch$,
\begin{equation}\label{eq-ribbon}
(\theta_X)^*= (\varphi_0)_X^* (\varphi_2(|X|^{-1},|X|)_X^{-1})^* \varphi_{|X|^{-1}}^1(\varphi_{|X|}(X)) \theta_{\varphi_{|X|}(X)^*}.
\end{equation}

For a   $G$-ribbon category $(\cc,\varphi,\tau)$, the category   $\cc_1$ is a ribbon category  (in the usual sense of the word) with braiding
\eqref{eq-usualbraiding}  and twist $\{ (\varphi_0)^{-1}_X \theta_X\co X \to X
\}_{X \in \cc_1}$.

%


\subsection{Example}\label{sect-ex-H-vect-rib}
The following example   of a $G$-ribbon category is  adapted from \cite{MNS}. Let $\pi \co H \to G$ be a group epimorphism with kernel $K$.
 Let $\dd$ be the category of $H$-graded finitely generated projective $\kk$-modules $M=\oplus_{h \in H} M_h$ endowed with a right action of   $K$  such that $M_h \cdot k\subset M_{k^{-1}hk}$ for all $h \in H$ and $k \in K$. Since $M$ is finitely generated, $M_h=0$ for all but a finite number of $h \in H$. Morphisms in $\dd$ are $H$-graded $K$-linear morphisms. The category $\dd$ is monoidal: the monoidal product of $M , N \in \dd$ is the  $\kk$-module $M \otimes N=M \otimes_\kk N $ with diagonal $K$-action and $H$-grading  $(M\otimes N)_h=\oplus_{h_1h_2=h} M_{h_1} \otimes_\kk N_{h_2}$ for $h\in  H$. The monoidal unit of $\dd$ is $\kk$ in degree $1\in H$  with  trivial action of $K$. The category $\dd$ is   $\kk$-additive in the obvious way and pivotal: the dual of $M \in \dd$ is the $\kk$-module $M^*=\Hom_\kk(M,\kk)$ with $H$-grading   $(M^*)_h=\Hom(M_{h^{-1}},\kk)$ for $h\in H$ and action of $K$ defined by $(f\cdot k)(m)=f(m\cdot k^{-1})$ for $k \in K$, $f\in M^*$,   $m \in M$. The left and right (co)evaluation morphisms are the usual ones (i.e., are inherited from the pivotal category of finitely generated projective $\kk$-modules). The category $\dd$ is $G$-graded as follows: for $\alpha \in G$,  $\dd_\alpha$ is the full subcategory of all $M\in \dd$ such that $M_h=0$ whenever $\pi(h)\neq \alpha$.


Any set-theoretic section $s$ of $\pi$, i.e., a  map $s\co G \to H$ such that $\pi s=\id_G$ defines a structure of a $G$-ribbon category on~$\dd$ as follows.  For  $\alpha \in G$ and  $M  \in \dd$,  set $\varphi_\alpha(M)=M$ as a $\kk$-module with $H$-grading $\varphi_\alpha(M)_h=M_{s(\alpha)^{-1}hs(\alpha)}$ for $h\in H$ and     right $K$-action     $m  k=m\cdot s(\alpha)ks(\alpha)^{-1}$ for  $m \in  \varphi_\alpha(M)$ and $k\in K$.  For a morphism $f$ in $\dd$, set $\varphi_\alpha(f)=f$. This defines a strict monoidal endofunctor $\varphi_\alpha$ of $\dd$. For $\alpha,\beta \in G$ and $M\in \dd$, the formulas $ m  \mapsto  m \cdot s(\beta)s(\alpha)s(\beta\alpha)^{-1}$ and $m  \mapsto m \cdot s(1)^{-1}$ define  isomorphisms, respectively,
 $$
\varphi_2(\alpha,\beta)_M \co  \varphi_\alpha\varphi_\beta(M)  \to  \varphi_{\beta\alpha}(M) \quad {\text {and}}\quad
(\varphi_0)_M \co   M  \to \varphi_1(M).
$$
This defines a pivotal crossing in $\dd$. Given  $M\in \dd$ and $N \in \dd_\alpha$ with $\alpha\in G$,  the $G$-braiding
$\tau_{M,N} \co  M \otimes N \to  N \otimes \varphi_{\alpha}(M)$
carries $m \otimes n$ to $ n \otimes (m \cdot hs(\alpha)^{-1})
$
for  $m \in M$, $h \in \pi^{-1}(\alpha)$, and $n \in N_h$.   For  $M \in \dd_\alpha$ with $\alpha \in G$,
 the (self-dual) twist $\theta_M\co M \to \varphi_\alpha(M)$ carries $m \in M_h$ with $h \in \pi^{-1}(\alpha)$ to
$ m \cdot hs(\alpha)^{-1}$.   In this way,  $\dd$ becomes a $G$-ribbon category.
 Though the structure of   a $G$-ribbon category on~$\dd$ depends on the choice of a section $s\co G \to H$, an appropriately  defined equivalence class of this structure is independent of  $s$, cf.\
  Section~\ref{sect-ex-Z-H-vect}.

\subsection{Remarks} \label{strictct}
1.  For   $G=\{1\}$,  the definition
of a  $G$-crossed category $\cc$  means  that $\cc$ is a
$\kk$-additive monoidal category   such
that   every object $X\in \cc$ gives rise to an object $X'\in
\cc$ and  an isomorphism $ X\approx X'$. Indeed, a strong
  monoidal functor $\varphi \co \overline{\{1\}}\to \Aut (\cc)$
yields an object  $X'=\varphi_1(X)$ and an isomorphism $
(\varphi_0)_X\co X\to X'$ for each $X\in \cc$. It is easy to see
that any system of objects and isomorphisms $\{(X'\in \cc, X\approx
X')\}_{X\in \cc}$ arises in this way from a unique   strong monoidal
functor $ \overline{\{1\}}\to \Aut (\cc)$. For
$G=\{1\}$, the notions of $G$-braided/$G$-ribbon
categories   are equivalent to the standard notions
of braided/ribbon  categories.

2.    The notions of   $G$-braided/$G$-ribbon  categories  were first introduced   in~\cite{Tu1} in a special case. Denoting by $G^\opp$ the group $G$ with opposite multiplication, a $G$-braided (resp.\@ $G$-ribbon) category in~\cite{Tu1} is a $G^\opp$-braided (resp.\@ $G^\opp$-ribbon) category in the  sense above whose crossing $\varphi$ is \emph{strict}, meaning that $\varphi$ is strict monoidal (i.e., $\varphi_2(\alpha,\beta)$  and $\varphi_0$ are  identity morphisms) and each $\varphi_\alpha$ is strict monoidal (i.e., $(\varphi_\alpha)_2$ and $(\varphi_\alpha)_0$ are identity morphisms). For instance, the   crossing in Example~\ref{sect-ex-H-vect-rib} is strict     if and only if    $s$ is a group homomorphism. Such an $s$ does not exist unless $H$ is a semidirect product of $K$ and $G$.

\section{Centers of $G$-graded categories} \label{sect-center-Gcat}

We  define and study   $G$-centers of $G$-graded categories. We begin
by recalling several   notions of the theory of
categories.

\subsection{Preliminaries}\label{sect-split-idem}
An \emph{idempotent} in a category $\cc$ is an
endomorphism $e$ of an object    $X\in \cc$  such that $e^2=e$. An
idempotent $e\co X \to X$ in $\cc$   \emph{splits} if there is
an object $E\in\cc$ and   morphisms $p\colon X \to E$ and
$q\colon E \to X$ such that $qp=e$ and $pq=\id_E$.
Such a {\it splitting triple} $(E,p,q)$ of $e$ is   unique up to   isomorphism:  if $(E',p' ,q' )$ is another splitting triple of $e$, then $\phi=p'q\co E \to E'$
is the (unique) isomorphism between $E$ and $E'$ such that $p'=\phi p$ and $q'=q \phi^{-1}$.
A
\emph{category with split idempotents} is a category in which
all idempotents split.

  A monoidal category $\cc$ is   \emph{pure} if $f
\otimes \id_X=\id_X \otimes f$  for all
$f \in \End_\cc(\un)$ and $X \in \cc$. For example, this condition is satisfied if $\End_\cc (\un)=\kk \, \id_\un$.

If a monoidal category $\cc$ is    pure and pivotal, then the  left and right traces  in   $\cc$ are
$\otimes$-multiplicative: $ \tr_{l/r} (f\otimes g)=\tr_{l/r} (f) \,
\tr_{l/r}(g)$  for any
endomorphisms $f,g$ of objects~of~$\cc$. In particular,
$\dim_{l/r} (X\otimes Y)= \dim_{l/r} (X) \dim_{l/r} (Y)$   for any $X,Y\in \cc$.

%
%
%

 We call a  pivotal $G$-graded category    $\cc$
\emph{non-singular} if it is pure, has split idempotents, and for all $\alpha \in G$, the subcategory
$\cc_\alpha$ of $\cc$ has  at least  one object whose left dimension is
invertible in $\End_\cc(\un)$.  Examples of such categories will be given in Section \ref{sect-ex-Z-H-vect}.

\subsection{Relative centers}\label{sect-rel-center}  Recall
the notion of a relative center of a monoidal category due to
P.\ Schauenburg \cite{Sch}.  Let $\cc$ be a monoidal category
and $\dd$ be a monoidal subcategory of $\cc$.  A (left)
\emph{half braiding of $\cc$ relative to $\dd$} is a pair $(A,\sigma)$, where  $A \in \cc$ and $ \sigma=\{\sigma_X \co {{A}} \otimes
X\to X \otimes {{A}}\}_{X \in \dd} $ is a family of isomorphisms
in $\cc$ which is natural with respect to $X  $  and
satisfies  for all $X,Y \in \dd$,
 \begin{equation}\label{axiom-half-braiding}
 \sigma_{X \otimes Y}=(\id_X \otimes
\sigma_Y)(\sigma_X \otimes \id_Y).
\end{equation} This implies that $\sigma_\un=\id_{{A}}$.

The (left) \emph{center of $\cc$ relative to $\dd$} is the
monoidal category $\zz(\cc;\dd)$ whose objects  are half
braidings of $\cc$ relative to $\dd$. A morphism
$({{A}},\sigma)\to ({{A}}',\sigma')$ in $\zz(\cc;\dd)$ is a
morphism $f \co {{A}} \to {{A}}'$ in $\cc$ such that $(\id_X
\otimes f)\sigma_X=\sigma'_X(f \otimes \id_X)$ for all
$X\in\dd$. The monoidal product in $\zz(\cc;\dd)$ is defined by
\begin{equation*}
({{A}},\sigma) \otimes ({{B}}, \rho)=\bigl({{A}} \otimes {{B}},(\sigma \otimes \id_{{B}})(\id_{{A}} \otimes \rho) \bigr),
\end{equation*}
and the  unit object of $\zz(\cc;\dd)$ is
$\un_{\zz(\cc;\dd)}=(\un,\{\id_X\}_{X \in \dd})$.    The {\it
forgetful functor} $  \zz(\cc;\dd)\to \cc$ carries $({{A}}, \sigma)$   to ${{A}}\in \cc$ and
acts in the obvious way on the morphisms.   This functor is  strict
monoidal and {\it conservative}   in the sense that a morphism in $  \zz(\cc;\dd)$ carried to an isomorphism in $\cc$ is itself an isomorphism.

The category $\zz(\cc;\dd)$   inherits most of the standard
properties of $\cc$. If   $\cc$ is a  category with split
idempotents, then so it   $\zz(\cc;\dd)$. If $\cc$ is pure, then
$\zz(\cc;\dd)$ is pure and
$\End_{\zz(\cc;\dd)}(\un_{\zz(\cc;\dd)})=\End_\cc(\un)$.
 If $\cc$ is rigid and $\dd$ is a rigid subcategory of $\cc$
(that is, a monoidal subcategory   stable under left and right
dualities), then  $\zz(\cc;\dd)$ is rigid. If $\cc$ is pivotal
and $\dd$ is a pivotal subcategory of $\cc$ (that is, a monoidal
subcategory  stable under duality), then $\zz(\cc;\dd)$ is
pivotal with $({{A}},\sigma)^*=({{A}}^*,\sigma^\dagger)$ for
$(A,\sigma) \in \zz(\cc;\dd)$, where
\begin{equation}\label{eq-inv-sigma}
 \psfrag{M}[Bc][Bc]{\scalebox{.9}{${{A}}$}}
 \psfrag{X}[Bc][Bc]{\scalebox{.9}{$X$}}
 \psfrag{s}[Bc][Bc]{\scalebox{.9}{$\sigma_{X^*}$}}
\sigma^\dagger_X= \rsdraw{.45}{1}{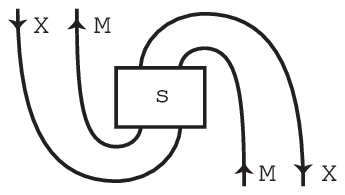}= \;\,\psfrag{s}[Bc][Bc]{\scalebox{.9}{$\sigma^{-1}_{X}$}} \rsdraw{.45}{1}{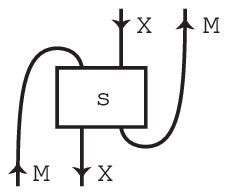} \co {{A}}^* \otimes X \to X \otimes {{A}}^*,
\end{equation}
and $\lev_{({{A}},\sigma)}= \lev_{{A}}$, $\lcoev_{({{A}},\sigma)}= \lcoev_{{A}}$,
$\rev_{({{A}},\sigma)}= \rev_{{A}}$, $\rcoev_{({{A}},\sigma)}= \rcoev_{{A}}$. The
traces of morphisms and   dimensions of objects
in~$\zz(\cc;\dd)$ are the same as in~$\cc$.

If $\cc$ is   \kt additive, then so is $\zz(\cc;\dd)$, and the
forgetful functor $ \zz(\cc;\dd)\to \cc$ is \kt linear. If $\cc$ is
an abelian category, then so is $\zz(\cc;\dd)$.

The center $\zz(\cc;\cc)$ of $\cc$ relative to itself is   the
center $\zz(\cc)$ of~$\cc$ in the sense of A.\ Joyal, R.\ Street,
and V.\ Drinfeld. The center $\zz(\cc;*)$ of $\cc$ relative to its
trivial subcategory $*$ formed by the single object $  \un $ and the
single morphism $ \id_\un $  is canonically isomorphic to $\cc$.
When $\cc$ is pure,     $\zz(\cc;\{\un\})$ is also canonically
isomorphic to~$\cc$,   where $\{\un\}$ is the full subcategory of
$\cc$ having a single object $\un$. The canonical  isomorphism
$\cc\to \zz(\cc;\{\un\})$ carries any object $A$ of $\cc$ to the
half-braiding $(A,\id_A\co   A \otimes \un \to \un \otimes A$); the
naturality of $\id_A$ with respect to endomorphisms of $\un$ is
verified using the purity of $\cc$.

\subsection{The $G$-center}\label{sect-rel-center-G}
 For  a  $G$-graded
 category $\cc$ (over $\kk$), we set $\zz_G(\cc)=\zz(\cc;\cc_1)$   and call $\zz_G(\cc)$  the
\emph{$G$-center of $\cc$}.  By
Section~\ref{sect-rel-center}, the category $\zz_G(\cc )$
is  \kt additive.   For $\alpha \in G$, let
$\zz_\alpha(\cc)$ be the full subcategory of $\zz_G(\cc)$  formed by
the   half braidings $(A, \sigma)$ relative to $\cc_1$ with
$A\in \cc_\alpha$. This system of subcategories  turns $\zz_G(\cc)$ into a $G$-graded category (over $\kk$). By definition,    $|(A,\sigma)|=|A| =\alpha$ for any   $(A,\sigma) \in
\zz_\alpha(\cc) $. If $\cc$ is pivotal (or pure, or with split idempotents), then so is $\zz_G(\cc)$.
The main result of this section is the
following theorem.

\begin{thm}\label{thm-G-center}
Let $\cc$ be a non-singular pivotal
$G$-graded  category. Then  $\zz_G(\cc)$ has
a canonical structure of a pivotal $G$-braided category with pivotal crossing.
\end{thm}

 When $G$ is finite, $\kk$  is a field of characteristic zero, and
$\cc$ is a   $G$-fusion category (see Section \ref{Pre-fusion and fusion  categories} for the definition),
Theorem~\ref{thm-G-center} was first obtained by
Gelaki, Naidu, and Nikshych \cite{GNN}. In difference to \cite{GNN},
we give an explicit construction of  the crossing and the $G$-braiding in
$\zz_G(\cc)$.

Theorem~\ref{thm-G-center} is proved in
Sections~\ref{sect-action of-G} and~\ref{sect-braiding in-G}. Several lemmas are stated in these sections without proof which is postponed to  Section~\ref{sect-proof-lems-defdouble}.
Throughout the proof of Theorem~\ref{thm-G-center}  we keep the assumptions of
this theorem and use the following notation. For   $V\in \cc$, we set $d_V=\dim_l(V)\in \End_\cc (\un)$. For $\alpha \in G$, we  denote by   $\ee_\alpha$    the class of objects of
$\cc_\alpha$ with invertible left dimension.

\subsection{The crossing}\label{sect-action of-G}  The crossing in $\zz_G(\cc)$ is constructed  in three steps. At Step~1, we construct a family of monoidal endofunctors of $\zz_G(\cc)$ numerated by objects of $\cc$ belonging to $\amalg_{\alpha \in G}\, \ee_\alpha$. At Step 2, we construct a system of   isomorphisms between these endofunctors. At Step 3, we define the crossing as the limit of the resulting projective system of endofunctors and isomorphisms.

Step 1. For any
   $V \in \ee_\alpha$ with $\alpha\in G$, we define a  monoidal endofunctor $\varphi_V$ of $\zz_G(\cc)$. We begin with a   lemma.

\begin{lem}\label{lem-Pi-idempot} For  any $(A,\sigma) \in \zz_G(\cc)$, the morphism
$$
\pi^V_{(A,\sigma)}= d_V^{-1}\; \psfrag{A}[Bc][Bc]{\scalebox{.9}{$A$}}
 \psfrag{V}[Bc][Bc]{\scalebox{.9}{$V$}}
 \psfrag{s}[Bc][Bc]{\scalebox{.9}{$\sigma_{V \otimes V^*}$}}
 \rsdraw{.45}{1}{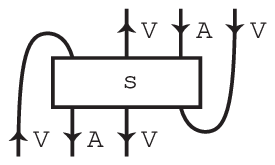} \in \End_\cc(V^* \otimes A \otimes V)
$$
  is an idempotent.
\end{lem}
\begin{proof}
Using \eqref{axiom-half-braiding} and the
naturality of $\sigma$, we obtain
\begin{center}
$\ds \left(\pi^V_{(A,\sigma)}\right)^2= d_V^{-2}\;
\psfrag{A}[Bc][Bc]{\scalebox{.8}{$A$}}
 \psfrag{V}[Bc][Bc]{\scalebox{.8}{$V$}}
 \psfrag{s}[Bc][Bc]{\scalebox{.9}{$\sigma_{V \otimes V^*}$}}
 \rsdraw{.45}{1}{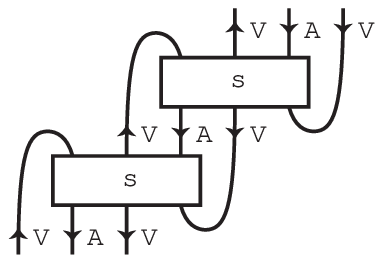}
 \!\!\!=\,d_V^{-2}\; \psfrag{A}[Bc][Bc]{\scalebox{.8}{$A$}}
 \psfrag{V}[Bc][Bc]{\scalebox{.8}{$V$}}
 \psfrag{s}[Bc][Bc]{\scalebox{.9}{$\sigma_{V \otimes V^*\otimes V \otimes V^*}$}}
 \rsdraw{.45}{1}{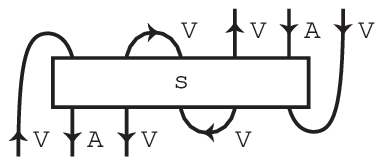}$\\
 $\ds = \, d_V^{-2}\;\; \psfrag{A}[Bc][Bc]{\scalebox{.8}{$A$}}
 \psfrag{V}[Bc][Bc]{\scalebox{.8}{$V$}}
 \psfrag{s}[Bc][Bc]{\scalebox{.9}{$\sigma_{V \otimes V^*}$}}
 \rsdraw{.45}{1}{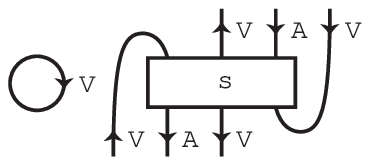}= \pi^V_{(A,\sigma)}.$
\end{center}
\end{proof}
Since  all  idempotents in $\cc$   split,  there exist an
object $E^V_{(A,\sigma)} \in \cc$ and morphisms $p^V_{(A,\sigma)}\co
V^* \otimes A \otimes V \to E^V_{(A,\sigma)}$ and
$q^V_{(A,\sigma)}\co E^V_{(A,\sigma)} \to V^* \otimes A \otimes V$
such that
\begin{equation}\label{eq-idemp-split}
\pi^V_{(A,\sigma)}=q^V_{(A,\sigma)}p^V_{(A,\sigma)}\quad \text{and} \quad p^V_{(A,\sigma)}q^V_{(A,\sigma)}=\id_{E^V_{(A,\sigma)}}.
\end{equation}
We will depict the morphisms $p^V_{(A,\sigma)}$ and $q^V_{(A,\sigma)}$ as
$$
 \psfrag{A}[Bl][Bl]{\scalebox{.8}{$A$}}
 \psfrag{V}[Bl][Bl]{\scalebox{.8}{$V$}}
 \psfrag{E}[Bl][Bl]{\scalebox{.8}{$E_{(A,\sigma)}^V$}}
 p^V_{(A,\sigma)}=\rsdraw{.45}{1}{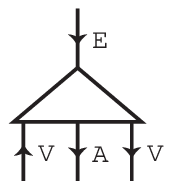}\quad \text{and} \quad  q^V_{(A,\sigma)}=\rsdraw{.45}{1}{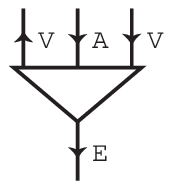} \;.
$$

 We  can now define an  endofunctor $\varphi_V$ of $\zz_G(\cc)$ as follows.   For
$(A,\sigma) $, set  $\varphi_V(A,\sigma)=(E^V_{(A,\sigma)},\gamma^V_{(A,\sigma)}) \in \zz_G(\cc)$  where, for each $X \in \cc_1$,
$$
\gamma^V_{(A,\sigma),X}= d_V^{-1}\; \psfrag{A}[Bc][Bc]{\scalebox{.8}{$A$}}
 \psfrag{V}[Bc][Bc]{\scalebox{.8}{$V$}}
  \psfrag{X}[Bc][Bc]{\scalebox{.8}{$X$}}
   \psfrag{E}[Bl][Bl]{\scalebox{.8}{$E^V_{(A,\sigma)}$}}
 \psfrag{s}[Bc][Bc]{\scalebox{.9}{$\sigma_{V \otimes X \otimes V^*}$}}
 \rsdraw{.45}{1}{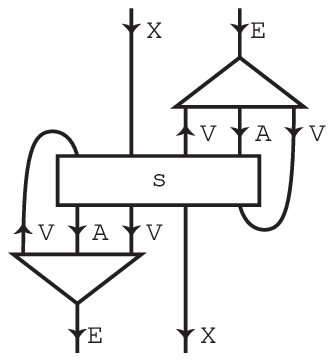} : E^V_{(A,\sigma)} \otimes X\to X\otimes E^V_{(A,\sigma)}.
$$
We show in Section~\ref{sect-proof-lems-defdouble}  that $\gamma^V_{(A,\sigma)}$ is a  half-braiding of $\cc$ relative to $\cc_1$ so that $\varphi_V(A,\sigma)  \in \zz_G(\cc)$.
 If $A  \in \cc_\beta $ with $\beta \in G$, then we
always   choose $E^V_{(A, \sigma)}$  in   $ \cc_{\alpha^{-1} \beta
\alpha}$ so that $\varphi_V(A,\sigma)  \in \zz_{\alpha^{-1} \beta
\alpha} (\cc)$.
For
  a  morphism $f\co (A,\sigma)
\to (B,\rho)$ in $\zz_G(\cc)$, set
$$\varphi_V(f)=
 \psfrag{V}[Bc][Bc]{\scalebox{.8}{$V$}}
 \psfrag{A}[Bc][Bc]{\scalebox{.8}{$A$}}
 \psfrag{B}[Bc][Bc]{\scalebox{.8}{$B$}}
 \psfrag{E}[Bl][Bl]{\scalebox{.8}{$E^V_{(A,\sigma)}$}}
 \psfrag{F}[Bl][Bl]{\scalebox{.8}{$E^V_{(B,\rho)}$}}
 \psfrag{f}[Bc][Bl]{\scalebox{.9}{$f$}}
 \rsdraw{.45}{1}{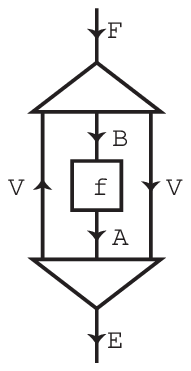} \; \co \varphi_V(A,\sigma) \to \varphi_V(B,\rho).
 $$
 This defines $\varphi_V$ as a functor. To turn $\varphi_V$ into a monoidal functor, set for any $(A,\sigma),(B,\rho)\in  \zz_\beta(\cc)$,
$$(\varphi_V)_2\bigl((A,\sigma),(B,\rho)\bigr)=  \;\psfrag{A}[Bc][Bc]{\scalebox{.8}{$A$}}
 \psfrag{V}[Bc][Bc]{\scalebox{.8}{$V$}}
 \psfrag{B}[Bc][Bc]{\scalebox{.8}{$B$}}
 \psfrag{E}[Bl][Bl]{\scalebox{.8}{$E_{(A,\sigma)}^V$}}
 \psfrag{F}[Bl][Bl]{\scalebox{.8}{$E_{(B,\rho)}^V$}}
 \psfrag{G}[Bl][Bl]{\scalebox{.8}{$E_{(A,\sigma)\otimes (B,\rho)}^V$}}
 \rsdraw{.45}{1}{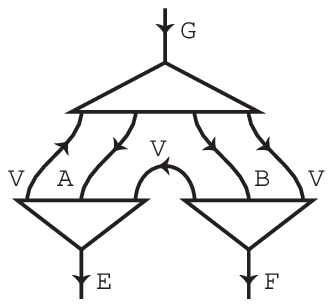} \quad \text{and} \quad
 (\varphi_V)_0= \;
 \psfrag{V}[Bc][Bc]{\scalebox{.8}{$V$}}
 \psfrag{F}[Bl][Bl]{\scalebox{.8}{$E^V_{(\un,\id)}$}}
 \rsdraw{.45}{1}{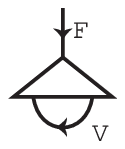}\;.
$$

\begin{lem}\label{lem-phiV-def}
$(\varphi_V,(\varphi_V)_2,(\varphi_V)_0)$ is a well-defined  pivotal strong  monoidal $\kk$-linear  endofunctor of $\zz_G(\cc)$ such that
$\varphi_V\bigl(\zz_{\beta}(\cc)\bigr) \subset \zz_{\alpha^{-1}\beta \alpha}(\cc)$ for all $\beta \in G$.
\end{lem}

 We now study   the endofunctors
$\{ \varphi_V\}_{V}$. Pick
  any    $U \in \ee_\alpha$, $V \in \ee_\beta$,   $W \in
\ee_{\beta \alpha}$ with $\alpha,\beta \in G$. For each $(A,\sigma)\in \zz_G(\cc)$, consider the
morphism
$$
\zeta^{U,V,W}_{(A,\sigma)}=  d_U^{-1}d_V^{-1}\;\psfrag{A}[Bl][Bl]{\scalebox{.7}{$A$}}
 \psfrag{V}[Bl][Bl]{\scalebox{.7}{$U$}}
  \psfrag{R}[Bl][Bl]{\scalebox{.7}{$V$}}
    \psfrag{T}[Bl][Bl]{\scalebox{.7}{$W$}}
   \psfrag{E}[Bl][Bl]{\scalebox{.8}{$E_{\varphi_V(A,\sigma)}^U$}}
      \psfrag{F}[Bl][Bl]{\scalebox{.7}{$E_{(A,\sigma)}^V$}}
         \psfrag{G}[Bl][Bl]{\scalebox{.8}{$E_{(A,\sigma)}^W$}}
 \psfrag{t}[Bc][Bc]{\scalebox{.8}{$p_{(A,\sigma)}^W$}}
  \psfrag{s}[Bc][Bc]{\scalebox{.9}{$\sigma_{V \otimes U \otimes W^*}$}}
 \psfrag{u}[Bc][Bc]{\scalebox{.8}{$q_{(A,\sigma)}^V$}}
 \psfrag{e}[Bc][Bc]{\scalebox{.8}{$q_{\varphi_V(A,\sigma)}^U$}}
 \rsdraw{.45}{1}{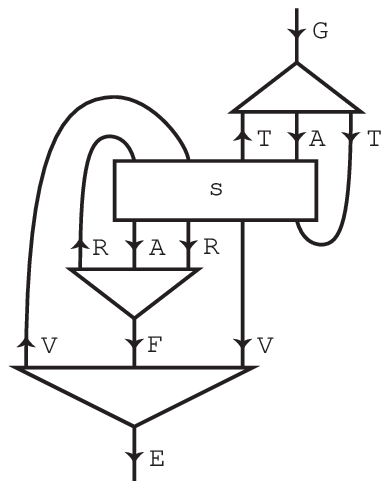} \colon \varphi_U \varphi_V(A,\sigma)  \to \varphi_W (A,\sigma) .
$$

\begin{lem}\label{lem-quotient-varphi}
\begin{enumerate}
\labela
\item   The family  $\zeta^{U,V,W}=\{\zeta^{U,V,W}_{(A,\sigma)}\}_{(A,\sigma)\in
\zz_G(\cc)}$  is a monoidal natural isomorphism from
$\varphi_U\varphi_V$ to $\varphi_W$.

\item  For  all $U \in \ee_\alpha$, $V \in \ee_\beta$, $W \in \ee_{\gamma}$  with $\alpha, \beta, \gamma\in G$ and for all
$R \in \ee_{\beta \alpha}$, $S \in \ee_{\gamma \beta}$, and $T \in
\ee_{\gamma\beta\alpha}$, the following diagram   commutes:
$$
   { \xymatrix@R=1cm @C=2.5cm {
\varphi_U\varphi_V \varphi_W \ar[r]^-{\varphi_U\bigl(\zeta^{V,W,S}\bigr)}\ar[d]_{\zeta^{U,V,R}_{\varphi_W}} & \varphi_U\varphi_S
\ar[d]^{\zeta^{U,S,T}} \\
\varphi_R \varphi_W \ar[r]_-{\zeta^{R,W,T}} & \varphi_T.}}
$$
\end{enumerate}
\end{lem}

For $U \in \ee_1$ and each $(A,\sigma)\in \zz_G(\cc)$, consider the
morphism
$$
\eta^U_{(A,\sigma)}=  \;\psfrag{A}[Bl][Bl]{\scalebox{.8}{$A$}}
 \psfrag{V}[Bl][Bl]{\scalebox{.8}{$U$}}
   \psfrag{E}[Bl][Bl]{\scalebox{.8}{$E_{(A,\sigma)}^U$}}
  \psfrag{s}[Bc][Bc]{\scalebox{.9}{$\sigma_{U}$}}
 \rsdraw{.45}{1}{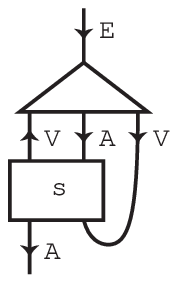} \; \colon (A,\sigma)  \to \varphi_U (A,\sigma).
$$

\begin{lem}\label{lem-quotient-varphi+}
\begin{enumerate}
\labela
\item  The family $\eta^U=\{\eta^U_{(A,\sigma)}\}_{(A,\sigma)\in
\zz_G(\cc)}$  is a monoidal natural isomorphism from
$1_{\zz_G(\cc)}$ to $\varphi_U$.

\item  For all $U \in \ee_1$ and $V \in \ee_\alpha$ with $\alpha\in G$, the following diagram    commutes:
$$
\xymatrix@R=1cm @C=2.5cm {
\varphi_V \ar[rd]^-{\id_{\varphi_V}} \ar[r]^-{\varphi_V\bigl(\eta^U\bigr)} \ar[d]_{\eta^U_{\varphi_V}} & \varphi_V\varphi_U  \ar[d]^{\zeta^{V,U,V}}\\
\varphi_U\varphi_V \ar[r]_-{\zeta^{U,V,V}} & \varphi_V.
}
$$
\end{enumerate}
\end{lem}

Step 2. For each $\alpha\in G$, we   construct    isomorphisms between the  endofunctors  $\{\varphi_V\,\vert \, V\in \ee_\alpha\}$.    For   $U, V \in \ee_\alpha$, define
$\delta^{U,V}=\{\delta^{U,V}_{(A,\sigma)}\}_{(A,\sigma)\in
\zz_G(\cc)}$ by
$$
\delta^{U,V}_{(A,\sigma)}=d_V^{-1}\;
 \psfrag{A}[Bc][Bc]{\scalebox{.8}{$A$}}
 \psfrag{U}[Bc][Bc]{\scalebox{.8}{$V$}}
 \psfrag{V}[Bc][Bc]{\scalebox{.8}{$U$}}
 \psfrag{s}[Bc][Bc]{\scalebox{.9}{$\sigma_{V \otimes U^*}$}}
 \psfrag{E}[Bl][Bl]{\scalebox{.8}{$E_{(A,\sigma)}^V$}}
 \psfrag{F}[Bl][Bl]{\scalebox{.8}{$E_{(A,\sigma)}^U$}}
 \rsdraw{.45}{1}{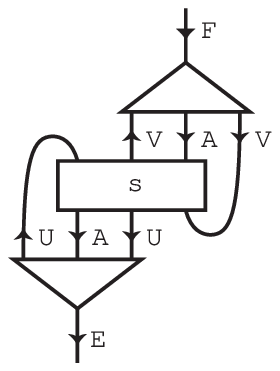} \; \co \varphi_V(A,\sigma) \to \varphi_U(A,\sigma).
$$
\begin{lem}\label{lem-phiV-to-projlim}
\begin{enumerate}
\labela
\item $\delta^{U,V}$ is a monoidal natural isomorphism from $\varphi_U$ to $\varphi_V$.
\item $\delta^{U,U}=\id_{\varphi_U}$ and $\delta^{U,V}\delta^{V,W}=\delta^{U,W}$ for any $U,V,W \in \ee_\alpha$.
 \item For all $U,U' \in \ee_\alpha$, $V,V' \in \ee_\beta$, and $W,W' \in \ee_{\beta \alpha}$, the following diagram of monoidal natural isomorphisms commutes:
$$
    \xymatrix@R=1cm @C=2.5cm {
\varphi_U\varphi_V \ar[r]^-{\zeta^{U,V,W}} \ar[d]_{\varphi_{U'}\bigl(\delta^{V',V}\bigr)\delta^{U',U}_{\varphi_{V}}} & \varphi_W
\ar[d]^{\delta^{W',W}} \\
\varphi_{U'}\varphi_{V'} \ar[r]_-{\zeta^{U',V',W'}} & \varphi_{W'}.}
$$

\item $\delta^{U',U}\eta^{U}=\eta^{U'}$ for all $U,U' \in \ee_1$.
\end{enumerate}
\end{lem}

 Note as a consequence of (a) and (b) that $ (\delta^{U,V}_{(A,\sigma)} )^{-1}=  \delta^{V,U}_{(A,\sigma)}$.

Step 3. By Lemmas~\ref{lem-phiV-def} and \ref{lem-phiV-to-projlim}(a),(b), the
family  $(\varphi_V, \delta^{U,V})_{U,V \in \ee_\alpha}$ is a
projective system in the category of pivotal strong
 monoidal $\kk$-linear endofunctors of $\zz_G(\cc)$. Since all  $\delta^{U,V}$'s are
 isomorphisms, this system has a well-defined
  projective limit
$$
\varphi_\alpha= \varprojlim (\varphi_V, \delta^{U,V})_{U,V \in \ee_\alpha}
$$
which is  a pivotal strong
 monoidal $\kk$-linear endofunctor  of $\zz_G(\cc)$. By Lemma~\ref{lem-phiV-def}, we can
assume that  $\varphi_\alpha\bigl(\zz_\beta(\cc)\bigr) \subset
\zz_{\alpha^{-1}\beta\alpha}(\cc)$ for all $ \beta \in G$.

Denote by $\iota^\alpha=\{\iota^\alpha_V\}_{V \in \ee_\alpha}$  the
 universal cone associated with the projective limit above: for $V \in \ee_\alpha$,
\begin{equation}\label{univ-cones}
\iota^\alpha_V=\{(\iota^\alpha_V)_{(A,\sigma)} \co
\varphi_\alpha(A,\sigma) \to \varphi_V(A,\sigma)\}_{(A,\sigma)\in
\zz_G(\cc)}
\end{equation} is a monoidal natural isomorphism from
$\varphi_\alpha$ to $\varphi_V$.

By Lemma \ref{lem-phiV-to-projlim}(c),(d), the transformations $\zeta$ and $\eta$
induce monoidal natural isomorphisms $\varphi_2(\alpha,\beta)\co
\varphi_\alpha\varphi_\beta \to \varphi_{\beta\alpha}$   and
$\varphi_0\co 1_{\zz_G(\cc)} \to \varphi_1$, respectively. These
isomorphisms are related to the universal cone as
follows: for $U \in \ee_\alpha$, $V \in \ee_\beta$, $W \in
\ee_{\beta\alpha}$, and $R \in \ee_1$, the following diagrams commute:
$$
    \xymatrix@R=1cm @C=2.5cm {
\varphi_\alpha\varphi_\beta \ar[r]^-{(\varphi_2)(\alpha,\beta)} \ar[d]_{\varphi_U(\iota^\beta_V)(\iota^\alpha_U)_{\varphi_\beta}} & \varphi_{\beta\alpha}
\ar[d]^{\iota^{\beta\alpha}_W} \\
\varphi_U \varphi_V \ar[r]_-{\zeta^{U,V,W}} & \varphi_W}
\qquad \qquad
\xymatrix@R=1cm @C=.5cm {
& 1_{\zz_G(\cc)} \ar[ld]_-{\varphi_0} \ar[dr]^{\eta^R}   \\
\varphi_1 \ar[rr]_-{\iota^1_R} && \varphi_R.}
$$
By   Lemmas~\ref{lem-quotient-varphi}(b) and
~\ref{lem-quotient-varphi+}(b),  $\varphi_2$ and $\varphi_0$ satisfy
\eqref{crossing5} and \eqref{crossing6}.   Note that $\varphi_2$ and
$\varphi_0$ induce natural isomorphisms  $\varphi_\alpha
\varphi_{\alpha^{-1}} \simeq \varphi_1 \simeq 1_{\zz_G(\cc)}$ and
$\varphi_{\alpha^{-1}} \varphi_\alpha \simeq \varphi_1 \simeq
1_{\zz_G(\cc)}$ for $\alpha \in G$.  Hence, the endofunctor
$\varphi_\alpha$ of $\zz_G(\cc)$ is an equivalence. Therefore
$$\varphi=(\varphi,\varphi_2,\varphi_0) \co \overline{G} \to
\Aut\bigl(\zz_G(\cc)\bigr), \,\, \alpha \mapsto
\varphi_\alpha$$ is a strong monoidal functor such that
$\varphi_\alpha\bigl(\zz_\beta(\cc)\bigr) \subset
\zz_{\alpha^{-1}\beta\alpha}(\cc)$ for all $\alpha,\beta \in G$.
Thus,  $\varphi$ is a crossing in $\zz_G(\cc)$. It is pivotal because all  $\varphi_\alpha$'s are pivotal.


\subsection{The $G$-braiding}\label{sect-braiding in-G}
We construct a $G$-braiding in  $\zz_G(\cc)$ following the   scheme of Section \ref{sect-action of-G}.
 For $(A,\sigma) \in \zz_G(\cc)$,  $V
\in \ee_\alpha$ with $\alpha \in G$, and $X \in \cch$, set
$$
\Gamma^{V}_{(A,\sigma),X}=  \;\psfrag{A}[Bl][Bl]{\scalebox{.8}{$A$}}
\psfrag{B}[Bl][Bl]{\scalebox{.8}{$X$}}
 \psfrag{V}[Bl][Bl]{\scalebox{.8}{$V$}}
   \psfrag{E}[Bl][Bl]{\scalebox{.8}{$E_{(A,\sigma)}^V$}}
  \psfrag{s}[Bc][Bc]{\scalebox{.9}{$\sigma_{X \otimes V^*}$}}
 \rsdraw{.45}{1}{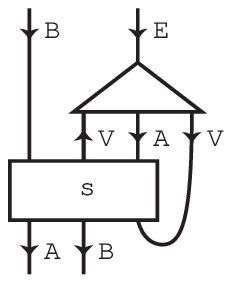}\; \colon A \otimes X  \to X \otimes E_{(A,\sigma)}^V.
$$
The next lemma shows that these morphisms are isomorphisms compatible with   the transformations introduced in Section \ref{sect-action of-G}.

\begin{lem}\label{lem-quotient-Gamma}
\begin{enumerate}
\labela
\item   $\Gamma^{V}_{(A,\sigma),X}$ is an isomorphism natural in $(A,\sigma) $ and in $X  $, and
$$
\bigl( \Gamma^{V}_{(A,\sigma),X} \bigr)^{-1}=d_V^{-1}
\;\psfrag{A}[Bl][Bl]{\scalebox{.8}{$A$}}
\psfrag{B}[Bl][Bl]{\scalebox{.8}{$X$}}
 \psfrag{V}[Bl][Bl]{\scalebox{.8}{$V$}}
   \psfrag{E}[Bl][Bl]{\scalebox{.8}{$E_{(A,\sigma)}^V$}}
  \psfrag{s}[Bc][Bc]{\scalebox{.9}{$\sigma_{X \otimes V^*}^{-1}$}}
 \rsdraw{.45}{1}{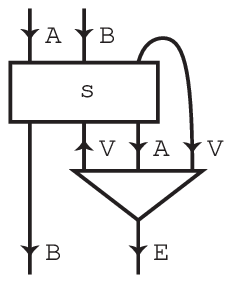} \;.
$$
\item For any $U,V \in \ee_\alpha$, the following diagram commutes:
$$    \xymatrix@R=1cm @C=.5cm {
& A \otimes X \ar[ld]_-{\Gamma^{V}_{(A,\sigma),X}} \ar[dr]^{\Gamma^{U}_{(A,\sigma),X}}   \\
X \otimes E_{(A,\sigma)}^V \ar[rr]_-{\id_X \otimes \delta^{U,V}} && X \otimes E_{(A,\sigma)}^U.}
$$
\item For all $(A,\sigma) \in \zz_G(\cc)$, $U \in \ee_\alpha$, $V \in \ee_\beta$, $W \in \ee_{\alpha\beta}$, $X \in \cc_\alpha$, and $Y \in \cc_\beta$,
the following diagram commutes:
$$
\xymatrix@R=1cm @C=3cm {
A \otimes X \otimes Y \ar[r]^-{\Gamma^W_{(A,\sigma), X \otimes Y}}\ar[d]_{\Gamma^U_{(A,\sigma), X} \otimes \id_Y} & X \otimes Y \otimes E^W_{(A,\sigma)} \\
X \otimes E^U_{(A,\sigma)} \otimes Y \ar[r]_-{\id_X \otimes \Gamma^V_{\varphi_U(A,\sigma), Y}} & X \otimes Y \otimes E^V_{\varphi_U(A,\sigma)}.  \ar[u]_{\id_{X \otimes Y} \otimes \zeta^{V,U,W}_{(A,\sigma)}}
}
$$
\item For all $(A,\sigma),(B,\rho) \in \zz_G(\cc)$, $V \in \ee_\alpha$, and $X \in \cc_\alpha$,
the following diagram commutes:
$$
\xymatrix@R=1cm @C=3cm {
A \otimes B \otimes X \ar[r]^-{\Gamma^V_{(A,\sigma)\otimes (B,\rho), X}}\ar[d]_{\id_A\otimes \Gamma^V_{(B,\rho), X}} & X \otimes E^V_{(A,\sigma)\otimes (B,\rho)} \\
A \otimes X \otimes E^V_{(B,\rho)} \ar[r]_-{\Gamma^V_{(A,\sigma), X} \otimes \id_{E^V_{(B,\rho)}}} & X \otimes  E^V_{(A,\sigma)} \otimes E^V_{(B,\rho)}. \ar[u]_{\id_X \otimes (\varphi_V)_2((A,\sigma), (B,\rho))}
}
$$
\item  $\Gamma^V_{(A,\sigma), \un}=\eta^V_{(A,\sigma)}$ for
any $V \in \ee_1$.
\item $\Gamma^V_{(\un,\id), X}=\id_X \otimes (\varphi_V)_0$ for all $V \in \ee_\alpha$ and $X \in \cc_\alpha$.
\item For all $\alpha,\beta \in G$, $(A,\sigma) \in \zz_G(\cc)$, $(B,\rho) \in \zz_\beta(\cc)$, $U \in \ee_\alpha$, $V \in \ee_\beta$, $W \in \ee_{\beta \alpha}$, and $S \in \ee_{\alpha^{-1}\beta\alpha}$,
the following diagram commutes:
$$
\xymatrix@R=1cm @C=3cm{
E^U_{(A,\sigma)} \otimes E^U_{(B,\rho)} \ar[r]^-{(\varphi_U)_2((A,\sigma),(B,\rho))}\ar[d]^{\Gamma^S_{\varphi_U(A,\sigma), E^U_{(B,\rho)}}} &E^U_{(A,\sigma) \otimes (B,\rho)} \ar[d]^-{\varphi_U(\Gamma^V_{(A,\sigma), B})} \\
E^U_{(B,\rho)} \otimes E^S_{\varphi_U(A,\sigma)} \ar[d]^-{\id_{E^U_{(B,\rho)}} \otimes \, \zeta^{S,U,W}_{(A,\sigma)}} & E^U_{(B,\rho)  \otimes \varphi_V(A,\sigma)} \\
E^U_{(B,\rho)} \otimes E^W_{(A,\sigma)} \ar[r]_-{\id_{E^U_{(B,\rho)}} \otimes \, (\zeta^{U,V,W}_{(A,\sigma)})^{-1}} &  E^U_{(B,\rho)} \otimes E^U_{\varphi_V(A,\sigma)}. \ar[u]_-{(\varphi_U)_2((B,\rho),\varphi_V(A,\sigma))}
}
$$
\item For all $(A,\sigma) \in \zz_G(\cc)$, $\alpha \in G$, $(B,\rho) \in \zz_\alpha(\cc)$, $V \in \ee_\alpha$, and $X \in \cc_1$,
the following diagram commutes:
$$
\xymatrix@R=1cm @C=3cm{
A \otimes B \otimes X \ar[r]^-{\Gamma^V_{(A,\sigma),B}\otimes \id_X}\ar[d]^{\id_A \otimes \rho_X} &B \otimes E^V_{(A,\sigma)} \otimes X \ar[d]^-{\id_B \otimes \gamma^V_{(A,\sigma), X}} \\
A \otimes X \otimes B \ar[d]^-{\sigma_X \otimes \id_B} & B \otimes X \otimes E^V_{(A,\sigma)} \ar[d]^-{\rho_X \otimes \id_{E^V_{(A,\sigma)}}} \\
X \otimes A \otimes B \ar[r]_-{\id_X \otimes \Gamma^V_{(A,\sigma),B}}
& X \otimes B \otimes E^V_{(A,\sigma)}.
}
$$
\end{enumerate}
\end{lem}

By Lemma~\ref{lem-quotient-Gamma}(a),(b), the transformation
$\Gamma$ induces a family of  isomorphisms
\begin{equation}\label{eq-tau-center}
 \{\tau_{(A,\sigma),X} \co A \otimes X \to X \otimes \uu\bigl(\varphi_{|X|}(A,\sigma) \bigr)\}_{(A,\sigma)
  \in \zz_G(\cc), X \in \cch},
\end{equation}
where $\uu \co \zz_G(\cc) \to \cc$ is the forgetful functor.  This family is natural in $(A,\sigma)  $ and in $X  $
 and is related to the universal cones $\{\iota^\alpha\}_{\alpha\in G}$ as follows: for any $(A,\sigma) \in \zz_G(\cc)$, $X \in \cch$, and $V \in \ee_{|X|}$,
$$
\bigl(\id_X \otimes (\iota^{|X|}_V)_{(A,\sigma)} \bigr) \tau_{(A,\sigma),X}=\Gamma^V_{(A,\sigma),X}.
$$
We call the family \eqref{eq-tau-center} the \emph{enhanced $G$-braiding} in $\zz_G(\cc)$.

\begin{lem}\label{lem-weak-tau}
For all $(A,\sigma)  \in \zz_G(\cc)$ and $X,Y \in \cch$,
\begin{enumerate}
\labela
\item $\tau_{(A,\sigma), X \otimes Y} =(\id_{X \otimes Y} \otimes \varphi_2(|Y|,|X|)_{(A,\sigma)})(\id_X \otimes \tau_{\varphi_{|X|}(A,\sigma),Y})(\tau_{(A,\sigma),X} \otimes \id_Y)$;
\item  Given $(B,\rho) \in \zz_G(\cc)$,
\begin{align*} \tau_{(A,\sigma) \otimes (B,\rho), X} = &
\bigl(\id_X \otimes (\varphi_{|X|})_2((A,\sigma),(B,\rho))
\bigr)\bigl(\tau_{(A,\sigma),X} \otimes
\id_{\varphi_{|X|}(B,\rho)}\bigr) \circ\\
     &\quad \circ \bigl(\id_A \otimes \tau_{(B,\rho),X}\bigr);
     \end{align*}
\item $\tau_{(A,\sigma),\un} =(\varphi_0)_{(A,\sigma)}$;
\item $\tau_{(\un,\id),X}  =\id_X \otimes (\varphi_{|X|})_0$;
\item   The inverse of  $\tau_{(A,\sigma),X}$ is computed by
\begin{align*}
     \tau_{(A,\sigma),X}^{-1}=&
     (\rev_X \otimes (\varphi_0)_{(A,\sigma)}^{-1}\varphi_2(|X|^{-1},|X|)_{(A,\sigma)} \otimes \id_X )\circ\\
     &\quad \circ
      (\id_X \otimes \tau_{\varphi_{|X|}(A,\sigma),X^*} \otimes \id_X)(\id_{X \otimes A} \otimes \rcoev_X)\\
=& (\id_{A \otimes X} \otimes \lev_{\varphi_{|X|}(A,\sigma)}(\varphi_{|X|}^l(A,\sigma) \otimes \id_{\varphi_{|X|}(A,\sigma)}) ) \circ\\
&\quad \circ
   (\id_A \otimes \tau_{(A,\sigma)^*,X} \otimes \id_{\varphi_{|X|}(A,\sigma)})(\lcoev_A \otimes \id_{X \otimes \varphi_{|X|}(A,\sigma)}).
\end{align*}
\end{enumerate}
\end{lem}
\begin{proof}
Claims (a)-(d) follows respectively from Lemma~\ref{lem-quotient-Gamma}(c)-(f). Claim (e) follows from the expression of the inverse of  $\Gamma^V_{(A,\sigma),X}$ given in \ref{lem-quotient-Gamma}(a) and from the computation of $\sigma^{-1}_Y$ in terms of $\sigma_{Y^*}$ and $\sigma^\dagger_Y$ provided by \eqref{eq-inv-sigma} for any $Y \in \cc_1$.
\end{proof}

\begin{lem}\label{lem-braiding-ZG}
The family $
\tau=\{\tau_{(A,\sigma),(B,\rho)} \}_{(A,\sigma) \in \zz_G(\cc), (B,\rho) \in \zz_G(\cc)_{\mathrm{hom}}}
$
defined by
$$
\tau_{(A,\sigma),(B,\rho)}=\tau_{(A,\sigma),B} \co (A,\sigma) \otimes (B,\rho) \to
  (B,\rho) \otimes \varphi_{|(B,\rho)|}(A,\sigma).
$$  is a $G$-braiding in  $\zz_G(\cc)$.
\end{lem}
\begin{proof}
  Lemma~\ref{lem-quotient-Gamma}(h) implies that $\tau_{(A,\sigma),(B,\rho)}$ is   a morphism in $\zz_G(\cc)$. It is an isomorphism since $\uu(\tau_{(A,\sigma),(B,\rho)})=\tau_{(A,\sigma),B}$ is an isomorphism in $\cc$, and the forgetful functor $\uu$ is conservative.
The naturality of the enhanced braiding implies the naturality of $\tau$.   The formulas
\eqref{eq-braiding1}, \eqref{eq-braiding2}, \eqref{eq-braiding3}  follow respectively from Lemmas~\ref{lem-weak-tau}(a),  \ref{lem-weak-tau}(b), and  \ref{lem-quotient-Gamma}(e).
\end{proof}

Theorem~\ref{thm-G-center} is a direct consequence of Lemma~\ref{lem-braiding-ZG}.

\subsection{Proof of Lemmas~\ref*{lem-phiV-def} -- \ref*{lem-quotient-Gamma}}\label{sect-proof-lems-defdouble}
Let $(A,\sigma)\in\zz_G(\cc)$. For $X \in \cc_1$, we depict
 the morphism $\sigma_X \co A \otimes X \to X \otimes A$ and its inverse $\sigma_X^{-1} \co X \otimes A \to A \otimes X$ by
$$
 \psfrag{A}[Br][Br]{\scalebox{.8}{\textcolor{red}{$A$}}}
 \psfrag{X}[Bl][Bl]{\scalebox{.8}{$X$}}
\sigma_X=\rsdraw{.45}{.9}{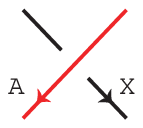} \qquad \text{and} \qquad  \psfrag{A}[Bl][Bl]{\scalebox{.8}{\textcolor{red}{$A$}}}
 \psfrag{X}[Br][Br]{\scalebox{.8}{$X$}} \sigma_X^{-1}=\rsdraw{.45}{.9}{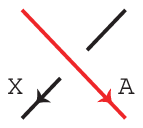}\;\;.
$$
Formula \eqref{eq-inv-sigma} implies that
$$
 \psfrag{A}[Bl][Bl]{\scalebox{.8}{\textcolor{red}{$A$}}}
 \psfrag{X}[Bl][Bl]{\scalebox{.8}{$X$}} \rsdraw{.45}{.9}{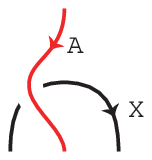} \; = \;  \psfrag{A}[Br][Br]{\scalebox{.8}{\textcolor{red}{$A$}}}
 \psfrag{X}[Br][Br]{\scalebox{.8}{$X$}}
 \rsdraw{.45}{.9}{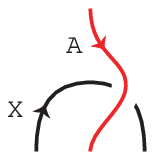}
\qquad \text{and} \qquad
 \psfrag{A}[Bl][Bl]{\scalebox{.8}{\textcolor{red}{$A$}}}
 \psfrag{X}[Bl][Bl]{\scalebox{.8}{$X$}} \rsdraw{.45}{.9}{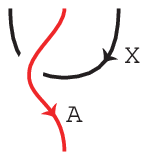} \; = \;  \psfrag{A}[Br][Br]{\scalebox{.8}{\textcolor{red}{$A$}}}
 \psfrag{X}[Br][Br]{\scalebox{.8}{$X$}}
 \rsdraw{.45}{.9}{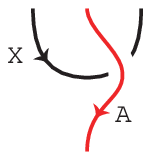} \;.$$
These two morphisms are pictorially represented  respectively as
$$
 \psfrag{A}[Bl][Bl]{\scalebox{.8}{\textcolor{red}{$A$}}}
 \psfrag{X}[Bl][Bl]{\scalebox{.8}{$X$}} \rsdraw{.45}{.9}{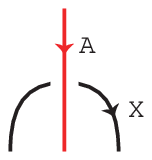} \qquad \text{and} \qquad
 \psfrag{A}[Bl][Bl]{\scalebox{.8}{\textcolor{red}{$A$}}}
 \psfrag{X}[Bl][Bl]{\scalebox{.8}{$X$}} \rsdraw{.45}{.9}{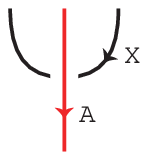}
 \;\;.
$$
Axiom~\eqref{axiom-half-braiding} implies that for any $X_1,\dots,X_n \in \cc_1$,
$$
 \psfrag{A}[Br][Br]{\scalebox{.8}{\textcolor{red}{$A$}}}
 \psfrag{X}[Bl][Bl]{\scalebox{.8}{$X_1\otimes \cdots \otimes X_n$}}
\rsdraw{.45}{.9}{sX}\qquad
  \psfrag{X}[Bl][Bl]{\scalebox{.8}{$X_1$}}
  \psfrag{Y}[Bl][Bl]{\scalebox{.8}{$X_n$}}
\qquad =\rsdraw{.45}{.9}{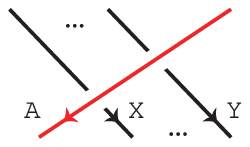}\;\;.
$$
 In generalization of this notation, if $X \in \cc_1$ decomposes as $X=X_1 \otimes \cdots \otimes X_n$ where $X_1,\dots,X_n$ are  any homogeneous  objects of $\cc$, then we will depict $\sigma_X$ as
$$
 \psfrag{A}[Br][Br]{\scalebox{.8}{\textcolor{red}{$A$}}}
 \psfrag{X}[Bl][Bl]{\scalebox{.8}{$X_1$}}
  \psfrag{Y}[Bl][Bl]{\scalebox{.8}{$X_n$}}
\sigma_X=\rsdraw{.45}{.9}{sXn}\;\;.
$$
As usual, if an arc colored by $X_i$ is oriented upwards,
then the corresponding object   in the source/target of  morphisms
is $X_i^*$. For example, if $X \in \cc_1$ decomposes as $X=X_1^* \otimes X_2 \otimes X_3^*$ where $X_1,X_2,X_3 \in \cc$, then we depict $\sigma_X$ as
$$
 \psfrag{A}[Br][Br]{\scalebox{.8}{\textcolor{red}{$A$}}}
 \psfrag{X}[Bl][Bl]{\scalebox{.8}{$X_1$}}
 \psfrag{Y}[Bl][Bl]{\scalebox{.8}{$X_3$}}
 \psfrag{Z}[Bl][Bl]{\scalebox{.8}{$X_3$}}
\sigma_X=\rsdraw{.45}{.9}{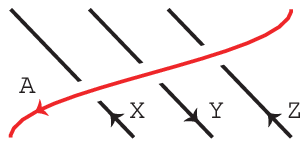}\;\;.
$$
In this pictorial formalism, for $V \in \cc$,
\begin{equation}\label{eq-dim-sigma}
 \psfrag{A}[Br][Br]{\scalebox{.8}{\textcolor{red}{$A$}}}
 \psfrag{V}[Bl][Bl]{\scalebox{.8}{$V$}}
\rsdraw{.45}{.9}{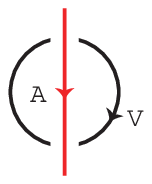}\; \;= \; d_V \;\,  \psfrag{A}[Bl][Bl]{\scalebox{.8}{\textcolor{red}{$A$}}} \rsdraw{.45}{.9}{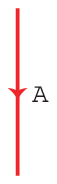} \;\;.
\end{equation}
Indeed, by   the naturality of $\sigma$,
$$
 \psfrag{A}[Br][Br]{\scalebox{.8}{\textcolor{red}{$A$}}}
 \psfrag{V}[Bl][Bl]{\scalebox{.8}{$V$}}
  \psfrag{B}[Bl][Bl]{\scalebox{.8}{$V^*\otimes V$}}
 \psfrag{i}[Bc][Bc]{\scalebox{.8}{$\id_{V^* \otimes V}$}}
\rsdraw{.45}{.9}{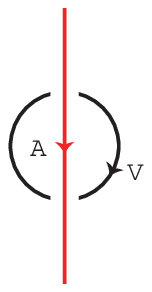}\; = \; \rsdraw{.45}{.9}{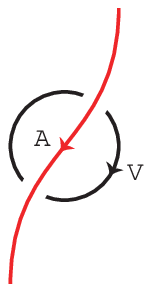}= \; \psfrag{A}[Bl][Bl]{\scalebox{.8}{\textcolor{red}{$A$}}} \rsdraw{.45}{.9}{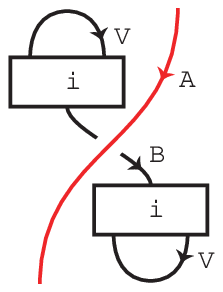}\;\,= \; \rsdraw{.45}{.9}{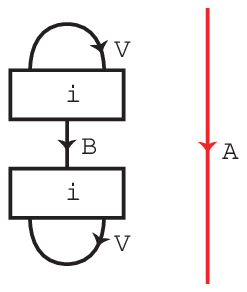}
\;= \; d_V \;\,   \rsdraw{.45}{.9}{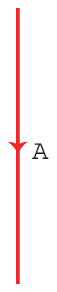} .
$$
%

\begin{lem}\label{lem-pict-demo}
Let   $(A,\sigma)\in\zz_G(\cc)$ and  $V\in\ee_\alpha$ with $\alpha \in G$. Then
\begin{center}
\quad \begin{minipage}{0.45\linewidth}
\begin{equation}\label{eq-Drel1}
 \psfrag{A}[Bl][Bl]{\scalebox{.8}{\textcolor{red}{$A$}}}
 \psfrag{V}[Bl][Bl]{\scalebox{.8}{$V$}}
 \psfrag{E}[Bl][Bl]{\scalebox{.8}{$E_{(A,\sigma)}^V$}}
\rsdraw{.45}{.9}{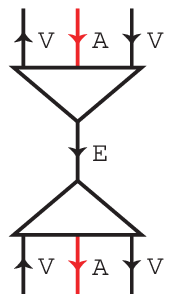}\; = \; d_V^{-1} \rsdraw{.45}{.9}{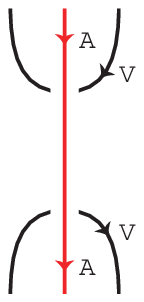} \;,
\end{equation}
\end{minipage}
\quad
\begin{minipage}{0.45\linewidth}
\begin{equation}\label{eq-Drel2}
 \psfrag{A}[Bl][Bl]{\scalebox{.8}{\textcolor{red}{$A$}}}
 \psfrag{V}[Bl][Bl]{\scalebox{.8}{$V$}}
 \psfrag{E}[Bl][Bl]{\scalebox{.8}{$E_{(A,\sigma)}^V$}}
\rsdraw{.45}{.9}{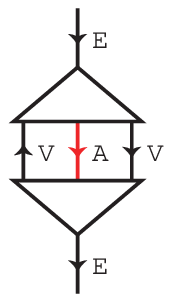}\; = \;  \rsdraw{.45}{.9}{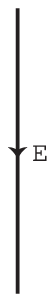} \qquad ,
\end{equation}
\end{minipage}\\[1em]
\quad \begin{minipage}{0.45\linewidth}
\begin{equation}\label{eq-Drel3}
 \psfrag{A}[Bl][Bl]{\scalebox{.8}{\textcolor{red}{$A$}}}
 \psfrag{V}[Bl][Bl]{\scalebox{.8}{$V$}}
 \psfrag{E}[Bl][Bl]{\scalebox{.8}{$E_{(A,\sigma)}^V$}}
\rsdraw{.45}{.9}{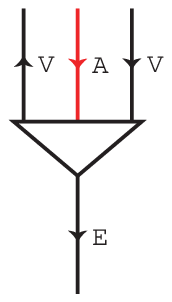}= \; d_V^{-1} \;
\psfrag{A}[Br][Br]{\scalebox{.8}{\textcolor{red}{$A$}}} \rsdraw{.45}{.9}{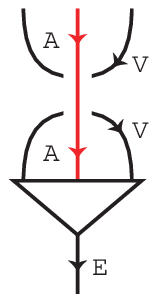} \;,
\end{equation}
\end{minipage}
\quad
\begin{minipage}{0.45\linewidth}
\begin{equation}\label{eq-Drel4}
 \psfrag{A}[Bl][Bl]{\scalebox{.8}{\textcolor{red}{$A$}}}
 \psfrag{V}[Bl][Bl]{\scalebox{.8}{$V$}}
 \psfrag{E}[Bl][Bl]{\scalebox{.8}{$E_{(A,\sigma)}^V$}}
\rsdraw{.45}{.9}{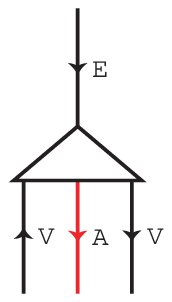} = \; d_V^{-1} \;
\psfrag{A}[Br][Br]{\scalebox{.8}{\textcolor{red}{$A$}}} \rsdraw{.45}{.9}{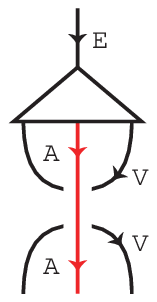} \;\,,
\end{equation}
\end{minipage}\\[1em]
\quad \begin{minipage}{0.45\linewidth}
\begin{equation}\label{eq-Drel5}
 \psfrag{A}[Br][Br]{\scalebox{.8}{\textcolor{red}{$A$}}}
 \psfrag{V}[Bl][Bl]{\scalebox{.8}{$V$}}
 \psfrag{E}[Bl][Bl]{\scalebox{.8}{$E_{(A,\sigma)}^V$}}
\rsdraw{.45}{.9}{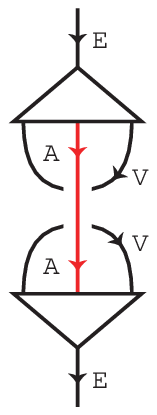}\; = \; d_V\;\,  \rsdraw{.45}{.9}{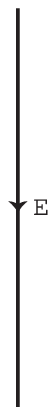} \qquad ,
\end{equation}
\end{minipage}
\quad
\begin{minipage}{0.45\linewidth}
\begin{equation}\label{eq-Drel6}
 \psfrag{A}[Bl][Bl]{\scalebox{.8}{\textcolor{red}{$A$}}}
 \psfrag{V}[Bl][Bl]{\scalebox{.8}{$V$}}
 \psfrag{E}[Bl][Bl]{\scalebox{.8}{$E_{(A,\sigma)}^V$}}
\rsdraw{.45}{.9}{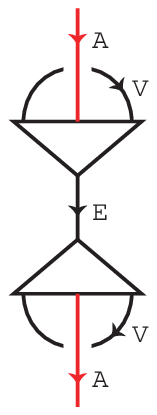}\; \;= \; d_V \;\, \rsdraw{.45}{.9}{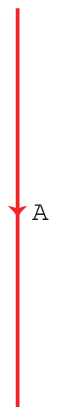} \;\;.
\end{equation}
\end{minipage}
\end{center}
\end{lem}
Note that the left hand side of \eqref{eq-Drel6} does not
necessarily depict a morphism
  in itself  because $V$ may not  belong  to $\cc_1$. The equality  \eqref{eq-Drel6}  means that in
any diagram, a piece as in the left-hand side of
 \eqref{eq-Drel6}  may be replaced with the piece as in right-hand side of \eqref{eq-Drel6}  and vice
 versa.

\begin{proof}
Equalities \eqref{eq-Drel1} and \eqref{eq-Drel2} follow directly from  \eqref{eq-idemp-split} and the definition of $\pi^V_{(A,\sigma)}$. Composing  on the right \eqref{eq-Drel1} with $q^V_{(A,\sigma)}$ and then using \eqref{eq-Drel2} gives \eqref{eq-Drel3}. Similarly~\eqref{eq-Drel4} is obtained by composing  on the left \eqref{eq-Drel1} with $p^V_{(A,\sigma)}$. Composing \eqref{eq-Drel4} with~\eqref{eq-Drel3} and then using \eqref{eq-Drel2} and \eqref{eq-dim-sigma} gives \eqref{eq-Drel5}.
Finally \eqref{eq-Drel6} is a direct consequence of \eqref{eq-Drel1} and~\eqref{eq-dim-sigma}.
\end{proof}

We compute now the functor $\varphi_V$ in this pictorial formalism for    $V\in\ee_\alpha$ with  $\alpha \in G$.   For $(A,\sigma) \in\zz_G(\cc)$, we have  $\varphi_V(A,\sigma)=(E^V_{(A,\sigma)},\gamma^V_{(A,\sigma)})$ where, for  $X \in \cc_1$,
$$
 \psfrag{A}[Br][Br]{\scalebox{.8}{\textcolor{red}{$A$}}}
 \psfrag{B}[Bl][Bl]{\scalebox{.8}{\textcolor{red}{$A$}}}
 \psfrag{U}[Bl][Bl]{\scalebox{.8}{$V$}}
 \psfrag{V}[Br][Br]{\scalebox{.8}{$V$}}
 \psfrag{X}[Bl][Bl]{\scalebox{.8}{$X$}}
 \psfrag{E}[Bl][Bl]{\scalebox{.8}{$E_{(A,\sigma)}^V$}}
 \psfrag{F}[Bl][Bl]{\scalebox{.8}{$E_{(A,\sigma)}^V$}}
 \gamma^V_{(A,\sigma),X}= d_V^{-1}\;\rsdraw{.45}{.9}{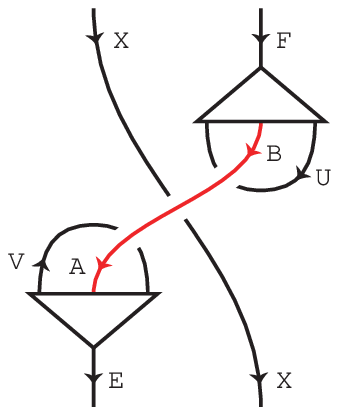}\;.
$$
Using \eqref{eq-Drel3} twice, we obtain that
$$
 \psfrag{V}[Bc][Bc]{\scalebox{.8}{$V$}}
 \psfrag{A}[Bc][Bc]{\scalebox{.8}{\textcolor{red}{$A$}}}
 \psfrag{B}[Bc][Bc]{\scalebox{.8}{\textcolor{red}{$B$}}}
 \psfrag{E}[Bl][Bl]{\scalebox{.8}{$E^V_{(A,\sigma)}$}}
 \psfrag{F}[Bl][Bl]{\scalebox{.8}{$E^V_{(B,\rho)}$}}
 \psfrag{G}[Bl][Bl]{\scalebox{.8}{$E_{(A,\sigma)\otimes (B,\rho)}^V$}}
 (\varphi_V)_2\bigl((A,\sigma),(B,\rho)\bigr)=  \rsdraw{.45}{.9}{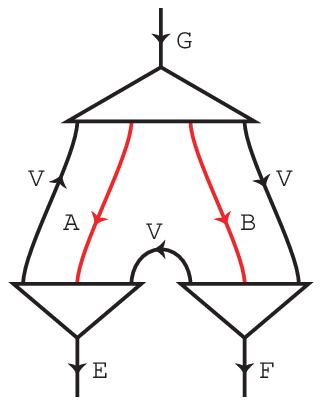}
=  \,d_V^{-2}\rsdraw{.45}{.9}{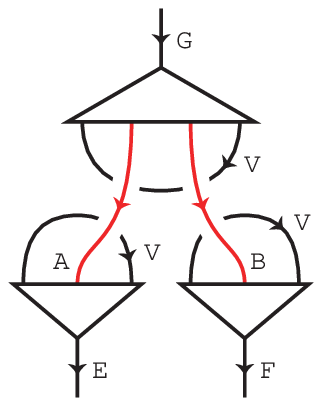} \;\;.
$$

Let us prove Lemma~\ref{lem-phiV-def}. Firstly, $\gamma=\gamma^V_{(A,\sigma)}$ is a half-braiding. Indeed, the naturality of $\sigma$ implies that of $\gamma$. Also,   using~\eqref{eq-Drel5} we obtain $\gamma_\un=\id_A$ and
$$
 \psfrag{X}[Bl][Bl]{\scalebox{.8}{$X$}}
 \psfrag{Y}[Bl][Bl]{\scalebox{.8}{$Y$}}
(\gamma_X \otimes \id_Y ) (\id_X \otimes \gamma_Y)= d_V^{-2}\,\rsdraw{.45}{.9}{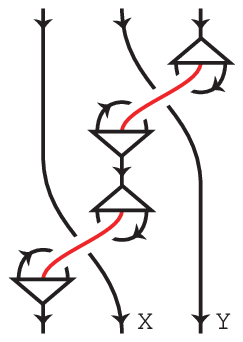}= d_V^{-1}\rsdraw{.45}{.9}{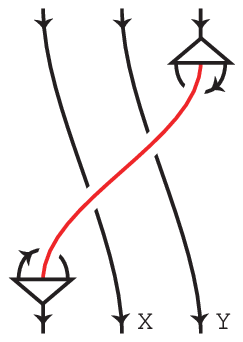}=\gamma_{X \otimes Y}.
$$
The category $\cc$ being rigid, these two equalities imply that $\gamma$ is invertible. Hence $\gamma$ is a half-braiding. Secondly, $\varphi_V$ is a functor since $\varphi_V(\id_{(A,\sigma)})=\id_{\varphi_V(A,\sigma)}$ by \eqref{eq-Drel2} and,  for two composable morphisms $g$, $f$ in $\zz_G(\cc)$, \eqref{eq-Drel1} and \eqref{eq-Drel3} give
$$
 \psfrag{f}[Bc][Bl]{\scalebox{.8}{\textcolor{red}{$f$}}}
 \psfrag{g}[Bc][Bc]{\scalebox{.8}{\textcolor{red}{$g$}}}
 \psfrag{h}[Bc][Bc]{\scalebox{.8}{\textcolor{red}{$gf$}}}
\varphi_V(g)\varphi_V(f)=\; \rsdraw{.45}{.9}{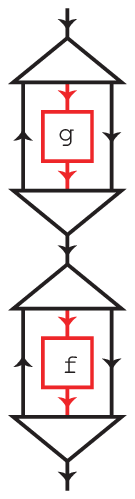}=d_V^{-1}\; \rsdraw{.45}{.9}{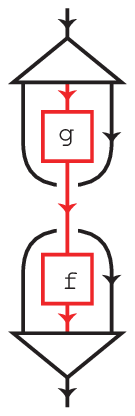}=d_V^{-1}\; \rsdraw{.45}{.9}{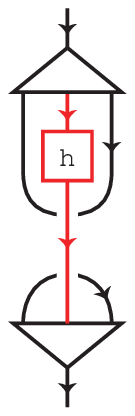}=\;\rsdraw{.45}{.9}{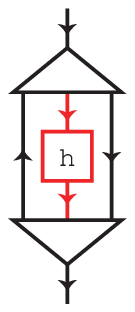}=\varphi_V(gf).
$$
Let us prove that $\varphi_V$ is strong monoidal. Let $(A,\sigma), (B,\rho), (C,\varrho) \in \zz_G(\cc)$.
Applying \eqref{eq-Drel6}, we obtain
$$
\psfrag{X}[Bl][Bl]{\scalebox{.8}{$X$}}
d_V^{-3}\rsdraw{.45}{.9}{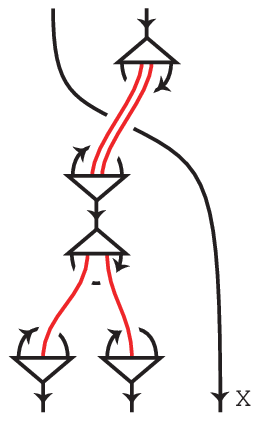}=d_V^{-2}\;\rsdraw{.45}{.9}{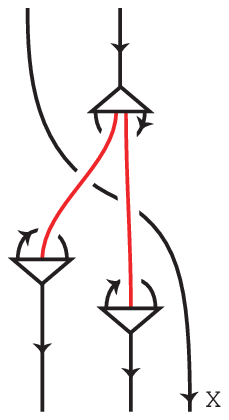}=d_V^{-4}\;\rsdraw{.45}{.9}{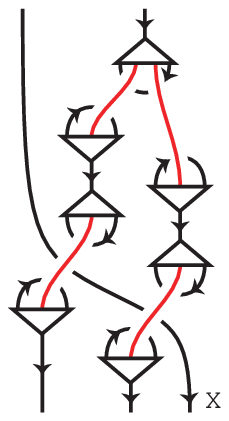},
$$
that is,
\begin{gather*}
\gamma^V_{(A,\sigma) \otimes (B,\rho),X}\left((\varphi_V)_2\bigl((A,\sigma),(B,\rho)\bigr) \otimes \id_X \right)
\\
= \bigl( \id_X \otimes (\varphi_V)_2\bigl((A,\sigma),(B,\rho)\bigr)\bigr)
\bigl( \gamma^V_{(A,\sigma),X} \otimes \id_{\varphi_V(B,\rho)}\bigr)
\bigl(\id_{\varphi_V(A,\sigma)} \otimes \gamma^V_{(B,\rho),X} \bigr).
\end{gather*}
Thus $(\varphi_V)_2\bigl((A,\sigma),(B,\rho)\bigr)$ is a morphism in $\zz_G(\cc)$. Similarly, $(\varphi_V)_0$ is a morphism in $\zz_G(\cc)$ because, by  \eqref{eq-Drel1},
$$
\psfrag{X}[Bl][Bl]{\scalebox{.8}{$X$}}
\gamma^V_{(\un,\id),X} \bigl((\varphi_V)_0 \otimes \id_X \bigr)=
d_V^{-1}\;\, \rsdraw{.4}{.9}{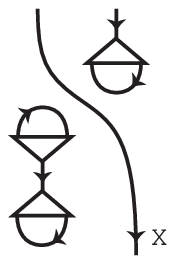}
=\id_X \otimes (\varphi_V)_0.
$$
Now $(\varphi_V)_2$ satisfies \eqref{stmonoidal1} since, by using \eqref{eq-Drel1}, we obtain that both
$$
(\varphi_V)_2\bigl((A,\sigma)\otimes (B,\rho),(C,\varrho)\bigr) \bigl( (\varphi_V)_2\bigl((A,\sigma),(B,\rho)\bigr) \otimes \id_{\varphi_V(C,\varrho)}\bigr)
$$
and
$$
(\varphi_V)_2\bigl((A,\sigma),(B,\rho)\otimes (C,\varrho)\bigr) \bigl(\id_{\varphi_V(A,\sigma)} \otimes  (\varphi_V)_2\bigl((B,\rho),(C,\varrho)\bigr) \bigr)
$$
are equal to
$$
d_V^{-3}\;\rsdraw{.45}{.9}{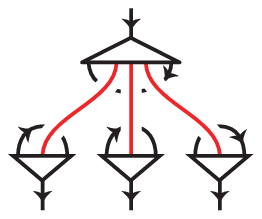}.
$$
Axiom \eqref{stmonoidal2} is a direct consequence of \eqref{eq-Drel1} applied to $(\un,\id)$. Hence $\varphi_V$ is a monoidal functor. It remains to prove that both $(\varphi_V)_2\bigl((A,\sigma),(B,\rho)\bigr)$ and $(\varphi_V)_0$ are isomorphisms in $\zz_G(\cc)$. Since the forgetful functor $\zz_G(\cc) \to \cc$ is conservative, we only need to verify that these morphisms  are isomorphisms in $\cc$. This can be done by verifying (using Lemma~\ref{lem-pict-demo}) that
$$
 \psfrag{V}[Bc][Bc]{\scalebox{.8}{$V$}}
 \psfrag{A}[Bc][Bc]{\scalebox{.8}{\textcolor{red}{$A$}}}
 \psfrag{B}[Bc][Bc]{\scalebox{.8}{\textcolor{red}{$B$}}}
 \psfrag{E}[Bl][Bl]{\scalebox{.8}{$E^V_{(A,\sigma)}$}}
 \psfrag{F}[Bl][Bl]{\scalebox{.8}{$E^V_{(B,\rho)}$}}
 \psfrag{G}[Bl][Bl]{\scalebox{.8}{$E_{(A,\sigma)\otimes (B,\rho)}^V$}}
(\varphi_V)_2\bigl((A,\sigma),(B,\rho)\bigr)^{-1}=d_V^{-1}\!\!  \rsdraw{.37}{.9}{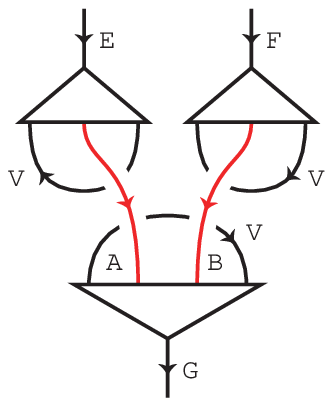}
\quad \text{and} \quad
 \psfrag{V}[Bc][Bc]{\scalebox{.8}{$V$}}
 \psfrag{F}[Bl][Bl]{\scalebox{.8}{$E^V_{(\un,\id)}$}}
(\varphi_V)_0^{-1}=d_V^{-1}\; \rsdraw{.37}{.9}{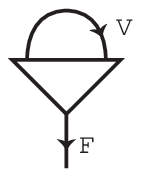} \;.
$$
  Finally, let us prove that $\varphi_V$ is pivotal. For $(A,\sigma)\in\zz_G(\cc)$, we obtain that
$$
 \psfrag{F}[Bl][Bl]{\scalebox{.8}{$E^V_{(A,\sigma)^*}$}}
 \psfrag{E}[Bl][Bl]{\scalebox{.8}{$E^V_{(A,\sigma)}$}}
\varphi_V^l(A,\sigma)=d_V^{-3}\; \rsdraw{.37}{.9}{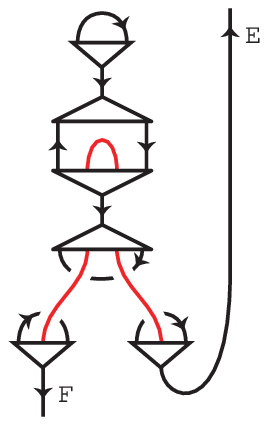} \; = d_V^{-3}\;\;\, \rsdraw{.37}{.9}{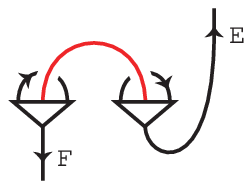}\;\;\,.
$$
A similar computation gives that
$$
 \psfrag{F}[Bl][Bl]{\scalebox{.8}{$E^V_{(A,\sigma)^*}$}}
 \psfrag{E}[Br][Br]{\scalebox{.8}{$E^V_{(A,\sigma)}$}}
\varphi_V^r(A,\sigma) = d_V^{-3}\;\;\, \rsdraw{.37}{.9}{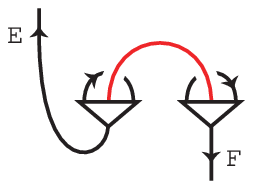}\;.
$$
Thus $\varphi_V^l(A,\sigma)=\varphi_V^r(A,\sigma)$ by the pivotality of $\cc$. This concludes the proof of Lemma~\ref{lem-phiV-def}. The proof
 of Lemmas \ref{lem-quotient-varphi}-\ref{lem-braiding-ZG} 
follows the same  lines  depicting the  morphisms involved  as
\begin{align*}
 &\psfrag{V}[Bc][Bc]{\scalebox{.8}{$V$}}
 \psfrag{U}[Bc][Bc]{\scalebox{.8}{$U$}}
 \psfrag{A}[Bc][Bc]{\scalebox{.8}{\textcolor{red}{$A$}}}
 \psfrag{E}[Bl][Bl]{\scalebox{.8}{$E^V_{(A,\sigma)}$}}
 \psfrag{F}[Bl][Bl]{\scalebox{.8}{$E^U_{(A,\sigma)}$}}
\delta^{U,V}_{(A,\sigma)}=  \,d_V^{-1}\rsdraw{.45}{.9}{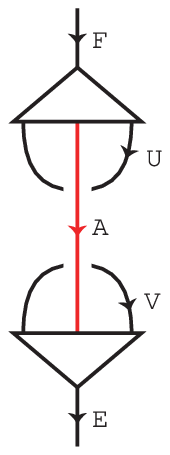} \;, &&
 \psfrag{V}[Bc][Bc]{\scalebox{.8}{$V$}}
 \psfrag{U}[Bc][Bc]{\scalebox{.8}{$U$}}
 \psfrag{W}[Bc][Bc]{\scalebox{.8}{$W$}}
 \psfrag{A}[Bc][Bc]{\scalebox{.8}{\textcolor{red}{$A$}}}
 \psfrag{E}[Bl][Bl]{\scalebox{.8}{$E^U_{\varphi_V(A,\sigma)}$}}
 \psfrag{F}[Bl][Bl]{\scalebox{.8}{$E^V_{(A,\sigma)}$}}
 \psfrag{G}[Bl][Bl]{\scalebox{.8}{$E^W_{(A,\sigma)}$}}
\zeta^{U,V,W}_{(A,\sigma)}=  d_U^{-1}d_V^{-1}\rsdraw{.45}{.9}{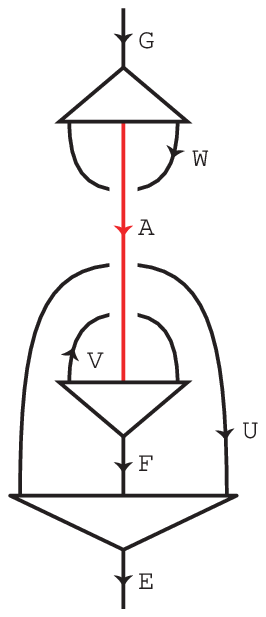} \;, \\
&
 \psfrag{E}[Bl][Bl]{\scalebox{.8}{$E^U_{(A,\sigma)}$}}
 \psfrag{U}[Bc][Bc]{\scalebox{.8}{$U$}}
 \psfrag{A}[Bc][Bc]{\scalebox{.8}{\textcolor{red}{$A$}}}
 \psfrag{F}[Bl][Bl]{\scalebox{.8}{$E^U_{(A,\sigma)}$}}
\eta^{U}_{(A,\sigma)}=  \; \rsdraw{.45}{.9}{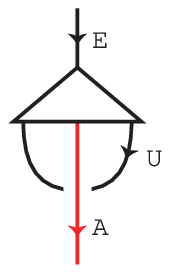} \;\;, &&
 \psfrag{V}[Bc][Bc]{\scalebox{.8}{$V$}}
 \psfrag{X}[Bc][Bc]{\scalebox{.8}{$X$}}
 \psfrag{A}[Bc][Bc]{\scalebox{.8}{\textcolor{red}{$A$}}}
 \psfrag{E}[Bl][Bl]{\scalebox{.8}{$E^V_{(A,\sigma)}$}}
\Gamma^{V}_{(A,\sigma),X}=  \rsdraw{.45}{.9}{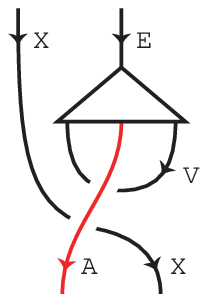} \;\;
\end{align*}
and depicting  the inverse  isomorphisms   by
\begin{gather*}
 \psfrag{V}[Bc][Bc]{\scalebox{.8}{$V$}}
 \psfrag{U}[Bc][Bc]{\scalebox{.8}{$U$}}
 \psfrag{W}[Bc][Bc]{\scalebox{.8}{$W$}}
 \psfrag{A}[Bc][Bc]{\scalebox{.8}{\textcolor{red}{$A$}}}
 \psfrag{E}[Bl][Bl]{\scalebox{.8}{$E^U_{\varphi_V(A,\sigma)}$}}
 \psfrag{F}[Bl][Bl]{\scalebox{.8}{$E^V_{(A,\sigma)}$}}
 \psfrag{G}[Bl][Bl]{\scalebox{.8}{$E^W_{(A,\sigma)}$}}
\bigl(\zeta^{U,V,W}_{(A,\sigma)}\bigr)^{-1}=  d_W^{-1}\rsdraw{.45}{.9}{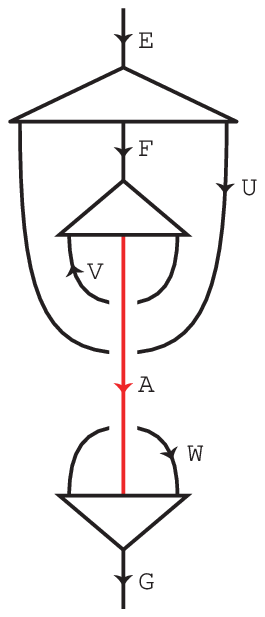} \;, \\
 \psfrag{E}[Bl][Bl]{\scalebox{.8}{$E^U_{(A,\sigma)}$}}
 \psfrag{U}[Bc][Bc]{\scalebox{.8}{$U$}}
 \psfrag{A}[Bc][Bc]{\scalebox{.8}{\textcolor{red}{$A$}}}
 \psfrag{F}[Bl][Bl]{\scalebox{.8}{$E^U_{(A,\sigma)}$}}
\bigl(\eta^{U}_{(A,\sigma)}\bigr)^{-1}=  d_U^{-1}\; \rsdraw{.45}{.9}{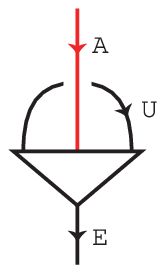} \;\;, \qquad \;\,
 \psfrag{V}[Bc][Bc]{\scalebox{.8}{$V$}}
 \psfrag{X}[Bc][Bc]{\scalebox{.8}{$X$}}
 \psfrag{A}[Bc][Bc]{\scalebox{.8}{\textcolor{red}{$A$}}}
 \psfrag{E}[Bl][Bl]{\scalebox{.8}{$E^V_{(A,\sigma)}$}}
\bigl(\Gamma^{V}_{(A,\sigma),X}\bigr)^{-1}= d_V^{-1}\; \rsdraw{.45}{.9}{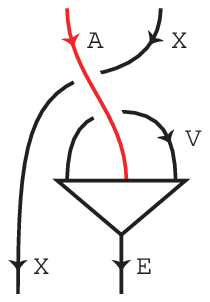} \;\;.
\end{gather*}
For example, let us check  Lemma~\ref{lem-quotient-Gamma}(e). Let $(A,\sigma) \in \zz_G(\cc)$, $(B,\rho) \in \zz_\beta(\cc)$, $U \in \ee_\alpha$, $V \in \ee_\beta$, $W \in \ee_{\beta \alpha}$, and $S \in \ee_{\alpha^{-1}\beta\alpha}$, with $\alpha,\beta \in G$. Then
\begin{center}
$ \displaystyle
(\varphi_U)_2\bigl((B,\rho),\varphi_V(A,\sigma)\bigr)\left(\id_{E^U_{(B,\rho)}} \otimes (\xi^{U,V,W}_{(A,\sigma)})^{-1}\xi^{S,U,W}_{(A,\sigma)}\right)\Gamma^S_{\varphi_U(A,\sigma),E^U_{(B,\rho)}}
$\\[.5em]
$ \displaystyle
\psfrag{E}[Bl][Bl]{\scalebox{.9}{\scriptsize $E^U_{(A,\sigma)}$}}
\psfrag{F}[Bl][Bl]{\scalebox{.9}{\scriptsize $E^U_{(B,\rho)}$}}
\psfrag{U}[Bc][Bc]{\scalebox{.9}{\scriptsize $U$}}
\psfrag{V}[Bc][Bc]{\scalebox{.9}{\scriptsize $V$}}
\psfrag{S}[Bc][Bc]{\scalebox{.9}{\scriptsize $S$}}
\psfrag{W}[Bc][Bc]{\scalebox{.9}{\scriptsize $W$}}
= d_S^{-1}d_U^{-4}d_W^{-1}\rsdraw{.45}{.9}{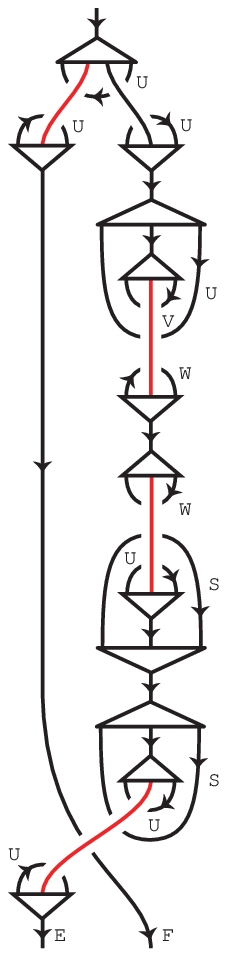}\overset{(i)}{=} d_S^{-2}d_U^{-5}\rsdraw{.45}{.9}{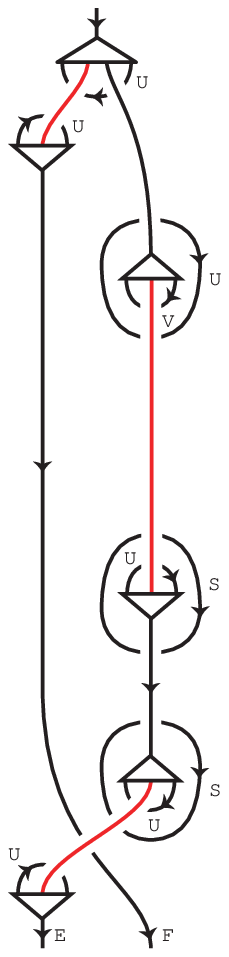}
\overset{(ii)}{=} d_S^{-2}d_U^{-5}d_V^{-1}\rsdraw{.45}{.9}{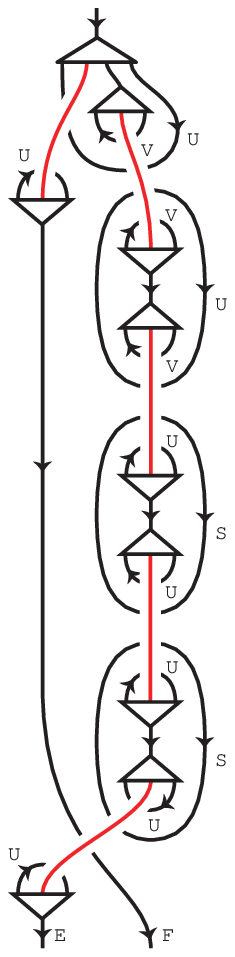}
$\\[.8em]
$ \displaystyle
\psfrag{E}[Bl][Bl]{\scalebox{.9}{\scriptsize $E^U_{(A,\sigma)}$}}
\psfrag{F}[Bl][Bl]{\scalebox{.9}{\scriptsize $E^U_{(B,\rho)}$}}
\psfrag{U}[Bc][Bc]{\scalebox{.9}{\scriptsize $U$}}
\psfrag{V}[Bc][Bc]{\scalebox{.9}{\scriptsize $V$}}
\psfrag{S}[Bc][Bc]{\scalebox{.9}{\scriptsize $S$}}
\psfrag{W}[Bc][Bc]{\scalebox{.9}{\scriptsize $W$}}
\overset{(iii)}{=}d_U^{-2}\;\rsdraw{.45}{.9}{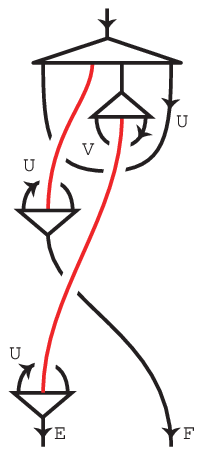}= d_U^{-2}\;\;\rsdraw{.45}{.9}{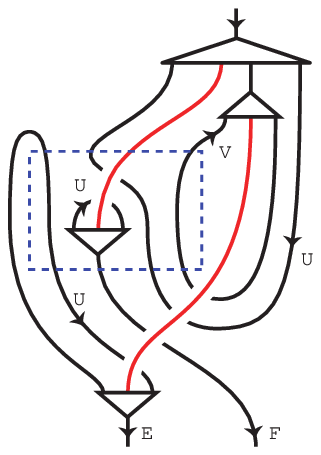}
\overset{(iv)}{=}d_U^{-2}\;\;\rsdraw{.45}{.9}{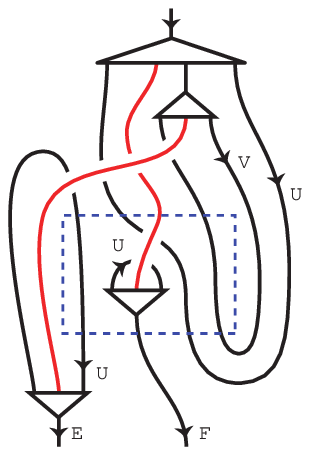}
$\\[.85em]
$ \displaystyle
\psfrag{E}[Bl][Bl]{\scalebox{.9}{\scriptsize $E^U_{(A,\sigma)}$}}
\psfrag{F}[Bl][Bl]{\scalebox{.9}{\scriptsize $E^U_{(B,\rho)}$}}
\psfrag{U}[Bc][Bc]{\scalebox{.9}{\scriptsize $U$}}
\psfrag{V}[Bc][Bc]{\scalebox{.9}{\scriptsize $V$}}
\psfrag{S}[Bc][Bc]{\scalebox{.9}{\scriptsize $S$}}
\psfrag{W}[Bc][Bc]{\scalebox{.9}{\scriptsize $W$}}
=d_U^{-2}\rsdraw{.45}{.9}{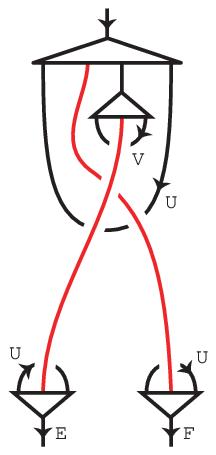}
\overset{(v)}{=}d_U^{-3}\rsdraw{.45}{.9}{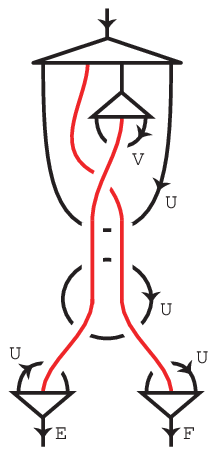} \overset{(vi)}{=}d_U^{-2}\rsdraw{.45}{.9}{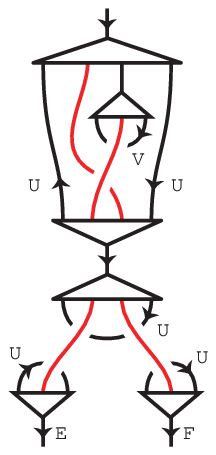}
$\\[1em]
$ \displaystyle
=\varphi_U\bigl(\Gamma^V_{(A,\sigma),B}\bigr) (\varphi_U)_2\bigl((A,\sigma),(B,\rho)\bigr).
$
\end{center}\vspace{1em}
\noindent Here, the equality $(i)$ is obtained by applying \eqref{eq-Drel1} twice and \eqref{eq-Drel6}, $(ii)$ follows from the definition of the half-braidings of $\varphi_U(A,\sigma)$ and $\varphi_V(A,\sigma)$, $(iii)$ is obtained by applying \eqref{eq-Drel5} and then \eqref{eq-dim-sigma}, $(iv)$ follows from the naturality of $\sigma$ applied to the morphism  delimited by the dotted blue box (which is indeed a morphism in $\cc_1$),  $(v)$  is obtained by applying \eqref{eq-dim-sigma}, and $(vi)$  is obtained by applying \eqref{eq-Drel1}.

\subsection{Remarks} 1.
The   $G$-center $\zz_G(\cc)$ is the \emph{left $G$-center} of $\cc$ while the center studied in \cite{GNN} is the \emph{right $G$-center} $\zz^r_G(\cc) $. The objects of $\zz^r_G(\cc) $ are right half braidings of $\cc$ relative to $\cc_1$, i.e., pairs  $({{A}}  \in \cc,\sigma)$  where
$
\sigma=\{\sigma_X \co   X\otimes {{A}}\to {{A}} \otimes X \}_{X \in \cc_1}
$
is a natural isomorphism such that
$\sigma_{X \otimes Y}=(\sigma_X \otimes \id_Y)(\id_X \otimes
\sigma_Y)$ for all
$X,Y \in \cc_1$. The left and right $G$-centers are related as follows.
Given a $G$-graded  category  $\dd$,   denote by $\dd^\rv$ the $G^\opp$-graded   category obtained from $\dd$ by replacing   $\otimes$ with  the opposite product $\otimes^\opp$  defined by $X \otimes^\opp Y=Y
\otimes X$  and by setting $(\dd^\rv)_\alpha=\dd_\alpha$ for $\alpha \in G$. Then
$\zz^r_G(\cc)=(\zz_{G^\opp}(\cc^{\rv}))^{\rv} $.

2. Each $G$-graded pure pivotal category $\cc$ with split idempotents contains a maximal non-singular graded subcategory. Namely, let $H$ be the set of all  $\alpha\in G$ such that $\cc_\alpha$ contains an object with invertible left dimension and an object with invertible right dimension. Then $H$ is a subgroup of $G$ and
 the $H$-category $ \oplus_{\alpha \in H}
  \,\cc_\alpha \subset \cc$ is non-singular.

\section{The modularity theorem}\label{sect-main}
We state in this section our main result concerning the modularity
of the $G$-center. We first discuss several important conditions on
categories.

\subsection{Sphericity}  A   \emph{spherical} category   is
  a pivotal   category $\cc$ such that  the left and right traces of endomorphisms in
$\cc$ coincide.   For any
endomorphism  $g$ of an object of such a $\cc$  set  $\tr (g)=\tr_l(g)=\tr_r(g)$ and   for any object $X\in \cc$  set
$\dim (X)=\dim_l(X)=\dim_r(X)=\tr(\id_X)$.  For instance, all ribbon  categories are  spherical, see \cite{Tu0}.


For a spherical category, the graphical calculus of Section \ref{sect-penrose} has the following additional feature:    the morphisms represented by
  diagrams in $\RR^2$ are invariant under isotopies of    the diagrams  in the 2-sphere $S^2=\RR^2\cup
\{\infty\}$. In other words, these morphisms are preserved under isotopies of    diagrams  in  $\RR^2$ and under isotopies pushing   arcs of
  diagrams across~$\infty$.  For example, the diagrams in Section \ref{sect-penrose}
representing $\tr_l(g)$ and $\tr_r(g)$ are related by such an
isotopy. The  sphericity   condition $\tr_l(g)=\tr_r(g)$ for all $g$   ensures the isotopy invariance.

 A $G$-graded category   is {\it spherical}  if it is  spherical  as a
monoidal category.

\subsection{Split semisimplicity}\label{sec-ssss}   An object  $i$ of  a $\kk$-additive category $\cc$ is
\emph{simple} if $\End_\cc(i)$ is a free \kt module of rank 1.   Then the map $\kk \to
\End_\cc(i), k \mapsto k \, \id_i$  is a \kt algebra isomorphism
which we use  to identify $\End_\cc(i)=\kk$.   All objects
isomorphic to a
 simple object are  simple. If $\cc$ is rigid, then the   left/right duals  of a   simple object of $\cc$ are  simple.

A $\kk$-additive category $\cc$ is \emph{split semisimple} if  each object of $\cc$ is a finite direct sum of  simple objects of $\cc$ and $\Hom_\cc(i,j)=0$ for any non-isomorphic  simple objects $i,j  $ of $\cc$.
 A  set $I$ of   simple objects of a split semisimple category $\cc$
is \emph{representative} if  every  simple object of $\cc$ is
isomorphic to a unique element of~$I$.   Then any   $X \in \cc$
splits as a (finite) direct sum of objects of $I$. In other words,  there
exists a finite family of morphisms $(p_\alpha \co X \to i_\alpha
\in I, q_\alpha \co i_\alpha \to X)_{\alpha \in \Lambda }$ in~$\cc$
such that
\begin{equation}\label{eq-prefusion-sum}
\id_X=\sum_{\alpha \in \Lambda} q_\alpha p_\alpha  \quad
\text{and} \quad p_\alpha q_\beta=\delta_{\alpha,\beta} \, \id_{i_\alpha} \quad
\text{for all} \quad \alpha,\beta\in \Lambda.
\end{equation}
Such a family $( p_\alpha  , q_\alpha )_{\alpha \in \Lambda }$  is
called  an \emph{$I$-partition} of $ X$.

   A split
semisimple   category  $\cc$  is {\it finite} if  the set of isomorphism
classes of simple objects of $\cc$ is finite. If a finite split semisimple category $\cc$ is pivotal, then
the \emph{dimension}   of
$\cc$ is defined by
$$\dim(\cc)=\sum_{i\in I} \dim_l(i)\dim_r(i) \in \End_\cc(\un) ,$$
where $I$ is a representative set   of   simple objects of $\cc$.
The sum here is well defined because $I$ is
finite and does not depend on the choice of~$I$.

 \subsection{$G$-fusion}\label{Pre-fusion and fusion  categories}   A \emph{$G$-pre-fusion category}   is a  $G$-graded  category $\cc$ (over $\kk$) such that the unit object
$\un$ is  simple and
\begin{enumerate}
  \labela
\item   $\cc$ is pivotal and  split semisimple as a $\kk$-additive category;
\item  for
all $\alpha\in G$, the category~$\cc_\alpha$  has at least one simple object.
\end{enumerate}

A  set $I$  of simple  objects of a $G$-pre-fusion category $\cc$ is
\emph{$G$-representative} if  all elements of $I$ are
homogeneous and   every  simple object of $\cc$ is
isomorphic to a unique element of~$I$. Any such  set  $I$ splits  as a disjoint
union $I=\amalg_{\alpha\in G}\, I_\alpha
 $  where $I_{\alpha}$ is the (non-empty) set   of all elements of $I$
 belonging to $\cc_{\alpha}$. The existence of a $G$-representative set $I$  follows from the fact that
 any simple object   of   $\cc$ is isomorphic
to a simple object of $\cc_\alpha$ for a unique $\alpha \in G$. Note also that    $\cc_{\alpha}$ is split semisimple for all $\alpha\in G$.

Any $G$-pre-fusion category $\cc$ is pure and both the left and right
dimensions of   simple objects of  $\cc$ are
invertible (see Lemma 4.1 of \cite{TVi1}). If $\kk$ is a field (or, more generally,   a local ring), then $\cc$ has   split idempotents.
 Therefore   any  $G$-pre-fusion category $\cc$ over
a field   is non-singular. Such a    $\cc$ satisfies   the
hypothesis of Theorem~\ref{thm-G-center}, and  so, $\zz_G(\cc)$
is a  pivotal  $G$-braided category with pivotal crossing.

 A \emph{$G$-fusion category}   is a
 $G$-pre-fusion category~$\cc$ (over $\kk$)  such that   the set of isomorphism
classes of simple objects of~$\cc_\alpha$   is finite for all $\alpha \in G$.

\subsection{Modularity}\label{sect-modularity-center}
A {\it $G$-modular category}  is  a $G$-ribbon $G$-fusion category $ \dd$  whose neutral component  $\dd_1$  is modular in the sense of \cite{Tu0}, that is,  the {\it $S$-matrix}
 $
(\tr\left(c_{j,i} c_{i,j} \right))_{i,j }
$
is invertible over $\kk$. Here $i,j$ run over a representative set of simple objects of $\dd_1$  and $c_{i,j}: i\otimes j \to i\otimes j$ is the braiding \eqref{eq-usualbraiding} in $\dd_1$.

\begin{thm}\label{thm-center-G-modular}
Let $\cc$ be a  spherical $G$-fusion
category over an algebraically closed field such that
$\dim(\cc_1)\neq 0$. Then   $\zz_G(\cc)$   is a $G$-modular
category.
\end{thm}
The proof of Theorem~\ref{thm-center-G-modular} given below is based on  the following two key lemmas.

\begin{lem}\label{lem-G-center-ribbon}
Let $\cc$ be a spherical $G$-pre-fusion category with split idempotents. Then  $\zz_G(\cc)$   is a $G$-ribbon
category.
\end{lem}

\begin{lem}\label{lem-center-G-fusion}
Let $\cc$ be a   $G$-fusion
category over an algebraically closed field such that
$\dim(\cc_1)\neq 0$. Then  $\zz_G(\cc)$   is a
$G$-fusion category.
\end{lem}

The proof of Lemma~\ref{lem-center-G-fusion} uses the following claim of independent interest.

\begin{lem}\label{thm-rel-center-semi}
Let $\cc$ be a    split semisimple  pivotal category over  an
algebraically closed field such that the unit object $\un_\cc$ is
simple. Let $\dd$ be a finite split semisimple pivotal subcategory
of $\cc$  (not necessarily full)  such that $\dim(\dd)\neq
0$. Then the relative center $\zz(\cc;\dd)$   is split semisimple.
\end{lem}

 Lemma~\ref{lem-G-center-ribbon} is proved in Section~\ref{sect-ribonness-ZG-new} and    Lemmas~\ref{lem-center-G-fusion} and~\ref{thm-rel-center-semi} are proved in Section~\ref{sect-bigproof-semiFULL}. The arguments  in Section  \ref{sect-bigproof-semiFULL}  use  the results of Sections \ref{sect-big-HM---} and \ref{sect-big-HM} concerned with   monads and coends, respectively.

\subsection{Proof of  Theorem~\ref*{thm-center-G-modular}}
By Lemmas~\ref{lem-G-center-ribbon} and~\ref{lem-center-G-fusion},   $\zz_G(\cc)$ is  a $G$-ribbon  $G$-fusion category.
The neutral component $\zz_1(\cc)$ of $\zz_G(\cc)$ is isomorphic to
the center $\zz(\cc_1)=\zz(\cc_1;\cc_1)$ of $\cc_1$. Since $\cc$ is    spherical,   so is~$\cc_1$. By \cite[Theorem 1.2]{Mu},
  $\zz(\cc_1) $ is modular.  Therefore $\zz_G(\cc)$ is
$G$-modular.

\subsection{Example}\label{sect-ex-Z-H-vect}    The $G$-ribbon category  $\dd =\dd(\pi)$   derived from
a  group epimorphism $\pi \co H \to G$   in
Section~\ref{sect-ex-H-vect-rib} can be realized as the $G$-center
of a non-singular   pivotal $G$-graded category. To   see this,
observe  that any  pivotal $H$-graded category $\cc$ gives rise to a
pivotal $G$-graded category      $\pi_*(\cc)$ which is equal to
$\cc$ as a pivotal category and has  the grading
$\pi_*(\cc)=\oplus_{\alpha \in G} \,\pi_*(\cc)_\alpha$ where
$\pi_*(\cc)_\alpha=\oplus_{h \in \pi^{-1}(\alpha)} \cc_h$. We apply
this observation to   the $H$-ribbon   category
$\cc=\dd(\id_H\co H\to
 H) $  of $H$-graded
finitely generated projective $\kk$-modules.
  It is easy to see that $\cc$ is spherical and has split
 idempotents.
 For $h\in H$, let $\kk[h]\in \cc$ be the $\kk$-module   which is $\kk$ in degree $h$ and   zero   in all other degrees.
 The category $\cc$    is  non-singular  since  its monoidal unit $\kk[1]$ is simple and for $h\in H$, the module $\kk[h]$ has categorical dimension 1.
This implies that the category $\cc^\pi=\pi_*(\cc)$ is non-singular.
By Theorem~\ref{thm-G-center},  the $G$-center    of $\cc^\pi$
 is a $G$-braided category.
 We claim that
  the $G$-braided categories $\dd$  and $\zz_G(\cc^\pi)$ are  equivalent. To see this, let
 $s\co G \to H$ be the map used in the definition of $\dd$ in Section~\ref{sect-ex-H-vect-rib}. Using
  the objects $(\kk[s(\alpha)])_{\alpha \in G}$  of $\cc $, one can explicitly describe  $\zz_G(\cc^\pi)$, cf.  Sections~\ref{sect-action of-G}--\ref{sect-ribonness-ZG}. In particular, a relative half braiding $(M,\sigma)\in \zz_\alpha(\cc^\pi)$ defines a right action of  $K=\Ker\, \pi$ on $M$ by  $m\cdot k=\sigma_{\kk[s(\alpha)]}(m \otimes 1_\kk)$ for $m \in M$ and $k \in K$. This determines the required  equivalence  $ \zz_G(\cc^\pi) \approx  \dd$.
Furthermore, if   $K$  is  finite
and    $\kk$ is an algebraically closed  field whose characteristic does not divide the order  $\# K$ of $K$, then it is
 easy to see that $\cc^\pi $ is a spherical $G$-fusion category and $\dim(\cc^\pi_1)=\# K $.   Theorem~\ref{thm-center-G-modular}  implies  that $\zz_G(\cc^\pi)$ is a $G$-modular category. We deduce that in this case  the category
   $\dd$ is $G$-modular.

\section{$G$-ribbonness re-examined}\label{sect-ribonness-ZG-new}

\subsection{A ribbonness criterion}\label{sect-ribonness-ZG}
Consider a non-singular $G$-graded pivotal category $\cc$. Let,  for
$\alpha \in G$, the symbol $\ee_\alpha$   denote the class of all
$V\in \cc_\alpha$ with invertible left dimension $d_V=\dim_l (V)\in
\kk$. By Theorem~\ref{thm-G-center}, the $G$-center $\zz_G(\cc)$ is
a  pivotal  $G$-braided category with pivotal crossing. The
corresponding twist $\theta$   in $\zz_G(\cc)$ is computed as
follows: if $(A,\sigma) \in \zz_\alpha(\cc)$ with $\alpha \in G$,
then  for any $U \in \ee_\alpha$, we have
$\theta_{(A,\sigma)}=(\iota^\alpha_U)^{-1}_{(A,\sigma)}\vartheta^{U}_{(A,\sigma)}$
where $\iota^\alpha$ is the universal cone  \eqref{univ-cones} and
$$
\vartheta^{U}_{(A,\sigma)}=  \;\psfrag{A}[Bl][Bl]{\scalebox{.7}{$A$}}
\psfrag{B}[Br][Br]{\scalebox{.8}{$A$}}
 \psfrag{V}[Bl][Bl]{\scalebox{.8}{$U$}}
   \psfrag{E}[Bl][Bl]{\scalebox{.8}{$E_{(A,\sigma)}^U$}}
 \psfrag{t}[Bc][Bc]{\scalebox{.8}{$p_{(A,\sigma)}^U$}}
  \psfrag{s}[Bc][Bc]{\scalebox{.8}{$\sigma_{U \otimes U^*}$}}
 \rsdraw{.45}{1}{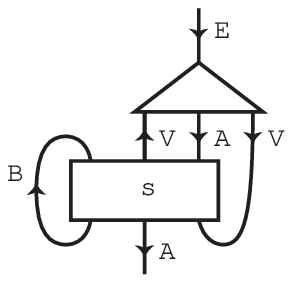}\; .
$$
 By definition, $\zz_G(\cc)$ is  $G$-ribbon if
$\theta$  is self-dual, see Section~\ref{sect-pivotality}.
The following lemma gives a
necessary and sufficient  condition for $\zz_G(\cc)$ to be $G$-ribbon.
\begin{lem}\label{lem-rib-car-center}
$\zz_G(\cc)$ is $G$-ribbon if and only if
\begin{center}
\psfrag{M}[Br][Br]{\scalebox{.8}{$A$}}
\psfrag{s}[Bc][Bc]{\scalebox{1}{$\sigma_{A \otimes U^*}$}}
\psfrag{A}[Bl][Bl]{\scalebox{.8}{$A$}}
\psfrag{V}[Bl][Bl]{\scalebox{.8}{$U$}}
\rsdraw{.45}{.9}{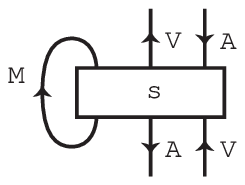} \; $=$ \;
\psfrag{s}[Bc][Bc]{\scalebox{1}{$\sigma_{U^* \otimes A}$}}
\rsdraw{.45}{.9}{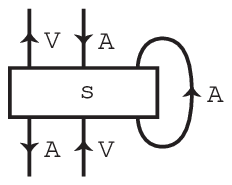}
\end{center}
for all $\alpha \in G$, $(A,\sigma) \in \zz_\alpha(\cc)$, and $U \in \ee_\alpha$.
\end{lem}
\begin{proof}
Recall that
$\theta$ is self-dual if for all $\alpha \in G$ and $(A,\sigma) \in \zz_\alpha(\cc)$,
\begin{equation}\label{eq-dem-Z-rib1}
(\theta_{(A,\sigma)})^*= \bigl((\varphi_0)_{(A,\sigma)}\bigr)^* \bigl(\varphi_2(\alpha^{-1},\alpha)_{(A,\sigma)}^{-1}\bigr)^* \varphi_{\alpha^{-1}}^1\bigl(\varphi_{\alpha}(A,\sigma)\bigr) \theta_{(\varphi_{\alpha}(A,\sigma))^*}.
\end{equation}
Pick any  $U \in \ee_\alpha$, $V \in \ee_{\alpha^{-1}}$, and $R \in \ee_1$. Composing \eqref{eq-dem-Z-rib1} on the right with $(\iota^\alpha_U)^*$, we rewrite   \eqref{eq-dem-Z-rib1} in the equivalent form
\begin{equation}\label{eq-dem-Z-rib2}
(\vartheta^{U}_{(A,\sigma)})^*= \bigl(\eta^R_{(A,\sigma)}\bigr)^* \bigl((\xi^{V,U,R}_{(A,\sigma)})^{-1}\bigr)^* \varphi_{V}^1\bigl(\varphi_{U}(A,\sigma)\bigr) \vartheta^V_{(\varphi_{U}(A,\sigma))^*}.
\end{equation}
Now,  using   the pictorial formalism of Section~\ref{sect-proof-lems-defdouble}, Lemma~\ref{lem-pict-demo}, and
the definition of $\vartheta$, $\eta$, $\xi$, $\varphi_{V}^1$,
we obtain that the left-hand side of \eqref{eq-dem-Z-rib2} is equal to
\begin{center}
\vspace{.5em}
$ \displaystyle
\psfrag{E}[Br][Br]{\scalebox{.9}{$E^U_{(A,\sigma)}$}}
\psfrag{A}[Bl][Bl]{\scalebox{.9}{\textcolor{red}{$A$}}}
\psfrag{U}[Bc][Bc]{\scalebox{.9}{\scriptsize $U$}}
\rsdraw{.45}{.9}{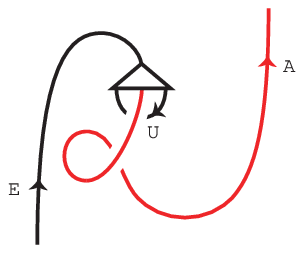}
$
\end{center}
while the right-hand  side of \eqref{eq-dem-Z-rib2} is equal to
\begin{center}
\vspace{.5em}
$ \displaystyle
\psfrag{E}[Br][Br]{\scalebox{.9}{$E^U_{(A,\sigma)}$}}
\psfrag{F}[Bl][Bl]{\scalebox{.9}{\scriptsize $E^U_{(A,\sigma)}$}}
\psfrag{H}[Br][Br]{\scalebox{.9}{\scriptsize $E^V_{\varphi_U(A,\sigma)^*}$}}
\psfrag{G}[Bl][Bl]{\scalebox{.9}{\scriptsize $E^V_{\varphi_U(A,\sigma)}$}}
\psfrag{K}[Bl][Bl]{\scalebox{.9}{\scriptsize $E^Z_{(A,\sigma)}$}}
\psfrag{A}[Bl][Bl]{\scalebox{.9}{\textcolor{red}{$A$}}}
\psfrag{U}[Bc][Bc]{\scalebox{.9}{\scriptsize $U$}}
\psfrag{V}[Bc][Bc]{\scalebox{.9}{\scriptsize $V$}}
\psfrag{Z}[Bc][Bc]{\scalebox{.9}{\scriptsize $Z$}}
d_U^{-1}d_V^{-1}d_Z^{-1}\rsdraw{.45}{.9}{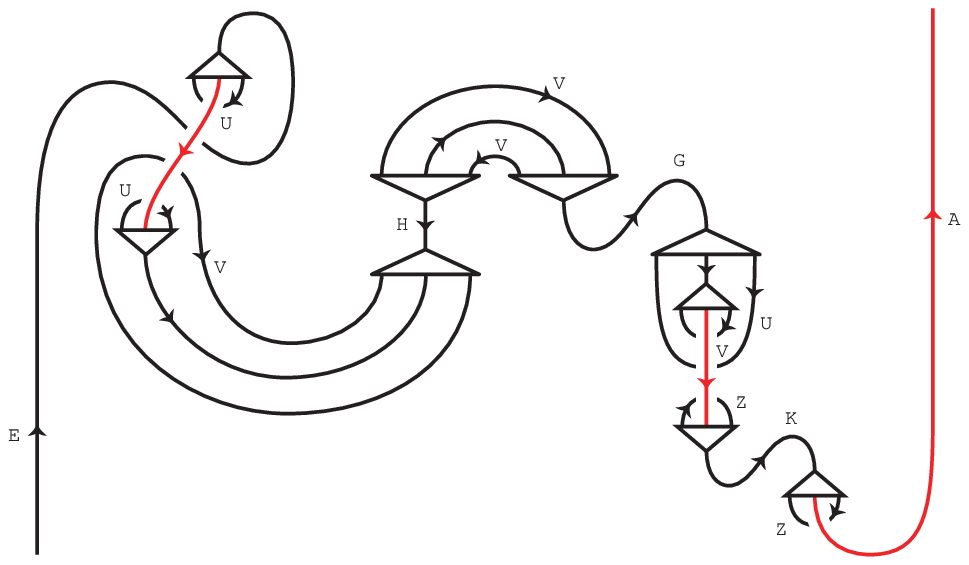}
$\\[1em]
$ \displaystyle
\psfrag{E}[Br][Br]{\scalebox{.9}{$E^U_{(A,\sigma)}$}}
\psfrag{F}[Bl][Bl]{\scalebox{.9}{\scriptsize $E^U_{(A,\sigma)}$}}
\psfrag{H}[Br][Br]{\scalebox{.9}{\scriptsize $E^V_{\varphi_U(A,\sigma)^*}$}}
\psfrag{G}[Bl][Bl]{\scalebox{.9}{\scriptsize $E^V_{\varphi_U(A,\sigma)}$}}
\psfrag{K}[Bl][Bl]{\scalebox{.9}{\scriptsize $E^Z_{(A,\sigma)}$}}
\psfrag{A}[Bl][Bl]{\scalebox{.9}{\textcolor{red}{$A$}}}
\psfrag{U}[Bc][Bc]{\scalebox{.9}{\scriptsize $U$}}
\psfrag{V}[Bc][Bc]{\scalebox{.9}{\scriptsize $V$}}
\psfrag{Z}[Bc][Bc]{\scalebox{.9}{\scriptsize $Z$}}
= \,d_U^{-1}d_V^{-2}\; \rsdraw{.45}{.9}{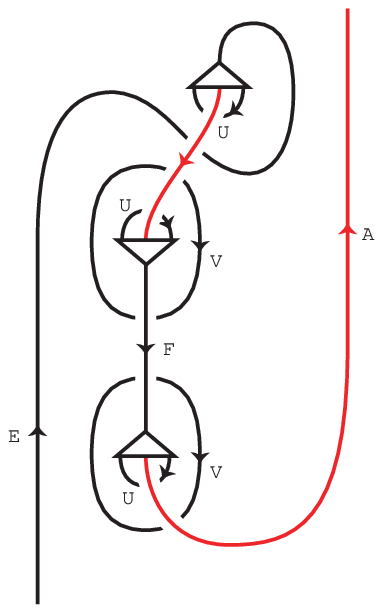}= \,d_U^{-2}d_V^{-2}\; \rsdraw{.45}{.9}{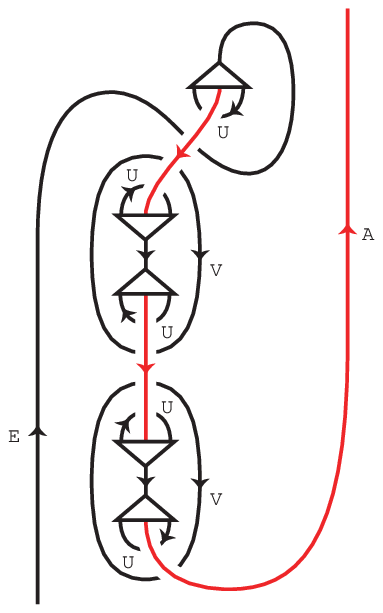}
$\\[.8em]
$ \displaystyle
\psfrag{E}[Br][Br]{\scalebox{.9}{$E^U_{(A,\sigma)}$}}
\psfrag{A}[Bl][Bl]{\scalebox{.9}{\textcolor{red}{$A$}}}
\psfrag{U}[Bc][Bc]{\scalebox{.9}{\scriptsize $U$}}
= \; \rsdraw{.45}{.9}{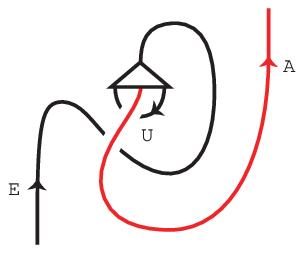}\,.
$
\end{center}
Therefore
\begin{center}
\vspace{.5em} \eqref{eq-dem-Z-rib2} $\quad \Longleftrightarrow \quad $
$ \displaystyle
\psfrag{E}[Br][Br]{\scalebox{.9}{$E^U_{(A,\sigma)}$}}
\psfrag{A}[Bl][Bl]{\scalebox{.9}{\textcolor{red}{$A$}}}
\psfrag{U}[Bc][Bc]{\scalebox{.9}{\scriptsize $U$}}
\rsdraw{.45}{.9}{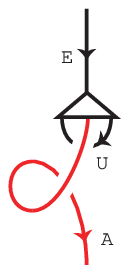} \; \,= \; \,\rsdraw{.45}{.9}{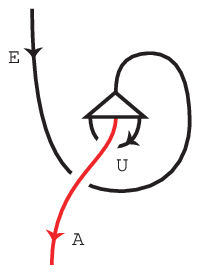} \quad \Longleftrightarrow \quad \rsdraw{.45}{.9}{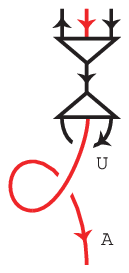} \; = \; \rsdraw{.45}{.9}{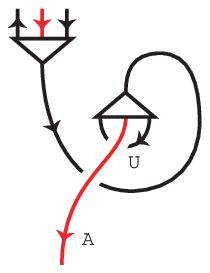}
$\\[1em]
$ \displaystyle
\psfrag{A}[Bl][Bl]{\scalebox{.9}{\textcolor{red}{$A$}}}
\psfrag{B}[Br][Br]{\scalebox{.9}{\textcolor{red}{$A$}}}
\psfrag{U}[Bl][Bl]{\scalebox{.9}{$U$}}
\Longleftrightarrow \quad \rsdraw{.45}{.9}{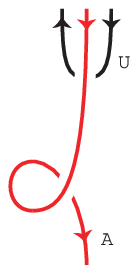} \; = \; \rsdraw{.45}{.9}{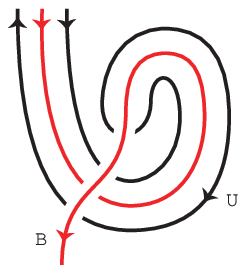}
\quad \Longleftrightarrow \quad \rsdraw{.45}{.9}{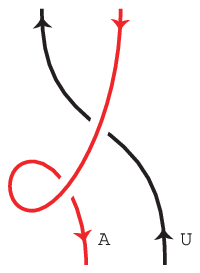} \; = \; \rsdraw{.45}{.9}{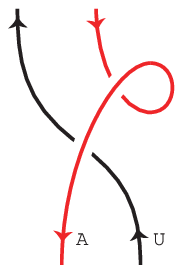}.
$
\end{center}
\vspace{.5em}
The last equality is exactly the condition of the lemma.
\end{proof}

\subsection{Proof of Lemma~\ref*{lem-G-center-ribbon}}\label{sect-ribonness-ZG+}  We   check the criterion of
Lemma~\ref{lem-rib-car-center} for any
    $(A,\sigma) $  and $U $. Pick a representative set $I$ of simple objects of
$\cc$. Let $$( p_\lambda\co A \otimes U^* \to i_\lambda ,
q_\lambda\co i_\lambda \to  A \otimes U^*)_{\lambda \in \Lambda }
\,,  \, ( p'_\omega\co U^*\otimes A \to i_\omega , q'_\omega\co
i_\omega \to  U^*\otimes A)_{\omega \in \Omega }$$ be
  $I$-partitions of $ A \otimes U^*$ and
  $ U^*\otimes A$, respectively. For any
$\lambda \in \Lambda$  and $\omega\in \Omega$ such that
$i_\lambda=i_\omega=i \in I$,   we have
\begin{center}
\psfrag{M}[Br][Br]{\scalebox{.8}{$A$}}
\psfrag{s}[Bc][Bc]{\scalebox{1}{$\sigma_{A \otimes U^*}$}}
\psfrag{A}[Bl][Bl]{\scalebox{.8}{$A$}}
\psfrag{V}[Bl][Bl]{\scalebox{.8}{$U$}}
\psfrag{i}[Bc][Bc]{\scalebox{.8}{$i$}}
\psfrag{p}[Bc][Bc]{\scalebox{1}{$p'_\omega$}}
\psfrag{q}[Bc][Bc]{\scalebox{1}{$q_\lambda$}}
\rsdraw{.45}{.9}{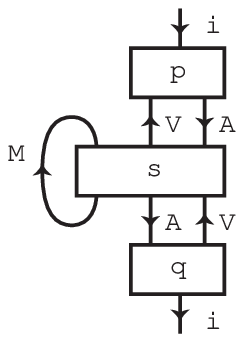}
\psfrag{M}[Br][Br]{\scalebox{.8}{$A$}}
\psfrag{s}[Bc][Bc]{\scalebox{1}{$\sigma_{A \otimes U^*}$}}
\psfrag{A}[Bl][Bl]{\scalebox{.8}{$A$}}
\psfrag{V}[Bl][Bl]{\scalebox{.8}{$U$}}
\psfrag{i}[Bc][Bc]{\scalebox{.8}{$i$}}
\psfrag{p}[Bc][Bc]{\scalebox{1}{$p'_\omega$}}
\psfrag{q}[Bc][Bc]{\scalebox{1}{$q_\lambda$}}
  $\,= \,d_i^{-1}$  \rsdraw{.45}{.9}{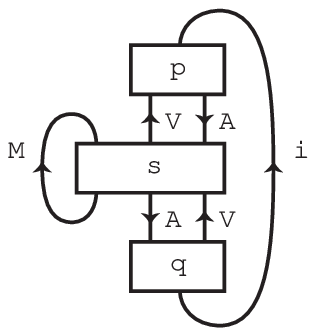}\, $\id_i$ \\[.5em]
  \psfrag{M}[Br][Br]{\scalebox{.8}{$A$}}
\psfrag{s}[Bc][Bc]{\scalebox{1}{$\sigma_{i}$}}
\psfrag{A}[Bl][Bl]{\scalebox{.8}{$A$}}
\psfrag{V}[Bl][Bl]{\scalebox{.8}{$U$}}
\psfrag{i}[Bc][Bc]{\scalebox{.8}{$i$}}
\psfrag{p}[Bc][Bc]{\scalebox{1}{$p'_\omega$}}
\psfrag{q}[Bc][Bc]{\scalebox{1}{$q_\lambda$}}
 $ = \,d_i^{-1} $ \rsdraw{.45}{.9}{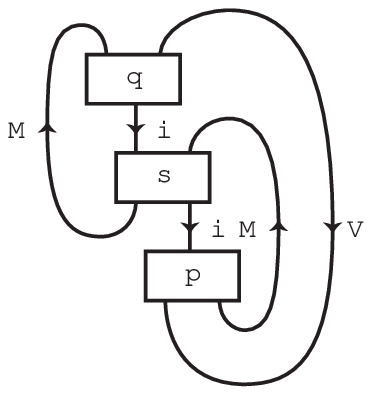}\, $\id_i$ \;
  \psfrag{V}[Br][Br]{\scalebox{.8}{$U$}}
  $= \,d_i^{-1}$  \rsdraw{.45}{.9}{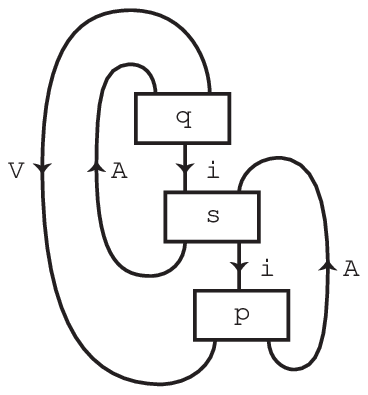} $\id_i$ \\[.5em]
\psfrag{M}[Br][Br]{\scalebox{.8}{$A$}}
\psfrag{s}[Bc][Bc]{\scalebox{1}{$\sigma_{U^* \otimes A}$}}
\psfrag{A}[Bl][Bl]{\scalebox{.8}{$A$}}
\psfrag{V}[Bl][Bl]{\scalebox{.8}{$U$}}
\psfrag{i}[Bc][Bc]{\scalebox{.8}{$i$}}
\psfrag{p}[Bc][Bc]{\scalebox{1}{$p'_\omega$}}
\psfrag{q}[Bc][Bc]{\scalebox{1}{$q_\lambda$}} $= \,d_i^{-1}$
\rsdraw{.45}{.9}{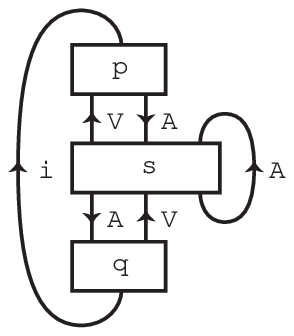}  \,$\id_i$ \;$=$\;
\psfrag{M}[Br][Br]{\scalebox{.8}{$A$}}
\psfrag{s}[Bc][Bc]{\scalebox{1}{$\sigma_{U^* \otimes A}$}}
\psfrag{A}[Bl][Bl]{\scalebox{.8}{$A$}}
\psfrag{V}[Bl][Bl]{\scalebox{.8}{$U$}}
\psfrag{i}[Bc][Bc]{\scalebox{.8}{$i$}}
\psfrag{p}[Bc][Bc]{\scalebox{1}{$p'_\omega$}}
\psfrag{q}[Bc][Bc]{\scalebox{1}{$q_\lambda$}}
\rsdraw{.45}{.9}{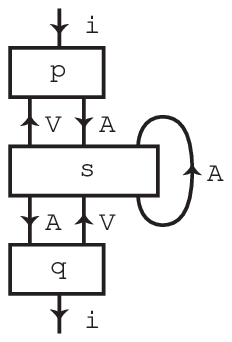}\,.
\end{center}
Here the first and the last equalities follow   from the  simplicity
of $i$ and the formulas $d_i=\dim_l(i)=\dim_r(i)$. The other
equalities follow from the   isotopy invariance in $S^2$ and the
naturality of $\sigma$. We conclude using that any morphism $f\co A
\otimes U^* \to U^*\otimes A$   expands as
$f=\sum_{\omega\in\Omega,\lambda\in\Lambda} q'_\omega (p'_\omega f
q_\lambda) p_\lambda$ where $p'_\omega f q_\lambda=0$ for
$i_\lambda\neq i_\omega$.

\section{Monads and Hopf monads}\label{sect-big-HM---}

Monads,  bimonads, and  Hopf monads generalize respectively
algebras, bialgebras, and Hopf algebras to the categorical setting.
The concept of a monad originated  in Godement's work on sheaf
cohomology in the 1950s.
Bimonads were introduced by Moerdijk \cite{Moer}  in 2002. Hopf
monads were introduced   by A. Brugui\`eres and the second author
\cite{BV2}  in 2006, see also   \cite{BV3,BLV}. We recall here the
basics of the theory of monads, bimonads, and Hopf monads needed in
the sequel.

\subsection{Monads and   modules}\label{Monads}
Given  a category $\cc$, we denote by $\End(\cc)$   the category
whose objects are endofunctors of $\cc$ (that is, functors   $\cc\to
\cc$) and morphisms are natural transformations between
the endofunctors.  The category $\End(\cc)$  is a  strict  monoidal category
with tensor product being composition of endofunctors and unit
object being the identity functor $1_\cc\co \cc\to \cc$. A
\emph{monad} on $\cc$  is a monoid in the category $\End(\cc)$, that
is, a triple $(T \in \End(\cc),\mu,\eta)$ consisting of a functor $T\co \cc \to
\cc$ and two natural transformations $$\mu=\{\mu_X\co T^2(X) \to
T(X)\}_{X \in \cc}\quad {\text {and}} \quad  \eta=\{\eta_X\co X \to
T(X)\}_{X \in \cc}$$  called the \emph{product} and the \emph{unit}
of $T$, such that for all $X\in\cc$, $$\mu_X
T(\mu_X)=\mu_X\mu_{T(X)} \quad {\text {and}} \quad
\mu_X\eta_{T(X)}=\id_{T(X)}=\mu_X T(\eta_X).$$ For example, the
identity functor $1_\cc\co \cc \to \cc $ is a monad on $\cc$ with
identity as product and unit. This  is   the
{\it trivial monad}.

Given a monad $T=(T, \mu, \eta)$ on $\cc$, a  {\it $T$\ti module} is a pair $(M \in\cc,r)$ where  $r\co T(M) \to M$ is a
morphism
  in $\cc$ such that $r T(r)= r \mu_M$ and $r \eta_M= \id_M$. We
  call such a morphism $r$ \emph{an action of $T$ on $M$}.
  A morphism from a $T$\ti module  $(M,r)$ to a $T$\ti module $(N,s)$ is a morphism $f \co M \to N$ in $\cc$ such that $f r=s T(f)$.
  This defines the {\it category  of $T$-modules}, $\cc^T$, with composition induced by that in $\cc$. We
  define
 a \emph{forgetful functor}  $U_T:\cc^T \to \cc$    by $U_T(M,r)=M$ and
 $U_T(f)=f$. We also define  the \emph{free module functor} $F_T\co\cc \to
\cc^T$ by
$F_T(X)=(T(X),\mu_X)$ for   $X\in \cc$ and $F_T(f)=T(f)$ for any morphism $f$ in $\cc$. The functors $F_T$ and $U_T$ are adjoint: there is a system of bijections $\Hom_{\cc^T} (F(X), Y)\cong \Hom_{\cc}(X, U_T(Y))$ natural in $X\in \cc$, $Y\in \cc^T$.  For an endofunctor $Q$ of
$\cc^T$,   the functor $Q\rtimes T=U_TQF_T\co \cc\to \cc $ is
 called the \emph{cross product} of $Q$ with $T$. For example, $1_{\cc^T}\rtimes T=T$. If  $T$ is the trivial monad, then
 $\cc^T=\cc$,  $F_T=U_T=1_\cc$, and $Q\rtimes T=Q$ for all $Q\in \End(\cc)$.

If $\cc$ is $\kk$-additive and $T$ is $\kk$-linear, then the category $\cc^T$
is $\kk$-additive and the   functors $U_T$, $F_T$ are $\kk$-linear.

\subsection{Comonoidal functors}\label{sect-comonofunctor} To introduce bimonads and Hopf monads we  ought to replace endofunctors
in the definitions above by comonoidal endofunctors. We recall here the relevant definitions.

 Let
$\cc$ and $\dd$ be  monoidal categories. A \emph{comonoidal functor}
from $\cc$ to $\dd$ is a triple $(F,F_2,F_0)$, where $F\co \cc \to
\dd$ is a functor, $$ F_2=\{F_2(X,Y) \co  F(X \otimes Y)\to F(X)
\otimes F(Y)\}_{X,Y \in \cc} $$ is a natural transformation from $F
\otimes$ to $F\otimes F$, and $F_0\co F(\un) \to \un$ is a morphism
in~$\dd$, such that
\begin{align*}
& \bigl(\id_{F(X)} \otimes F_2(Y,Z)\bigr) F_2(X,Y \otimes Z)= \bigl(F_2(X,Y) \otimes \id_{F(Z)}\bigr) F_2(X \otimes Y, Z) ;\\
& (\id_{F(X)} \otimes F_0) F_2(X,\un)=\id_{F(X)}=(F_0 \otimes \id_{F(X)}) F_2(\un,X) ;
\end{align*}
for all objects $X,Y,Z$ of $\cc$.   A comonoidal functor
$(F,F_2,F_0)$ is   \emph{strong} (resp.\@ \emph{strict}) if $F_0$
and all the morphisms $F_2 (X,Y)$ are isomorphisms (resp.\@
identities). The formula $(F,F_2,F_0)\mapsto (F,F^{-1}_2,F^{-1}_0)$
establishes a bijective correspondence between strong (resp.\@
 {strict}) comonoidal functors
  and strong (resp.\@  {strict}) monoidal functors.

 A natural transformation $\varphi=\{\varphi_X \co F(X) \to
G(X)\}_{X \in \cc}$ from a comonoidal
functor  $F\co\cc \to \dd$ to a comonoidal
functor  $G\co\cc \to \dd$ is  \emph{comonoidal} if $
G_0 \varphi_\un= F_0$ and $G_2(X,Y) \varphi_{X
\otimes Y}= (\varphi_X \otimes \varphi_Y) F_2(X,Y)
$
for all   $X,Y\in \cc$.

If $F\co \cc \to \dd$ and $G\co \dd \to \ee$ are   comonoidal
functors between monoidal categories, then their composition $GF\co
\cc \to \ee$
is a comonoidal functor with 
$$(GF)_0=G_0G(F_0)  \quad \text{and} \quad
(GF)_2=\{ G_2(F(X),F(Y))G(F_2(X,Y)\}_{X,Y \in \cc}
.$$


\subsection{Bimonads}
For a monoidal category $\cc$, denote by $\Endcm(\cc)$ the
 category   whose objects
are comonoidal endofunctors of $\cc$  and   morphisms are comonoidal natural
transformations. The category $\Endcm(\cc)$ is
strict monoidal with composition of comonoidal endofunctors as
tensor product and the identity functor $1_\cc $ as
monoidal unit. A \emph{bimonad} on    $\cc$ is a monoid in the
category $\Endcm(\cc)$. In other words, a bimonad on $\cc$ is a
monad $(T,\mu,\eta)$ on $\cc$ such that the  functor $T\co \cc \to
\cc$ and the natural transformations $\mu$ and $\eta$ are
comonoidal. For example, the trivial  monad   on $\cc$ with
identity morphisms for comonoidal structure and for $\mu$ and $\eta$ is a bimonad  called the {\it
trivial bimonad}.

  Let $T=((T,T_2,T_0),\mu,\eta)$ be a bimonad on a monoidal
category $\cc$. The category  of $T$\ti modules $\cc^T$  has a
monoidal structure with unit object $(\un,T_0)$ and with tensor
product
$$(M,r) \otimes (N,s)=\bigl(M \otimes N, (r \otimes s) \,
T_2(M,N)\bigr). $$ By \cite[Sect.\ 3.3]{BV3}, the forgetful functor
$U_T\co \cc^T \to \cc$ is strict monoidal while the free module
functor $F_T\co \cc\to \cc^T$ is comonoidal  with   $(F_T)_0=T_0$
and  $(F_T)_2(X,Y)=T_2(X,Y)$ for any $X,Y\in \cc$. By \cite[Sect.\
3.7]{BV3}, for any   $Q \in \Endcm(\cc^{T})$, the   cross product
   $Q \rtimes {T} =U_TQF_T \in \End(\cc)$   is
comonoidal with
  $(Q\rtimes {T})_0=Q_0Q((F_{T})_0)$ and
 $
(Q\rtimes {T})_2=Q_2Q((F_{T})_2)  $. The formula $Q\mapsto
Q\rtimes {T}$ defines a monoidal functor
\begin{equation}\label{rhorho+} ?\rtimes {T} \co \Endcm(\cc^{T}) \to
\Endcm(\cc) \end{equation}   with monoidal structure
$$
 ( (?\rtimes {T})_0)_X=\eta_X:X\to T(X) \quad {\text {and}} \quad  ((?\rtimes {T})_2(Q,R))_X=U_{T}Q(\varepsilon_{R F_{T}(X)})
$$
for any $X\in \cc$ and $Q,R \in \Endcm(\cc^{T})$, where
$\varepsilon$ is the counit of the adjunction $(F_{T},U_{T})$,
that is, the natural transformation $ F_{T}U_{T} \to 1_{\cc^{T}}$
carrying $(M,r) \in \cc^{T}$ to~$r$.

\subsection{Hopf monads} Given a  bimonad $(T,\mu,\eta)$ on a monoidal category $\cc$ and objects $X, Y\in \cc$,
one defines  the \emph{left fusion morphism}
$$H^l_{X,Y} =(T(X) \otimes \mu_Y)\, T_2(X,T(Y)) \colon T(X\otimes T(Y)) \to T(X) \otimes
T(Y)$$
 and the \emph{right fusion morphism}
$$
 H^r_{X,Y}=(\mu_X \otimes
T(Y))\, T_2(T(X),Y)\colon T(T(X)\otimes Y) \to T(X)\otimes T(Y),$$ see
\cite{BLV}. A \emph{Hopf monad} on   $\cc$ is a bimonad on $\cc$
whose left and right fusions are isomorphisms for all $X, Y\in \cc$.
For example, the trivial  bimonad   on $\cc$ is a Hopf monad  called the {\it
trivial Hopf monad}.

  When $\cc$ is a rigid  category (see Section~\ref{rigid categories}) and     $T$  is a Hopf monad on
$\cc$, the monoidal category  $\cc^T$ has a canonical   structure of a rigid category. This structure can be
computed from the natural transformations
$$s^l=\{s^l_X\co T(\leftidx{^\vee}{T}{}(X)) \to \leftidx{^\vee}{X}{}\}_{X \in \cc}
\quad {\text {and}} \quad s^r=\{s^r_X\co T(T(X)^\vee) \to
X^\vee\}_{X \in \cc}$$ called the \emph{left and right antipodes}
and determined by the   fusion morphisms:
$$
s^l_X= \bigl(T_0T(\lev_{T(X)})(H^l_{\leftidx{^\vee}{T}{}(X),X})^{-1} \otimes \leftidx{^\vee}{\eta}{_X}\bigr)
\bigl(\id_{T(\leftidx{^\vee}{T}{}(X))} \otimes \lcoev_{T(X)}\bigr)
$$
and
$$
s^r_X= \bigl(\eta_X^\vee \otimes T_0T(\rev_{T(X)})(H^r_{X,T(X)^\vee})^{-1}\bigr)
\bigl(\rcoev_{T(X)} \otimes \id_{T(T(X)^\vee)}\bigr).
$$
Then the  left and right duals of any $T$\ti module $(M,r)\in \cc^T$ are defined
   by
   $$\leftidx{^\vee}{(}{} M,r)=(\leftidx{^\vee}{M}{}, s^l_M T(\leftidx{^\vee}{r}{}) \colon T(\leftidx{^\vee}{M}  ) \to
   \leftidx{^\vee}{M}),$$
   $$  (M,r)^\vee=(M^\vee, s^r_M T(r^\vee) \colon T(M^\vee) \to M^\vee).$$
Though we shall not need it, note that, conversely,   a bimonad $T$ on
$\cc$  such that  $\cc^T$ is rigid is a Hopf monad.

%
%

\section{Coends, centralizers, and free objects}\label{sect-big-HM}

We outline  the theory of coends \cite{ML1} and discuss   connections with Hopf monads and relative centers.

\subsection{Coends}\label{sect-coend}
Let $\cc$, $\dd$ be categories and  $F\co \dd^\opp \times \dd \to \cc$ be a functor.    A \emph{dinatural
transformation} from  $F$
to an object $A$ of $\cc$  is a family $d=\{d_Y \co F(Y,Y) \to A\}_{Y \in \dd}$  of morphisms in~$\cc$
such that for every morphism $f\co X \to Y$ in~$\dd$ (viewed also as a morphism $Y\to X$ in $\dd^\opp$), the following diagram commutes:
\begin{equation*}\label{dinadina}
\begin{split}
    \xymatrix@R=1cm @C=3cm { F(Y,X) \ar[r]^-{F(\id_Y,f)}  \ar[d]_-{F(f, \id_X)}
    & F(Y,Y) \ar[d]^-{d_Y} \\
    F(X ,X) \ar[r]_-{d_X}
    & A.
    }
\end{split}
\end{equation*}    The {\it composition}  of such a $d$ with a morphism $\phi\co A\to B$ in $\cc$
is the dinatural transformation $ \phi\circ d= \{\phi \circ d_X \co
F(Y,Y) \to B\}_{Y \in \dd}$ from $F$ to $B$. A \emph{coend} of~$F$
is a pair $(C \in \cc,\rho)$ where $\rho$ is  a
dinatural transformation  from $F$ to $C$ satisfying the
following universality condition:   every dinatural transformation
$d$ from $F$ to  an object   of $\cc$ is the composition of $\rho$
with a morphism   in $\cc$  uniquely determined by $d$. If $F$ has a
coend $(C,\rho)$, then  it is unique up to (unique) isomorphism. One
writes $ C=\int^{Y \in \dd}F(Y,Y)$. In particular, if the category $\cc$ is monoidal, then for any
  functors $F_1 \colon \dd^\opp \to \cc$, $F_2\colon \dd \to \cc$,   we write $  \int^{Y \in \dd}F_1(Y)
\otimes F_2(Y)$ for the coend (if it exists) of the functor $  \dd^\opp \times
\dd \to \cc$ defined  on objects and morphisms by  $ (X,Y)\mapsto F_1(X)\otimes F_2(Y)$.

The following lemma gives  a sufficient condition for the existence of coends.

\begin{lem}\label{lem-finite-semi--}
Let $\cc$ and $\dd$ be   $\kk$-additive   categories.   If $\dd$ is
  finite split semisimple, then
 any \kt linear functor $F\co \dd^\opp \times \dd \to \cc$   has a
coend.
\end{lem}
\begin{proof}
Pick a (finite)  representative set $I$ of simple objects of $\dd$
and  set $C=\oplus_{i \in I} F(i,i) \in \cc$. For each object $Y\in \dd$, set $\rho_Y=\sum_{\alpha } F(q_\alpha,p_\alpha)\co F(Y,Y) \to C$ where
$(p_\alpha:Y\to i_\alpha,q_\alpha:i_\alpha\to Y )_{\alpha  }$ is an arbitrary $I$-partition of~$Y$.
 It
is easy to check that  $\rho_Y$ does not depend on the choice of the $I$-partition and $(C, \rho=\{\rho_Y \}_{Y  })$ is a coend of $F$. Indeed, each
dinatural transformation $d$ from $F$ to  any $A\in \cc$ is the
composition of $\rho$ with    $ \oplus_{i \in I}\, d_i \co C \to A$.
\end{proof}

The next lemma is a partial inverse to Lemma~\ref{lem-finite-semi--}. It shows that a  finiteness condition   is necessary to for the existence of coends.

\begin{lem}\label{lem-finite-semi}
Let $\cc$ be a $\kk$-additive pivotal category whose Hom-spaces are
projective $\kk$-modules of finite rank, and let $\dd$ be a  split
semisimple full subcategory of~$\cc$.   If the coend $ \int^{Y \in
\dd} Y^* \otimes Y$ exists in $\cc$, then $\dd$ is   finite.
\end{lem}
\begin{proof}
Let $C=\int^{Y \in \dd} Y^* \otimes Y$  be the coend  of the functor
$F(X,Y)=X^*\otimes Y$ with universal dinatural  transformation
$\rho=\{\rho_Y\co Y^* \otimes Y \to C\}_{Y \in \dd}$. Let $I$ be a representative set of isomorphism  classes of simple objects of $\dd$ and let $J\subset I$ be an arbitrary  finite
subset of $  I$. Set
$A=\oplus_{i \in J} \, i^* \otimes i \in \cc$.  For any $Y \in \dd$,   set $$d_Y=\sum_{\alpha, i_\alpha\in J} q^\ast_\alpha \otimes p_\alpha:
  Y^* \otimes Y \to A
$$  where
$(p_\alpha:Y\to i_\alpha,q_\alpha:i_\alpha\to Y )_{\alpha  }$ is an  $I$-partition of~$Y$.
 It
is easy to check   that $d_Y$ does not depend on the choice of  the $I$-partition and that the  family $\{d_Y\}_Y$ is  a dinatural
transformation from  $F $ to~$A$. Therefore there is a  morphism
$p\co C \to A$ such that $d_Y=p\rho_Y$ for all $Y \in \dd$. Set
$q=\sum_{i \in J} \rho_i \co A \to C$. Then $pq=\sum_{i \in J}
p\rho_i=\sum_{i \in J} d_i=\id_A$. Thus, the composition with
$q$ induces a split injection $\Hom_\cc(\un,A) \to
\Hom_\cc(\un,C)$. Hence
\begin{align*}
\mathrm{card}(J) &=\sum_{i \in J} \mathrm{rank}_\kk\bigl(\Hom_\cc(i,i)\bigr) =\sum_{i \in J} \mathrm{rank}_\kk\bigl(\Hom_\cc(\un,i^* \otimes i)\bigr)\\
& = \mathrm{rank}_\kk\bigl(\Hom_\cc(\un,A)\bigr)   \leq  \mathrm{rank}_\kk\bigl(\Hom_\cc(\un,C)\bigr).
\end{align*}
This bound  implies the claim of the lemma.
\end{proof}

\subsection{Lift of coends}
Let $T=((T,T_2,T_0), \mu, \eta)$ be  a Hopf monad  on a rigid category $\cc$ and let  $\dd$ be a subcategory   of $\cc$ such that $T(\dd)
\subset \dd$.   The functor $T$ restricts   to a monad on $\dd$, also
denoted~$T$, and the corresponding category of modules $\dd^T$ is a subcategory of $\cc^T$.
The following  lemma allows us to lift from $\cc$ to $\cc^T$ the coends of certain functors $ \dd^\opp \times \dd \to \cc$ associated with  endofunctors of $\cc^T$.

\begin{lem}\label{lem-HM-coend} Let $Q$ be an endofunctor of $\cc^T$ such that
  there exists a coend $$C=\int^{Y \in \dd}
\leftidx{^\vee}{(}{}Q\rtimes T)(Y)\otimes Y\in \cc.$$   Then
there exists a coend $\int^{M \in \dd^T} \leftidx{^\vee}{Q}{}(M)
\otimes M \in \cc^T$ carried by  the forgetful functor $
\cc^T \to \cc$  to $C$. More precisely, if
$$\rho=\{\rho_Y\co \leftidx{^\vee}{(}{}Q\rtimes T)(Y)\otimes Y \to
C\}_{Y \in \dd}$$ is the  universal dinatural transformation of $C$,
then  there is a unique morphism $r\co T(C) \to C$ in $\cc$ such that for all
$Y \in \dd$,
\begin{equation*} rT(\rho_Y)=
\rho_{T(Y)}\bigr(\leftidx{^\vee}{Q}{}(\mu_Y)s^l_{Q\rtimes T(Y)}
T(\leftidx{^\vee}{a}{_Y}) \otimes \id_{T(Y)}\bigl)\, T_2\bigr(
 \leftidx{^\vee}{(}{}Q \rtimes T)(Y),Y\bigr),
\end{equation*}
 where $a_Y $ is the $T$-action of   the $T$-module $QF_T(Y)$.
Then $r$ is an action of $T$ on $C$ and $(C,r)=\int^{M \in \dd^T}
\leftidx{^\vee}{Q}{}(M) \otimes M$  with universal dinatural
transformation
$$\varrho=\{\varrho_{(N,s)}=\rho_N (\leftidx{^\vee}{Q}{}(s) \otimes \id_N) \co
\leftidx{^\vee}{Q}{}(N,s) \otimes (N,s) \to (C,r)\}_{(N,s)\in\dd^T}.$$
\end{lem}

\begin{proof}
This is a direct corollary of Lemma 3.9 and Proposition 3.10
of~\cite{BV3}.
\end{proof}

\subsection{Centralizers of endofunctors}\label{sect-coend++}  Let $\cc$ be
a  rigid category and
$\dd$ be a  subcategory of $\cc$. An endofunctor ${{E}}$ of $ \cc $ is
\emph{$\dd$-centralizable} if for each $X\in \cc$, the functor $
\dd^\opp \times \dd \to \cc$  defined by $(Y,Y')\mapsto
\leftidx{^\vee}{{{E}}}{}(Y) \otimes X \otimes Y'$ has a  coend
\begin{equation*}
{{Z}}_{{E}}^\dd(X)=\int^{Y \in \dd} \leftidx{^\vee}{{{E}}}{}(Y) \otimes X \otimes Y \in \cc.
\end{equation*}
The correspondence $X\mapsto {{Z}}_{{E}}^\dd(X)$  extends to a  functor
${{Z}}_{{E}}^\dd\co \cc\to \cc$,  called the {\it $\dd$-centralizer} of ${{E}}$,
so that the associated universal dinatural transformation
\begin{equation}\label{rhorho}
\rho_{X,Y}\co  \leftidx{^\vee}{{{E}}}{}(Y) \otimes X \otimes Y \to {{Z}}_{{E}}^\dd(X)
\end{equation}
is natural in $X\in\cc$ and dinatural in~$Y\in\dd$. For $\dd=\cc$,
the notion of a centralizer of an endofunctor was introduced in
\cite{BV3}.

When the
identity endofunctor $1_\cc$  of $\cc$ is $\dd$-centralizable, we
 say that $\cc$ is \emph{$\dd$-centralizable}. For example, the  endofunctor   $1_\cc$ is  $\dd$-centralizable  if the category $\dd$ is finite split semisimple, see  Lemma~\ref{lem-finite-semi--}. Moreover, any   (finite) representative set $I$ of simple objects of
$\dd$ determines a
$\dd$-centralizer $Z :\cc\to\cc$ of $1_\cc$. The functor $Z$ carries any $X \in \cc$ to $Z(X)=\oplus_{i \in I}\, i^*
\otimes X \otimes i$ and carries any morphism $f$ in $\cc$ to $Z(f)=\oplus_{i \in I}\, \id_{i^*}\otimes f \otimes \id_i$.

\subsection{Relative centers and free objects}\label{sect-coend++M} Let $\cc$ be
a  rigid category and $\dd$ be  a  rigid subcategory of
$\cc$, i.e., $\dd$ is a monoidal subcategory of $\cc$ stable under both left and right
dualities.
Suppose  that  $\cc$ is $\dd$-centralizable. We construct a Hopf monad  ${{Z}}={{Z}}_\dd$ on
$\cc$ such that    the
relative center $\zz(\cc;\dd)$ is monoidally isomorphic to the category   $\cc^{{{Z}}}$.

 Let ${{{{{Z}}}}} \co \cc\to \cc$ be a
$\dd$-centralizer of $1_\cc$ with   universal dinatural
transformation  $$\rho=\{\rho_{X,Y}\co \leftidx{^\vee}{Y}{} \otimes
X \otimes Y \to {{{{{Z}}}}}(X)\}_{X \in \cc, Y \in \dd}.$$ For $X \in \cc$
and $Y \in \dd$, set
$$
\partial_{X,Y}=(\id_Y \otimes \rho_{X,Y})(\lcoev_Y \otimes \id_{X \otimes Y})\co X \otimes Y \to Y \otimes {{{{{Z}}}}}(X).
 $$
We depict the morphism $\partial_{X,Y}$ as follows:
$$
\partial_{X,Y}=
\psfrag{Y}[Bc][Bc]{\scalebox{.7}{$Y$}}
\psfrag{X}[Bc][Bc]{\scalebox{.7}{$X$}}
\psfrag{T}[Bc][Bc]{\scalebox{.7}{$Y$}}
\psfrag{Z}[Bc][Bc]{\scalebox{.7}{${{{{{Z}}}}}(X)$}}
\rsdraw{.45}{.9}{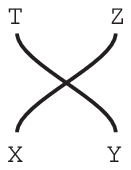}\;.
$$
For any $X,X_1,X_2 \in \cc$,   the parameter theorem and the Fubini
theorem for coends (see \cite{ML1}) imply the existence of  (unique)
morphisms
\begin{gather*}
\mu_X\co {{{{{Z}}}}}({{{{{Z}}}}} (X)) \to {{{{{Z}}}}}(X), \; {{{{{Z}}}}}_2(X_1,X_2) \co {{{{{Z}}}}}(X_1 \otimes X_2) \to {{{{{Z}}}}}(X_1) \otimes {{{{{Z}}}}}(X_2), \\
{{{{{Z}}}}}_0\co {{{{{Z}}}}}(\un) \to \un, \quad s^l_X\co {{{{{Z}}}}}(\leftidx{^\vee}{{{{{Z}}}}}(X)) \to \leftidx{^\vee}{X}{}, \quad s^r_X\co {{{{{Z}}}}}({{{{{Z}}}}}(X)^\vee)
\to X^\vee,
\end{gather*}
such that the equalities of morphisms shown in
Figure~\ref{fig-def-Z} hold for all $Y,Y_1,Y_2 \in \dd$,   where the
trivalent vertex in the third picture  stands   for
$\partial_{\un,Y}\co Y \to Y \otimes Z(\un)$.

\begin{figure}[t]
\begin{center}
\psfrag{A}[Bc][Bc]{\scalebox{.7}{$Y_1$}}
\psfrag{B}[Bc][Bc]{\scalebox{.7}{$Y_2$}}
\psfrag{C}[Bc][Bc]{\scalebox{.7}{$Y_1$}}
\psfrag{R}[Bc][Bc]{\scalebox{.7}{$Y_2$}}
\psfrag{X}[Bc][Bc]{\scalebox{.7}{$X$}}
\psfrag{U}[Bc][Bc]{\scalebox{.7}{$Y_1 \otimes Y_2$}}
\psfrag{L}[Bc][Bc]{\scalebox{.7}{${{{{{Z}}}}}(X)$}}
\psfrag{r}[Bc][Bc]{\scalebox{.9}{$\mu_X$}}
\rsdraw{.45}{.9}{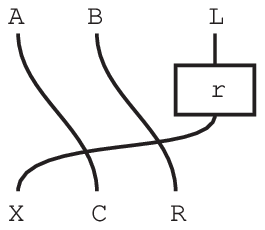} \, $=$ \, \rsdraw{.45}{.9}{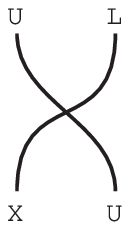},\\[1em]
\psfrag{A}[Bc][Bc]{\scalebox{.7}{${{{{{Z}}}}}(X_1)$}}
\psfrag{B}[Bc][Bc]{\scalebox{.7}{${{{{{Z}}}}}(X_2)$}}
\psfrag{C}[Bc][Bc]{\scalebox{.7}{$Y$}}
\psfrag{U}[Bc][Bc]{\scalebox{.7}{$Y$}}
\psfrag{X}[Bc][Bc]{\scalebox{.7}{$X_1 \otimes X_2$}}
\psfrag{E}[Bc][Bc]{\scalebox{.7}{$X_1$}}
\psfrag{H}[Bc][Bc]{\scalebox{.7}{$X_2$}}
\psfrag{T}[Bc][Bc]{\scalebox{.9}{${{{{{Z}}}}}_2(X_1,X_2)$}}
\psfrag{a}[Bc][Bc]{\scalebox{.7}{}}
\rsdraw{.45}{.9}{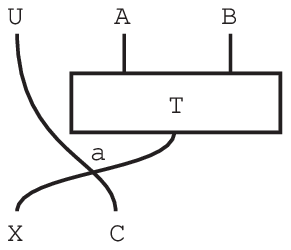} \, $=$ \, \rsdraw{.45}{.9}{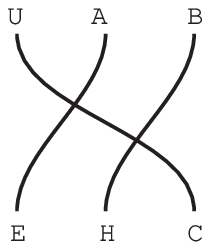},
\psfrag{a}[Bc][Bc]{\scalebox{.7}{}}
\psfrag{U}[Bc][Bc]{\scalebox{.7}{$Y$}}
\psfrag{X}[Bc][Bc]{\scalebox{.7}{$Y$}}
\psfrag{C}[Bc][Bc]{\scalebox{.7}{${{{{{Z}}}}}(\un)$}}
\psfrag{r}[Bc][Bc]{\scalebox{.9}{${{{{{Z}}}}}_0$}}
\qquad \quad \rsdraw{.45}{.9}{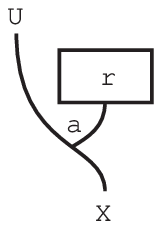} \, $=$ \, \rsdraw{.45}{.9}{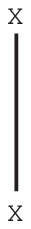}\;,\\[1em]
\psfrag{B}[Bc][Bc]{\scalebox{.7}{$Y$}}
\psfrag{X}[Bc][Bc]{\scalebox{.7}{$\leftidx{^\vee}{{{Z}}}{_\dd}(X)$}}
\psfrag{R}[Bc][Bc]{\scalebox{.7}{$Y$}}
\psfrag{e}[Bc][Bc]{\scalebox{.9}{$\lev_{{{{{{Z}}}}}(X)}$}}
\psfrag{a}[Bc][Bc]{\scalebox{.9}{$\rev_{Y}$}}
\psfrag{c}[Bc][Bc]{\scalebox{.9}{$\lcoev_{X}$}}
\psfrag{u}[Bc][Bc]{\scalebox{.9}{$\rcoev_{Y}$}}
\psfrag{L}[Bc][Bc]{\scalebox{.7}{$\leftidx{^\vee}{X}{}$}}
\psfrag{r}[Bc][Bc]{\scalebox{.9}{$s^l_X$}}
\rsdraw{.45}{.9}{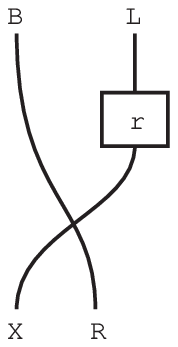}  $=$  \rsdraw{.45}{.9}{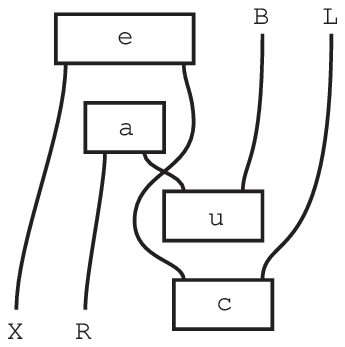}\;,
\qquad \psfrag{B}[Bc][Bc]{\scalebox{.7}{$Y$}}
\psfrag{X}[B][Bc]{\scalebox{.7}{${{{{{Z}}}}}(X)^\vee$}}
\psfrag{R}[Bc][Bc]{\scalebox{.7}{$Y$}}
\psfrag{L}[Bl][Bl]{\scalebox{.7}{$X^\vee$}}
\psfrag{r}[Bc][Bc]{\scalebox{.9}{$s^r_X$}}
\psfrag{e}[Bc][Bc]{\scalebox{.9}{$\rev_{{{{{{Z}}}}}(X)}$}}
\psfrag{a}[Bc][Bc]{\scalebox{.9}{$\lev_{Y}$}}
\psfrag{c}[Bc][Bc]{\scalebox{.9}{$\rcoev_{X}$}}
\psfrag{u}[Bc][Bc]{\scalebox{.9}{$\lcoev_{Y}$}}
\rsdraw{.45}{.9}{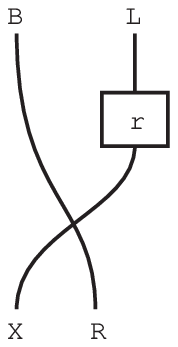}  $=$  \;\rsdraw{.45}{.9}{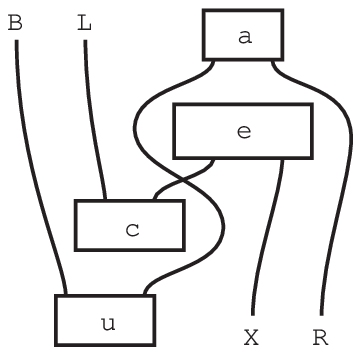}\;.
\end{center}
\caption{Structural morphisms of ${{{{{Z}}}}}={{Z}}_\dd$}
\label{fig-def-Z}
\end{figure}
\begin{lem}\label{lem-central-HM}
\begin{enumerate}
\labela
\item Let $\eta=\{\eta_X\}_{X\in \cc}$ where
$\eta_X=\partial_{X,\un}\co X \to {{{{{Z}}}}}(X)$ for all $X\in \cc$. Then ${{{{{Z}}}}}=\bigl(({{{{{Z}}}}},{{{{{Z}}}}}_2,{{{{{Z}}}}}_0),\mu,\eta\bigr)$ is a Hopf monad on $\cc$ with left antipode $s^l=\{s^l_X\}_{X\in \cc}$ and right antipode $s^r=\{s^r_X\}_{X\in \cc}$.
\item Let $\Psi\co \cc^{{{{{{Z}}}}}} \to \zz(\cc;\dd)$ be the functor carrying any object $(M,r)\in \cc^{{{{{{Z}}}}}} $ to
 $(M, \sigma^r)$ with $
\sigma^r=\{ \sigma^r_Y=(\id_Y \otimes r) \partial_{M,Y}:M\otimes Y\to Y\otimes M\}_{Y \in \dd}
$
and carrying any morphism   to itself. Then $\Psi$ is a strict monoidal isomorphism, and  the composition of $\Psi$ with the forgetful functor $\zz(\cc;\dd)\to \cc$ is equal to the
  forgetful functor  $U_{{{{{{Z}}}}}}\co \cc^{{{{{{Z}}}}}} \to \cc$.
  \item If $\cc$ is $\kk$-additive, then     the categories
$\cc^{{{Z}}}$ and  $\zz(\cc;\dd)$ are $\kk$-additive  and the functors ${{Z}} $ and  $\Psi $  are    $\kk$-linear.
\end{enumerate} \end{lem}
\begin{proof}
The proof of (a) and (b) is obtained from the proof of Theorems 5.6 and 5.12 in
\cite{BV3} by replacing $Y\in\cc$ with $Y\in\dd$ whenever necessary
and in particular by replacing  $ \int^{Y\in\cc}
\leftidx{^\vee}{Y}{} \otimes X \otimes Y$ with
$ \int^{Y\in\dd} \leftidx{^\vee}{Y}{} \otimes X \otimes Y$. Claim (c) is obvious.
\end{proof}

We use this lemma to define   free objects in $\zz(\cc;\dd)$. Recall from Section \ref{Monads} the free module functor $F_Z\co\cc \to \cc^Z$ is left adjoint to the forgetful functor $U_Z\cc^Z\to \cc$.
It is clear that the functor $\Psi F_Z\co\cc\to \zz(\cc;\dd)$ is left adjoint to the forgetful functor $\uu\co\zz(\cc;\dd) \to \cc$. An object of $\zz(\cc;\dd)$ is said to be {\it free} if it is isomorphic to $\Psi F_Z(X)$ for some $X\in \cc$.

\subsection{The case of $G$-centers}\label{Free objects} We shall  apply  Lemma~\ref{lem-central-HM} to study  the $G$-centers of $G$-graded categories. Consider a rigid $G$-graded category $\cc$ such that  $\cc$ is $\cc_1$-centralizable.    Lemma~\ref{lem-central-HM} provides an extension of any $\cc_1$-centralizer $Z:\cc\to \cc$  to a Hopf monad on $\cc$ and   a   $\kk$-linear strict monoidal isomorphism $\Psi\co \cc^{{{{{{Z}}}}}} \to \zz(\cc;\cc_1)=\zz_G(\cc)$.  We can always choose $Z$   so that $Z(\cc_\alpha)\subset \cc_\alpha$ for all $\alpha\in G$. Then $Z$ restricts to a monad on~$\cc_\alpha$, and the corresponding category of modules,  denoted   $ \cc^Z_\alpha$,   is a full subcategory of $\cc^Z$. This turns $\cc^Z$ into a $G$-graded category. It follows from the definitions that   $\Psi$
preserves the $G$-grading.   In the terminology above, an object of $\zz_G(\cc)$ is   \emph{free} if it is isomorphic to an object in the image of the functor
$ \Psi F_Z\co \cc \to \zz_G(\cc)$ left adjoint to the forgetful functor $\uu\co\zz_G(\cc) \to \cc$.

The $\cc_1$-centralizability condition on $\cc$ is satisfied, for example, if   $\cc_1$ is finite split semisimple.

%

\section{Proof of Lemmas~\ref*{lem-center-G-fusion} and~\ref*{thm-rel-center-semi}}\label{sect-bigproof-semiFULL}
 We begin with a fairly general lemma concerning morphisms
between indecomposable objects in abelian  categories. Next, we
formulate an extension of the   graphical calculus  allowing to
incorporate partitions of objects. Finally, we use these tools   and
the theory of Hopf monads  to prove
Lemmas~\ref{lem-center-G-fusion} and~\ref{thm-rel-center-semi}.

\subsection{Indecomposable objects in abelian categories}\label{sect-indeco}
We recall several standard definitions of the theory of categories. An object $X$ of an
additive category is \emph{indecomposable} if $X$ is non-zero and
whenever $X$ decomposes as $X=X_1 \oplus X_2$, we have
$X_1=\mathbf{0}$ or $X_2=\mathbf{0}$. For example,    all simple
objects of a $\kk$-additive category are indecomposable.   An
\emph{abelian category} is an additive category such that any
morphism $f$ has a kernel and a cokernel and
$\mathrm{coker}(\mathrm{ker}\, f) \simeq
\mathrm{ker}(\mathrm{coker}\, f)$.     A \emph{monomorphism} in a
category is a morphism $q\co X\to Y$ such that   any two morphisms
$f,g\co A\to X$ with $qf=qg$ must be equal. A \emph{retract} of a
morphism  $q\co X\to Y$ is a morphism $p \co Y \to X$ such that
$pq=\id_X$. Clearly, if $q$ has a retract, then $q$ is a
monomorphism.

\begin{lem}\label{lem-morphism-zero-or-iso}
Let $\aaa$ be an abelian  category in which any monomorphism has a
retract. Then any morphism between  indecomposable objects of
$\aaa$ is either zero or an isomorphism.
\end{lem}
\begin{proof}
Let  $f\co M \to P$ be a morphism in $\aaa$ where $M,P$ are
indecomposable objects. Let $q\co N \to M$ be the kernel of $f$.
 Then $q$ is a monomorphism in $\aaa$ and, by assumption, $q$
has a retract.  Since $\aaa$ is abelian, this implies that  $N$
is a direct summand of $M$. The object $M$ being indecomposable, we
 have $N=\mathbf{0}$ or   $  M=N\oplus \mathbf{0}   \simeq N  $. In the
latter case,   $q$ is an isomorphism, and so $f=0$ (because
$fq=0$). Assume that $N=\mathbf{0}$. Then $f$ is a monomorphism
(since it is a morphism with zero kernel in an abelian category). By
assumption, $f$ has a retract $g\co P\to M$. In particular
$gf=\id_{M}$ and so $e=\id_P-fg$ is an idempotent of $P$.  Since
$\aaa$ has split idempotents (because it is abelian), there exist an
object $A\in \aaa$ and morphisms $u\co A \to P$ and $v\co P \to A$
in $\aaa$ such that $vu=\id_{A}$ and $e=uv$. In particular,
$\id_P=fg+uv$ and so $P=M\oplus A$. Since $P$ is indecomposable,
  $M=\mathbf{0}$ or $A=\mathbf{0}$. If $M=\mathbf{0}$, then
$f=0$. If $A=\mathbf{0}$, then $uv=0$, and so $fg=\id_P$, which
implies that $f$ is an isomorphism (because $gf=\id_{M}$).
\end{proof}

\subsection{Extension of graphical calculus}\label{sect-graph-calc-in-pre-fusion}

 Let $\cc$ be a split semisimple pivotal category. Clearly, the
Hom spaces in   $\cc$ are   free $\kk$-modules of finite rank.
For $X\in\cc$ and a  simple object $i \in \cc$, the  modules
$\Hom_\cc(X,i)$ and $\Hom_\cc(i,X)$ have  the  same   rank
denoted $N^i_X $ and called the \emph{multiplicity number}.

An \emph{$i$-decomposition} of $X$ is a family of morphisms
$(p_\alpha\co X \to i, q_\alpha \co i \to X)_{\alpha \in A}$ such
that ${\mathrm{card}} (A)=N_X^i$  and $p_\alpha \,q_\beta =
\delta_{\alpha,\beta}\,\id_i$ for all $\alpha,\beta \in A$. Note
that if  $I$ is a representative set of simple objects of $\cc$ and
$( p_\alpha  , q_\alpha )_{\alpha \in \Lambda }$  is an
$I$-partition of $ X$  in the sense of Section~\ref{sec-ssss},
  then for each $i \in I$ the family $( p_\alpha  , q_\alpha
)_{\alpha \in \Lambda, \, i_\alpha=i }$ is an $i$-decomposition of
$X$. Conversely, the union of $i$-decompositions of $X$  over
all $i\in I$ is  an $I$-partition of $X$.

Let $i$ be a simple object of $\cc$ and  let $(p_\alpha\co
X \to i, q_\alpha \co i \to X)_{\alpha \in A}$ be an
$i$-decomposition of an object $X$ of $\cc$. Consider a sum
\begin{equation*}
\psfrag{p}[Bc][Bc]{\scalebox{.9}{$p_\alpha$}}
\psfrag{q}[Bc][Bc]{\scalebox{.9}{$q_\alpha$}}
\psfrag{X}[Bc][Bc]{\scalebox{.9}{$X$}}
\psfrag{i}[Br][Bc]{\scalebox{.9}{$i$}}
\sum_{\alpha \in A} \; \rsdraw{.45}{.9}{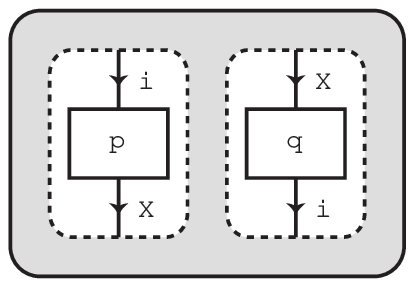}
\end{equation*}
where the gray area contains an oriented planar graph
whose edges and vertices are labeled by objects and morphisms of
$\cc$ not
 involving $p_\alpha, q_\alpha$. By the graphical calculus of Section~\ref{sect-penrose}, this sum represents a morphism in $\cc$. Note that the tensor $
\sum_{\alpha \in A} p_\alpha \otimes q_\alpha \in \Hom_\cc(X, i)
\otimes_\kk \Hom_\cc(i,X)$ does not depend on the choice of the
$i$-decomposition of $X$.   Therefore the sum above also does not
depend on this choice.    We graphically present this sum  by
\begin{equation*}
\psfrag{X}[Bc][Bc]{\scalebox{.9}{$X$}}
\psfrag{i}[Br][Bc]{\scalebox{.9}{$i$}}
\phantom{\sum_{\alpha \in A}} \;\rsdraw{.45}{.9}{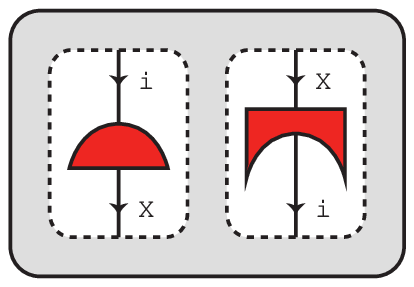}
\end{equation*}
where the   gray  area contains the same planar graph as before
and two curvilinear boxes are endowed  with one and the same color.
If several such pairs of boxes appear in a picture, they must
have different colors.



Note that   tensor products of objects may be depicted as
bunches of   strands. For example,
\begin{equation*}
\psfrag{Z}[Bl][Bl]{\scalebox{.9}{$X^* \otimes Y \otimes Z^*$}}
\psfrag{Y}[Bl][Bl]{\scalebox{.9}{$Y$}}
\psfrag{T}[Bl][Bl]{\scalebox{.9}{$Z$}}
\psfrag{X}[Br][Br]{\scalebox{.9}{$X$}}
\psfrag{i}[Bl][Bl]{\scalebox{.9}{$i$}}
\rsdraw{.45}{.9}{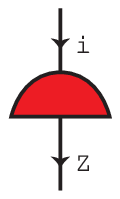} \qquad  \qquad \;= \; \rsdraw{.45}{.9}{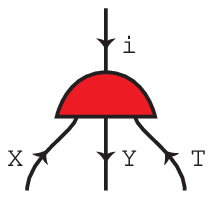} \qquad \, \text{and} \, \qquad
\rsdraw{.45}{.9}{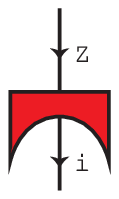} \qquad \qquad \; = \; \rsdraw{.45}{.9}{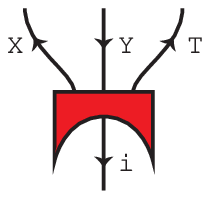}\;
\end{equation*}
where the equality sign means that the pictures represent the same
morphism of $\cc$.


To simplify the pictures, we will represent 
\begin{equation*}
\psfrag{e}[Br][Br]{\scalebox{.9}{$i$}}
\psfrag{R}[Br][Br]{\scalebox{.9}{$X$}}
\psfrag{X}[Bc][Bc]{\scalebox{.9}{$X$}}
\psfrag{i}[Bl][Bl]{\scalebox{.9}{$i$}}
\rsdraw{.45}{.9}{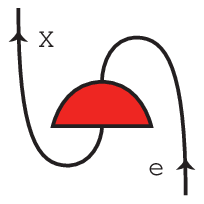} \quad \text{by} \quad \rsdraw{.45}{.9}{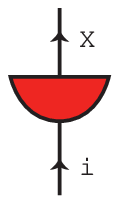}  \qquad \text{and} \qquad
\psfrag{e}[Br][Br]{\scalebox{.9}{$i$}}
\psfrag{R}[Bc][Bc]{\scalebox{.9}{$X$}}
\psfrag{X}[Bc][Bc]{\scalebox{.9}{$X$}}
\psfrag{i}[Bl][Bl]{\scalebox{.9}{$i$}}
\rsdraw{.45}{.9}{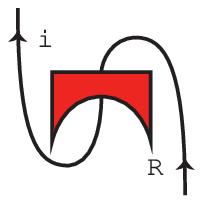} \quad \text{by} \quad \rsdraw{.45}{.9}{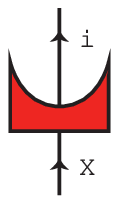}\;.
\end{equation*}

\subsection{Proof of  Lemma~\ref*{thm-rel-center-semi}}\label{sect-proof-thm-rel-center-semi-ppp}
 Fix a (finite) representative set $I$ of simple objects of
$\dd$ such that $\un \in I$. Consider the associated
$\dd$-centralizer $Z\co\cc\to\cc$ of $1_\cc$ as defined in Section
\ref{sect-coend++}. By Lemma~\ref{lem-central-HM}(a),  the functor
$Z $   extends to   a Hopf monad $((Z,Z_2,Z_0), \mu, \eta)$ on $\cc$
with structural morphisms   shown in Figure~\ref{fig-strucZ-fusion}.
\begin{figure}[t]
\begin{center}
          $\displaystyle Z_2(X,Y)=\sum_{i \in I}$\,
 \psfrag{i}[Br][Bc]{\scalebox{.85}{$i$}}
 \psfrag{X}[Bc][Bc]{\scalebox{.85}{$X$}}
 \psfrag{Y}[Bc][Bc]{\scalebox{.85}{$Y$}}
\rsdraw{.45}{.9}{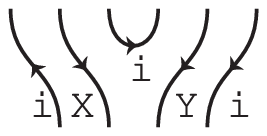}  $\co Z(X\otimes Y) \to Z(X)\otimes Z(Y)$, \\[1em] $\displaystyle Z_0=\sum_{i \in I}$\,
 \psfrag{i}[Br][Bc]{\scalebox{.85}{$i$}}
\rsdraw{.25}{.9}{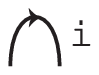} $\co Z(\un) \to \un$, \\[.3em]
$\displaystyle \mu_X=\!\!\sum_{i,j,k \in I}$\;
 \psfrag{i}[Br][Bc]{\scalebox{.85}{$i$}}
 \psfrag{j}[Br][Bc]{\scalebox{.85}{$j$}}
 \psfrag{k}[Bc][Bc]{\scalebox{.85}{$k$}}
 \psfrag{X}[Bc][Bc]{\scalebox{.85}{$X$}}
 \psfrag{p}[c][c]{\scalebox{.9}{$(q^\alpha_{i \otimes j})^*$}}
 \psfrag{q}[c][c]{\scalebox{.9}{$p^\alpha_{i \otimes j}$}}
\rsdraw{.5}{.9}{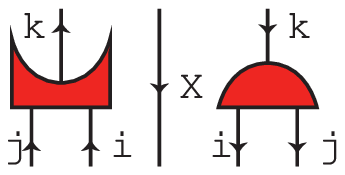}   $\co Z^2(X ) \to Z(X) $, \\[1em]
 $\displaystyle \eta_X=\id_X \co X \to X=\un^* \otimes X \otimes \un \hookrightarrow Z(X)$, \\[.8em]
$\displaystyle s^l_X=s^r_X=\sum_{i,j \in I}$
 \psfrag{i}[Br][Bc]{\scalebox{.85}{$i$}}
 \psfrag{u}[Bc][Bc]{\scalebox{.85}{$i^*$}}
 \psfrag{j}[Br][Bc]{\scalebox{.85}{$j$}}
 \psfrag{X}[Bc][Bc]{\scalebox{.85}{$X$}}
 \psfrag{p}[c][c]{\scalebox{.9}{$q^\alpha_{j^*}\psi_i^{-1}$}}
 \psfrag{q}[c][c]{\scalebox{.9}{$p^\alpha_{j^*}$}}
 \,\rsdraw{.35}{.9}{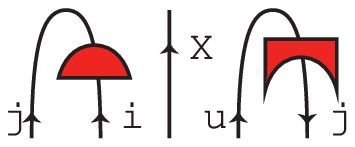} $\co Z(Z(X)^*) \to X^*$.
\end{center}
\caption{Structural morphisms of the Hopf monad $Z$}
\label{fig-strucZ-fusion}
\end{figure}

For $X \in \cc$, set
$$
\gamma_X=\sum_{i,j \in I} \frac{\dim_r(i)}{\dim(\dd)} \; \,
 \psfrag{i}[Bc][Bl]{\scalebox{.85}{$i$}}
 \psfrag{j}[Bc][Bl]{\scalebox{.85}{$j$}}
 \psfrag{X}[Bl][Bl]{\scalebox{.85}{$X$}}
 \rsdraw{.35}{.9}{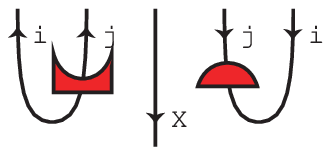} \co X \to Z^2(X)=Z(Z(X)).
$$
It is clear that $\gamma_X$ is natural in $X\in\cc$.
\begin{lem}\label{lem-Z-separable}
The natural transformation $\gamma=\{\gamma_X\co X \to Z^2(X)\}_{X\in\cc}$ satisfies
\begin{equation*}
 \mu_X \gamma_X= \eta_X  \quad \text{and} \quad Z(\mu_X) \gamma_{Z(X)}=\mu_{Z(X)} Z(\gamma_X)
\end{equation*} for any $X\in\cc$.
\end{lem}
In terminology of \cite[Section 6]{BV2},
Lemma~\ref{lem-Z-separable} may be reformulated by saying that the
Hopf monad $Z$ is separable.
\begin{proof}
For $X \in \cc$,
\begin{center}
$\displaystyle
\mu_X \gamma_X=\sum_{i,j,k \in I} \frac{\dim_r(i)}{\dim(\dd)} \; \;
 \psfrag{i}[Bl][Bl]{\scalebox{.85}{$i$}}
  \psfrag{k}[Bl][Bl]{\scalebox{.85}{$k$}}
 \psfrag{j}[Bc][Bl]{\scalebox{.85}{$j$}}
 \psfrag{X}[Bl][Bl]{\scalebox{.85}{$X$}}
 \rsdraw{.35}{.9}{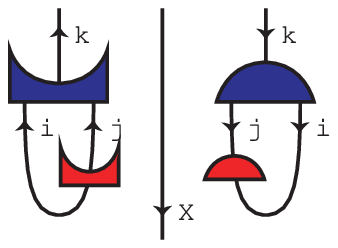} \; \;\overset{(i)}{=}\; \sum_{i \in I} \frac{\dim_r(i)}{\dim(\dd)} \; \;
 \psfrag{i}[Bl][Bl]{\scalebox{.85}{$i$}}
   \psfrag{k}[Bl][Bl]{\scalebox{.85}{$k$}}
 \psfrag{X}[Bl][Bl]{\scalebox{.85}{$X$}}
 \rsdraw{.35}{.9}{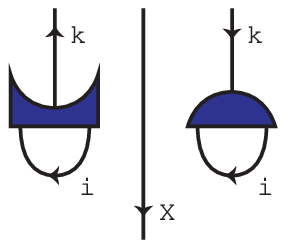}
$\\[.8em]
$\displaystyle
\overset{(ii)}{=} \sum_{i \in I} \frac{\dim_r(i)}{\dim(\dd)} \; \;
 \psfrag{i}[Bl][Bl]{\scalebox{.85}{$i$}}
 \psfrag{X}[Bl][Bl]{\scalebox{.85}{$X$}}
 \rsdraw{.35}{.9}{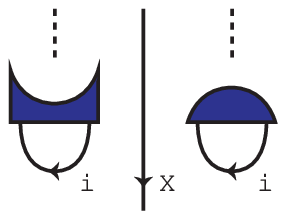}
 \;\;\overset{(iii)}{=}\; \sum_{i \in I} \frac{\dim_r(i)}{\dim(\dd)\dim_l(i)} \; \;
 \psfrag{i}[Bl][Bl]{\scalebox{.85}{$i$}}
 \psfrag{X}[Bl][Bl]{\scalebox{.85}{$X$}}
 \rsdraw{.35}{.9}{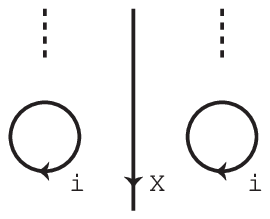}
$\\[.8em]
$\displaystyle
=\; \sum_{i \in I} \frac{\dim_r(i)\dim_l(i)}{\dim(\dd)} \; \;
 \psfrag{i}[Bl][Bl]{\scalebox{.85}{$i$}}
 \psfrag{X}[Bl][Bl]{\scalebox{.85}{$X$}}
 \rsdraw{.35}{.9}{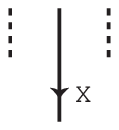}
 \;\;\overset{(iv)}{=}\;
 \psfrag{X}[Bl][Bl]{\scalebox{.85}{$X$}}
 \rsdraw{.35}{.9}{separable6}\;\; \overset{(v)}{=}\; \eta_X,
$
\end{center}
where the dotted lines represent $\id_\un$ and can be removed
without changing the morphisms (we depicted them in order to
remember which factor of~$Z(X)$ is concerned). In the above, the
equality $(i)$ follows from the fact that  there are no non-zero
morphisms between  non-isomorphic simple objects,
 and so
\begin{equation}\label{eq-derijd}
\psfrag{a}[Bc][Bc]{\scalebox{.9}{$i$}}
\psfrag{z}[Bc][Bc]{\scalebox{.9}{$j$}}
 \rsdraw{.45}{.9}{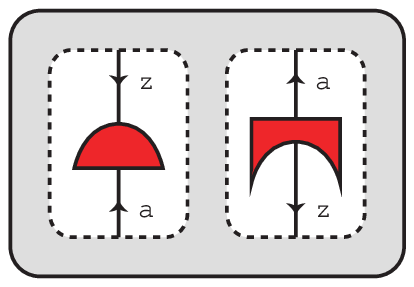} \; = \; \delta_{j,i^*} \; \psfrag{z}[Bc][Bc]{\scalebox{.9}{$i$}} \rsdraw{.45}{.9}{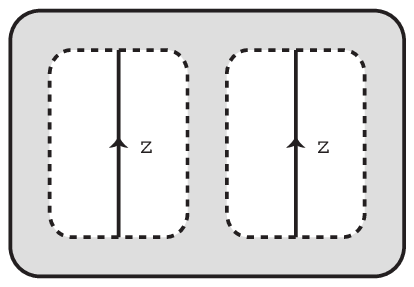}
\end{equation}
where $\delta_{j,i^*}=1$ if $j$ is isomorphic to $i^*$ and
$\delta_{j,i^*}=0$ otherwise. (If $j$ is isomorphic to
 $i^*$, then the right picture implicitly involves a box labeled with an isomorphism  $i^*\to j$ attached to the top of the left string and a box labeled with
 the inverse isomorphism
  $j\to i^*$  attached to the bottom of the right string.) The equality (ii) follows
from the fact that for $k \in I$, a morphism $\un \to k$ is zero
unless $k=\un$. The equality (iii) follows from the equality
\begin{equation}\label{eq-triangle-coev}
\psfrag{i}[Br][Br]{\scalebox{.9}{$i\,$}}
\psfrag{j}[Bl][Bl]{\scalebox{.9}{$i$}}
 \rsdraw{.45}{.9}{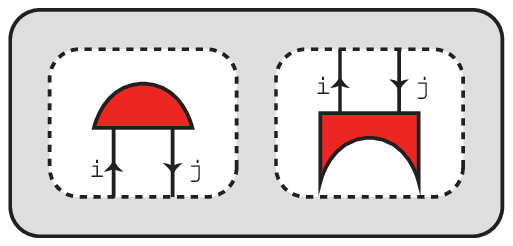}\, = \, \frac{1}{\dim_l (i)}\; \psfrag{i}[Bc][Bc]{\scalebox{.9}{$i$}}
\rsdraw{.45}{.9}{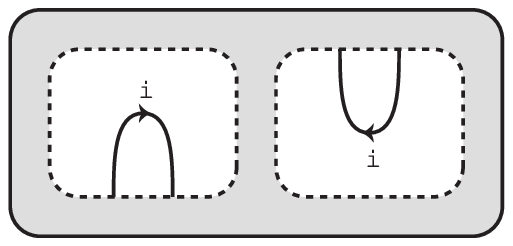} \;,
\end{equation}
which is a consequence of the duality and the fact that
$\Hom_\cc(i,i)=\kk$. Finally the equalities (iv) and (v) follow  from the definitions of $\dim(\dd)$ and $\eta_X$, respectively.

Let us prove the second equality of the lemma. We have:
\begin{center}
$\displaystyle
Z(\mu_X) \gamma_{Z(X)}=\sum_{i,j,k,n \in I} \frac{\dim_r(i)}{\dim(\dd)} \; \;
 \psfrag{i}[Bl][Bl]{\scalebox{.85}{$i$}}
  \psfrag{k}[Bl][Bl]{\scalebox{.85}{$k$}}
   \psfrag{n}[Bl][Bl]{\scalebox{.85}{$n$}}
 \psfrag{j}[Br][Bc]{\scalebox{.85}{$j$}}
 \psfrag{X}[Bl][Bl]{\scalebox{.85}{$X$}}
 \rsdraw{.35}{.9}{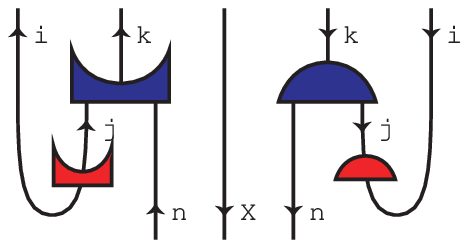}
$\\[.8em]
$\displaystyle
\overset{(i)}{=} \sum_{i,k,n \in I} \frac{\dim_r(i)}{\dim(\dd)} \; \;
 \psfrag{i}[Bl][Bl]{\scalebox{.85}{$i$}}
   \psfrag{k}[Bl][Bl]{\scalebox{.85}{$k$}}
   \psfrag{n}[Bl][Bl]{\scalebox{.85}{$n$}}
 \psfrag{X}[Bl][Bl]{\scalebox{.85}{$X$}}
 \rsdraw{.35}{.9}{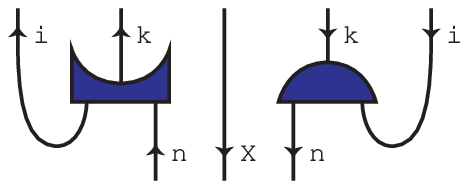}$\\[.8em]
$\displaystyle
\overset{(ii)}{=}\;
 \sum_{i,k,n \in I} \frac{\dim_l(k)}{\dim(\dd)} \; \;
 \psfrag{i}[Bl][Bl]{\scalebox{.85}{$i$}}
   \psfrag{k}[Bl][Bl]{\scalebox{.85}{$k$}}
   \psfrag{n}[Bl][Bl]{\scalebox{.85}{$n$}}
 \psfrag{X}[Bl][Bl]{\scalebox{.85}{$X$}}
 \rsdraw{.35}{.9}{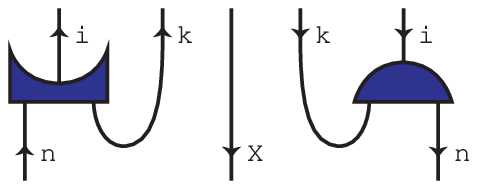}
$\\[.8em]
$\displaystyle
\overset{(iii)}{=}\;
 \sum_{i,j,k,n \in I} \frac{\dim_r(j)}{\dim(\dd)} \; \;
 \psfrag{i}[Bl][Bl]{\scalebox{.85}{$i$}}
   \psfrag{k}[Bl][Bl]{\scalebox{.85}{$k$}}
   \psfrag{n}[Bl][Bl]{\scalebox{.85}{$n$}}
    \psfrag{j}[Bl][Bl]{\scalebox{.85}{$j$}}
 \psfrag{X}[Bl][Bl]{\scalebox{.85}{$X$}}
 \rsdraw{.35}{.9}{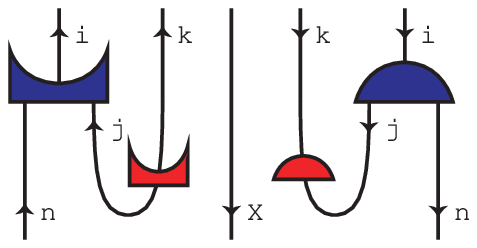} \; \; = \; \mu_{Z(X)}Z(\gamma_X).
$
\end{center}
In the above, the equality $(i)$ follows from \eqref{eq-derijd}, $(iii)$ follows from \eqref{eq-derijd} and
 the fact that if $j \simeq k^*$ then $\dim_l(k)=\dim_l(j^*)=\dim_r(j)$, and $(ii)$ follows from the following equality:
 for  any  $i,k,n \in I$,
\begin{equation}\label{eq-slidepict}
 \psfrag{i}[Bl][Bl]{\scalebox{.85}{$i$}}
   \psfrag{k}[Br][Br]{\scalebox{.85}{$k$}}
   \psfrag{n}[Bl][Bl]{\scalebox{.85}{$n$}}
 \rsdraw{.35}{.9}{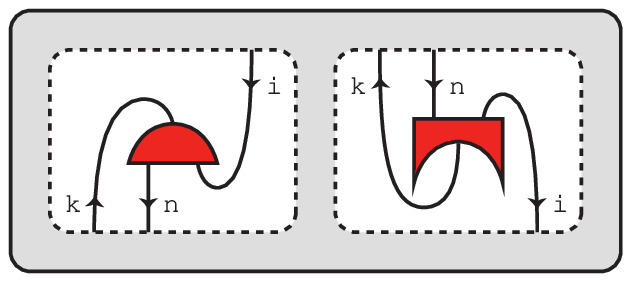}=\frac{\dim_l(k)}{\dim_r(i)} \;\rsdraw{.35}{.9}{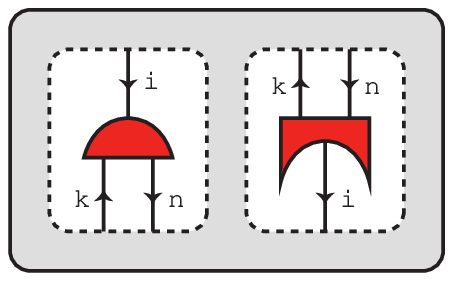}\;.
\end{equation}
It remains to prove \eqref{eq-slidepict}. Let $(p_\alpha\co n \otimes i^* \to k,q_\alpha \co k \to n \otimes i^*)_{\alpha \in A}$ be a $k$-decomposition of $n \otimes i^*$. For $\alpha \in A$, set
$$
P_\alpha=
 \psfrag{i}[Bl][Bl]{\scalebox{.9}{$i$}}
   \psfrag{k}[Br][Br]{\scalebox{.9}{$k$}}
   \psfrag{n}[Bl][Bl]{\scalebox{.9}{$n$}}
    \psfrag{p}[Bc][Bc]{\scalebox{1}{$p_\alpha$}}
  \rsdraw{.35}{.9}{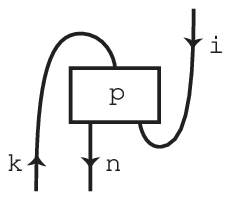}
\quad \text{and} \quad
Q_\alpha=
 \psfrag{i}[Bl][Bl]{\scalebox{.9}{$i$}}
   \psfrag{k}[Br][Br]{\scalebox{.9}{$k$}}
   \psfrag{n}[Bl][Bl]{\scalebox{.9}{$n$}}
    \psfrag{q}[Bc][Bc]{\scalebox{1}{$q_\alpha$}}
  \rsdraw{.35}{.9}{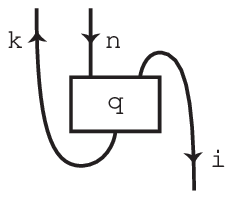}  \;.
$$
For $\alpha,\beta \in A$,
$$
P_\alpha Q_\beta=\frac{\tr_r(P_\alpha Q_\beta)}{\dim_r(i)} \id_i=\frac{\tr_l(p_\alpha q_\beta)}{\dim_r(i)} \id_i
= \frac{\tr_l(\delta_{\alpha,\beta} \, \id_k)}{\dim_r(i)} \id_i= \delta_{\alpha,\beta}\, \frac{\dim_l(k)}{\dim_r(i)} \id_i.
$$
Therefore, since ${\rm  {card}} (A)=N_{n \otimes i^*}^k=N_{k^* \otimes n}^i$, we obtain that the family
$$
\left ( \frac{\dim_r(i)}{\dim_l(k)}P_\alpha,Q_\alpha \right)_{\alpha \in A}
$$
is  an  $i$-decomposition of $k^* \otimes n$, from which we
deduce \eqref{eq-slidepict} and the lemma.
\end{proof}

\begin{lem}\label{lem-mono-direct-sum}
 Any monomorphism in   the category  $\cc^Z$   of $Z$-modules   has a retract.
\end{lem}
\begin{proof}
Let $q\co (N,s) \to (M,r)$ be a monomorphism   in $\cc^Z$. The
forgetful functor  $  \cc^Z \to \cc$ is a right adjoint of a
functor $\cc \to \cc^Z$ and so carries monomorphisms in $\cc^Z$ to
monomorphisms in $\cc$. Thus, $q\co N \to M$ is a monomorphism
in~$\cc$. Expanding $N$ and $M$ as direct sums of simple objects of
$\cc$ we can represent $q$ by a matrix over the field $\kk$. Using
 standard arguments of linear algebra, we conclude that
there is a morphism $v \co M \to N$ in $\cc$ such that $vq=\id_N$.
Set $p=sZ(vr)\gamma_M\co M \to N$ where $\gamma$ is the natural
transformation of Lemma~\ref{lem-Z-separable}. Then
\begin{align*}
sZ(p)&=sZ(s)Z^2(vr)Z(\gamma_M)\\
&=s\mu_NZ^2(vr)Z(\gamma_M) \quad \text{since $s$ is a $Z$-action,}\\
&=sZ(vr)\mu_{Z(M)}Z(\gamma_M)\quad \text{by the naturality of $\mu$,}\\
&=sZ(vr\mu_M) \gamma_{Z(M)} \quad \text{by Lemma~\ref{lem-Z-separable},}\\
&=sZ(vr)Z^2(r) \gamma_{Z(M)}\quad \text{since $r$ is a $Z$-action,}\\
&=sZ(vr)\gamma_M r\quad \text{by the naturality of $\gamma$,}\\
&=pr.
\end{align*}
Thus, $p$ is a morphism $(M,r)\to (N,s)$ in $\cc^Z$. Also,
\begin{align*}
pq&=sZ(vr)\gamma_Mq\\
&=sZ(vrZ(q))\gamma_N \quad \text{by the naturality of $\gamma$,}\\
&=sZ(vqs)\gamma_N \quad \text{since $q$ is $Z$-linear,}\\
&=sZ(s)\gamma_N \quad \text{since $vq=\id_N$,}\\
&=s\mu_N\gamma_N \quad \text{since $s$ is a $Z$-action,}\\
&=s\eta_N \quad \text{by Lemma~\ref{lem-Z-separable},}\\
&=\id_N \quad \text{since $s$ is a $Z$-action.}
\end{align*}
Hence, $p$ is a retract of $q$.
\end{proof}

We can now complete the proof of
Lemma~\ref*{thm-rel-center-semi}. Since $\cc$ is a split semisimple
category over $\kk$   which is assumed to be a field,
$\cc$ is an abelian $\kk$-additive category with finite-dimensional
Hom-spaces. Since the Hopf monad $Z$ is $\kk$-linear and preserves
cokernels (because $Z$ has a right adjoint by
\cite[Corollary~3.12]{BV2}), we deduce that $\cc^Z$ is an abelian
$\kk$-additive category with finite-dimensional Hom-spaces.

Combining Lemmas~\ref{lem-morphism-zero-or-iso}
and~\ref{lem-mono-direct-sum}, we obtain that the Hom-spaces between
non-isomorphic indecomposable $Z$-modules are zero, and
  the  algebra of endomorphisms of an indecomposable $Z$-module is a  finite-dimensional division $\kk$-algebra.
 Since the field $\kk$ is algebraically closed, such an algebra is isomorphic to
 $\kk$. Thus, all indecomposable $Z$-modules are simple.  The finite-dimensionality of the   End-spaces   in
$\cc^Z$  implies that any $Z$-module is  a finite direct sum of
indecomposable $Z$-modules. Hence,   $\cc^Z$ is split semisimple. By
Lemma~\ref{lem-central-HM}(b),(c),
 the  $\kk$-additive categories  $\zz(\cc;\dd)$
and $ \cc^Z$ are isomorphic. Therefore, $\zz(\cc;\dd)$ is split
semisimple. This concludes the proof of
Lemma~\ref{thm-rel-center-semi}.

\subsection{Proof of Lemma~\ref*{lem-center-G-fusion}}\label{sect-bigproof-semi-e}
Since $\cc$ is a $G$-fusion category, it is split semisimple, its unit object $\un$ is simple,  and
each $\cc_\alpha$ with $\alpha\in G$ is finite  split semisimple.
   Lemma
~\ref{thm-rel-center-semi} applied to $\dd=\cc_1$ yields that
$\zz_G(\cc)=\zz(\cc;\cc_1)$ is split semisimple. The unit object of
$\zz_G(\cc)$ is simple because   $\un_\cc$ is simple.
   It remains to prove that the
set of isomorphism classes of simple objects of $\zz_\alpha(G)$ is
finite for all   $\alpha \in G$.
In the  notation of Section~\ref{Free objects}, it is enough to prove that
the set of isomorphism classes of simple objects of the category $\cc^Z_{\alpha}$
is finite.    By Lemma~\ref{lem-finite-semi}, it suffices to prove that
the coend $$\int^{(N,s)\in \cc^Z_{\alpha}} (N,s)^* \otimes (N,s)$$
exists in $\cc^Z$.   Lemma~\ref{lem-HM-coend} and the equality
$1_{\cc^Z}\rtimes Z=Z$ imply that it is enough to establish the
existence of a coend $\int^{Y\in\cc_\alpha} Z(Y)^* \otimes Y$
  in $\cc$.
  The latter  follows from Lemma~\ref{lem-finite-semi--} because $\cc_\alpha$ is a finite
  split
semisimple subcategory  of $\cc$.

\section{Crossing and $G$-braiding via   free functors}\label{sect-free-functor-big}

In this  section,   $\cc$ is a rigid $G$-graded category which is
\emph{$G$-centralizable} in the sense that $\cc$ is
$\cc_\alpha$-centralizable for all $\alpha \in G$  (see
Section~\ref{sect-coend}). By Section~\ref{Free objects},   the
forgetful functor $ \zz_G(\cc) \to \cc$ has a left adjoint $\ff\co
\cc \to \zz_G(\cc)$, and the   objects of $\zz_G(\cc)$ isomorphic to
  objects in the image of  $\ff$ are said to be free. We introduce here a larger class of $G$-free objects of
$\zz_G(\cc)$ and  compute for them the crossing and the $G$-braiding
in $\zz_G(\cc)$.

 \subsection{Free functors}\label{sect-def-free-hb}
By assumption,  for all $\alpha \in G$ and $X \in \cc$, the coend $Z_\alpha(X)=\int^{Y \in \cc_\alpha} \leftidx{^\vee}{}Y \otimes X \otimes Y$ exists in $\cc$.
Let
$Z_\alpha=Z^{\cc_\alpha}_{1_\cc}\colon \cc\to \cc$ be a $\cc_\alpha$-centralizer of $1_\cc$
with  universal dinatural transformation
\begin{equation}\label{eq-def-rho-of-Z}
\rho^\alpha=\{\rho^\alpha_{X,Y}\co \leftidx{^\vee}{Y}{} \otimes X
\otimes Y \to Z_\alpha(X)\}_{X\in \cc, Y \in \cc_\alpha}.
\end{equation}
If $X \in \cc_\beta $ with $\beta \in G$, then we
always   choose $Z_\alpha(X)$ in $ \cc_{\alpha^{-1} \beta
\alpha}$.

For any $X\in
\cc$ and $Y \in \cc_\alpha$, set
\begin{equation}\label{eq-def-partial}
\partial_{X,Y}^\alpha=(\id_Y \otimes \rho^\alpha_{X,Y})(\lcoev_Y \otimes \id_{X \otimes Y})\co X \otimes Y \to Y \otimes Z_\alpha(X),
\end{equation}
which we depict as
$$
\partial_{X,Y}^\alpha=
\psfrag{Y}[Bc][Bc]{\scalebox{.7}{$Y$}}
\psfrag{a}[Bc][Bc]{\scalebox{.7}{$\alpha$}}
\psfrag{X}[Bc][Bc]{\scalebox{.7}{$X$}}
\psfrag{T}[Bc][Bc]{\scalebox{.7}{$Y$}}
\psfrag{Z}[Bc][Bc]{\scalebox{.7}{$Z_\alpha(X)$}}
\rsdraw{.45}{.9}{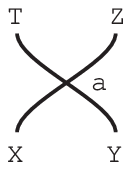}\;.
$$
For any $ X_1,X_2 \in \cc$,   the parameter theorem and the Fubini
theorem for coends (see \cite{ML1}) imply the existence of  unique
morphisms
$$  (Z_\alpha)_2(X_1,X_2) \co Z_\alpha(X_1 \otimes X_2) \to Z_\alpha(X_1) \otimes
Z_\alpha(X_2), \quad (Z_\alpha)_0\co Z_\alpha(\un) \to \un$$ such
that the first two equalities of Figure~\ref{fig-def-Zgen} are
satisfied for all $Y  \in \cc_\alpha$.  Similarly, for any
$\alpha,\beta \in G$ and $X\in \cc$ there is a unique morphism $
Z_2(\alpha,\beta)_X\co Z_\alpha Z_\beta(X) \to Z_{\beta\alpha}(X)$
%
such that the third equality of Figure~\ref{fig-def-Zgen} is
satisfied for all $Y  \in \cc_\alpha$ and $Y' \in \cc_\beta$.
Finally, for all $X\in \cc$, set $  (Z_0)_X=\partial^1_{X,\un}\co X
\to Z_1(X)$.
\begin{figure}[t]
\begin{center}
\psfrag{A}[Bc][Bc]{\scalebox{.7}{$Z_\alpha(X_1)$}}
\psfrag{B}[Bc][Bc]{\scalebox{.7}{$Z_\alpha(X_2)$}}
\psfrag{C}[Bc][Bc]{\scalebox{.7}{$Y$}}
\psfrag{U}[Bc][Bc]{\scalebox{.7}{$Y$}}
\psfrag{X}[Bc][Bc]{\scalebox{.7}{$X_1 \otimes X_2$}}
\psfrag{E}[Bc][Bc]{\scalebox{.7}{$X_1$}}
\psfrag{H}[Bc][Bc]{\scalebox{.7}{$X_2$}}
\psfrag{T}[Bc][Bc]{\scalebox{.9}{$(Z_\alpha)_2(X_1,X_2)$}}
\psfrag{a}[Bc][Bc]{\scalebox{.7}{$\alpha$}}
\rsdraw{.45}{.9}{defdela} \, $=$ \, \rsdraw{.45}{.9}{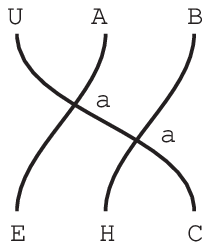},
\psfrag{U}[Bc][Bc]{\scalebox{.7}{$Y$}}
\psfrag{X}[Bc][Bc]{\scalebox{.7}{$Y$}}
\psfrag{C}[Bc][Bc]{\scalebox{.7}{$Z_1(\un)$}}
\psfrag{r}[Bc][Bc]{\scalebox{.9}{$(Z_\alpha)_0$}}
\psfrag{a}[Bc][Bc]{\scalebox{.7}{$\alpha$}}
\qquad \quad \rsdraw{.45}{.9}{defZ0a} \, $=$ \, \rsdraw{.45}{.9}{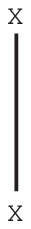}\;,\\[1em]
\psfrag{A}[Bc][Bc]{\scalebox{.7}{$Y'$}}
\psfrag{B}[Bc][Bc]{\scalebox{.7}{$Y$}}
\psfrag{C}[Bc][Bc]{\scalebox{.7}{$Y'$}}
\psfrag{R}[Bc][Bc]{\scalebox{.7}{$Y$}}
\psfrag{X}[Bc][Bc]{\scalebox{.7}{$X$}}
\psfrag{U}[Bc][Bc]{\scalebox{.7}{$Y' \otimes Y$}}
\psfrag{L}[Bc][Bc]{\scalebox{.7}{$Z_{\beta\alpha}(X)$}}
\psfrag{r}[Bc][Bc]{\scalebox{.9}{$Z_2(\alpha,\beta)_X$}}
\psfrag{a}[Bc][Bc]{\scalebox{.7}{$\alpha$}}
\psfrag{b}[Bc][Bc]{\scalebox{.7}{$\beta$}}
\psfrag{k}[Bl][Bl]{\scalebox{.7}{$\beta\alpha$}}
\rsdraw{.45}{.9}{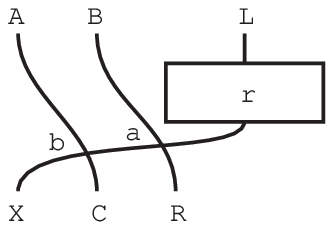} \, $=$ \, \rsdraw{.45}{.9}{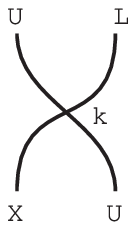}\, .
\end{center}
\caption{The structural morphisms of $Z$}
\label{fig-def-Zgen}
\end{figure}
\begin{lem}\label{lem-Z-def-global}
The  endofunctor
$Z_\alpha=(Z_\alpha,(Z_\alpha)_2,(Z_\alpha)_0)$ of $\cc$ is
comonoidal for all $\alpha\in G$. The formula $\alpha \mapsto
 Z_\alpha $ defines a monoidal
 functor $$Z=(Z,Z_2,Z_0)\co \overline{G} \to
\Endcm(\cc)$$ such that $Z_\alpha(\cc_\beta)\subset\cc_{\alpha^{-1}\beta\alpha}$ for all $\alpha,\beta \in G$.
\end{lem}
\begin{proof}
The proof is obtained from that of Theorems 5.6   in
\cite{BV3}, by replacing $Y\in\cc$ with $Y\in \cc_\alpha$   and in particular by replacing    $ \int^{Y\in\cc}
 $ with
$ \int^{Y\in\cc_\alpha} $.
\end{proof}

For $\alpha=1$, the functor $Z_\alpha=Z_1\co\cc\to \cc$, endowed with the product
$\mu=Z_2(1,1)$ and unit $\eta=Z_0$, is nothing but the Hopf monad
  produced by Lemma~\ref{lem-central-HM}(a) for
$\dd=\cc_1$. Let   $F_1=F_{Z_1}\co \cc \to \cc^{Z_1}$ be the free module functor, see Section~\ref{Monads}. Recall that   $F_1(X)=(Z_1(X),Z_2(1,1)_X)$ for any object $X $ of $ \cc$ and $F_1(f)=Z_1(f)$ for any morphism $f$ in $\cc$. The functor $F_1$ is comonoidal with comonoidal  structure    $(F_1)_2= (Z_1)_2$ and $(F_1)_0= (Z_1)_0$.
The forgetful functor $U_1=U_{Z_1}\co \cc^{Z_1} \to \cc$ is  strict monoidal and therefore is comonoidal in the canonical way. One can check that $U_1F_1=Z_1$ as comonoidal functors.
Our immediate aim is to introduce $\alpha$-analogues of $F_1$ for all $\alpha\in G$.

Pick any $\alpha \in G$. For all $X,Y \in \cc$ and any morphism $f$ in $\cc$, set
\begin{align*}
&F_\alpha(X)=(Z_\alpha(X),Z_2(1,\alpha)_X), & F_\alpha(f)=Z_\alpha(f),\\
&(F_\alpha)_2(X,Y)= (Z_\alpha)_2(X,Y), & (F_\alpha)_0= (Z_\alpha)_0.
\end{align*}
\begin{lem}\label{lem-Za-lifts}
$F_\alpha=(F_\alpha,(F_\alpha)_2 ,(F_\alpha)_0)\co \cc \to \cc^{Z_1}$ is a comonoidal functor such that $U_1F_\alpha=Z_\alpha$ as comonoidal functors.
\end{lem}
\begin{proof}
The monoidality of $Z$ implies that $Z_2(1,\alpha)_X (Z_0)_{Z_\alpha(X)}=\id_{Z_\alpha(X)}$ and
$$Z_2(1,\alpha)_XZ_1(Z_2(1,\alpha)_X)=Z_2(1,\alpha)_X Z_2(1,1)_{Z_\alpha(X)}.$$ 
Therefore $F_\alpha(X) \in \cc^{Z_1}$. The naturality of $Z_2(1,\alpha)$ implies that 
$$
Z_\alpha(f) Z_2(1,\alpha)_X =Z_2(1,\alpha)_YZ_1(Z_\alpha(f) ).
$$
This means that $Z_\alpha(f)$ is $Z_1$-linear. Hence $F_\alpha$ is a well-defined functor. Now the comonoidality of the natural transformation $Z_2(1,\alpha)$ gives
\begin{gather*}
Z_\alpha(X,Y)Z_2(1,\alpha)_{X \otimes Y}\\
=(Z_2(1,\alpha)_X \otimes Z_2(1,\alpha)_Y)(Z_1)_2(Z_\alpha(X),Z_\alpha(Y))Z_1((Z_\alpha)_2(X,Y))
\end{gather*}
and
$
(Z_\alpha)_0Z_2(1,\alpha)_\un=(Z_1)_0Z_1((Z_\alpha)_0).
$
Thus $(Z_\alpha)_2(X,Y)$  and $(Z_\alpha)_0$ are $Z_1$-linear. Hence $F_\alpha$ is comonoidal.
Clearly $U_1F_\alpha=Z_\alpha$ as comonoidal functors because $U_1$ is strict monoidal.
\end{proof}

By Section \ref{Free objects} (where $Z$ should be replaced with $Z_1$), the category
$\cc^{Z_1} $ is $G$-graded with $\cc_\alpha^{Z_1}=(\cc_\alpha)^{Z_1}$ for all $\alpha\in G$, and    we have
  a $\kk$-linear strict monoidal isomorphism of  $G$-graded categories $\Psi\co \cc^{Z_1}
\to  \zz_G(\cc)$.   For $\alpha \in G$, set $\ff_\alpha=\Psi F_\alpha\co \cc \to \zz_G(\cc)$.
By definition, for   $X \in \cc$,
$$
\ff_\alpha(X)=(Z_\alpha(X),\sigma^\alpha_X) \quad \text{with} \quad
\psfrag{Z}[Bl][Bl]{\scalebox{.8}{$Z_\alpha(X)$}}
\psfrag{T}[Br][Br]{\scalebox{.8}{$Z_\alpha(X)$}}
\psfrag{Y}[Bl][Bl]{\scalebox{.8}{$Y$}}
\psfrag{S}[Br][Br]{\scalebox{.8}{$Y$}}
\psfrag{r}[Bc][Bc]{\scalebox{.9}{$Z_2(1,\alpha)_X$}}
\psfrag{a}[Bc][Bc]{\scalebox{.8}{$1$}}
\sigma^\alpha_{X,Y}=\rsdraw{.45}{.9}{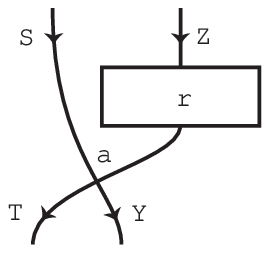} \quad \text{for all $Y \in \cc_1$.}
$$
Clearly,  $\ff_\alpha(\cc_\beta) \subset \zz_\beta(\cc)$ for all $\beta \in G$. By Lemma~\ref{lem-Za-lifts}, $\ff_\alpha$ is a comonoidal functor and $\uu \ff_\alpha=Z_\alpha$ as comonoidal functors, where $\uu\co\zz_G(\cc) \to \cc$ is the (strict monoidal) forgetful functor.

We call $\ff_\alpha\co \cc\to \zz_G(\cc) $ the \emph{$\alpha$-free
functor}, and we call the objects of $\zz_G(\cc)$ isomorphic to
objects in the image of  $\ff_\alpha$ \emph{$\alpha$-free}. For
$\alpha=1$, this definition is equivalent to the definition of free
objects given in Section~\ref{Free objects}. Collectively, the
$\alpha$-free objects of $\zz_G(\cc)$ with $\alpha\in G$ are said to
be \emph{$G$-free}.

 The following lemma will be used in Section \ref{sect-objects-Cab}.

\begin{lem}\label{lem-Z2-comono}
 $Z_2(\alpha,\beta)$ is a comonoidal natural transformation from $\ff_\alpha Z_\beta$ to $\ff_{\beta\alpha}$ for all $\alpha,\beta  \in G$.
\end{lem}
\begin{proof}
Since the isomorphism $\Psi\co \cc^{Z_1}
\to \zz_G(\cc)$   is  strict monoidal   and acts as the identity on morphisms, it is enough to prove that $Z_2(\alpha,\beta)$ is a comonoidal natural transformation from $F_\alpha Z_\beta$ to $F_{\beta\alpha}$.
Since $U_1 F_\gamma=Z_\gamma$ as comonoidal functors for all $\gamma \in G$ and $Z_2(\alpha,\beta)$ is a comonoidal natural transformation from $Z_\alpha Z_\beta$ to $Z_{\beta\alpha}$, we only need to check that for $X\in\cc$, $Z_2(\alpha,\beta)_X$ is a morphism in $C^{Z_1}$ from $F_\alpha Z_\beta(X)$ to $F_{\beta\alpha}(X)$, i.e., that
$$
Z_2(\alpha,\beta)_XZ_2(1,\alpha)_{Z_\beta(X)}=Z_2(1,\beta\alpha)_X Z_1(Z_2(\alpha,\beta)_X).
$$
This equality holds by the monoidality of $Z$ (see Lemma~\ref{lem-Z-def-global}).
\end{proof}

\subsection{Computations on $G$-free objects}
Assume that, as above,  $\cc$ is a $G$-centralizable rigid  $G$-graded category, and assume additionally  that $\cc$ is pivotal and non-singular. Then $\zz_G(\cc)$ is $G$-braided by Theorem~\ref{thm-G-center}.  The next theorem computes the action of the   crossing $\varphi \co
\overline{G} \to \Aut(\zz_G(\cc))$ of $\zz_G(\cc)$ on the $G$-free objects.
Note that each   auto-equivalence $\varphi_\alpha$ of $\zz_G(\cc)$, being strong monoidal, can be seen as a strong comonoidal endofunctor $(\varphi_\alpha,(\varphi_\alpha)_2^{-1},(\varphi_\alpha)_0^{-1})$ of $\zz_G(\cc)$.

\begin{thm}\label{lemu}
For $\alpha,\beta \in G$, there is a canonical comonoidal isomorphism
$$\omega^{\alpha,\beta}=\{\omega^{\alpha,\beta}_X \co \varphi_\alpha \ff_\beta(X) \to \ff_{\beta\alpha}(X)\}_{X  \in \cc}$$
from $\varphi_\alpha \ff_\beta$ to $\ff_{\beta\alpha}$. Moreover, the family $\omega=\{\omega^{\alpha,\beta}\}_{\alpha,\beta \in G}$ is compatible with the monoidal structure of the crossing $ \varphi$ in the following sense:  $\omega^{1,\alpha}_X=(\varphi_0)_{\ff_\alpha(X)}^{-1}$ and for all $\alpha,\beta,\gamma \in G$ and $X \in \cc$, the following diagram commutes
\begin{equation*}
\begin{split}
\xymatrix@R=1cm @C=2.8cm { \varphi_\alpha\varphi_\beta\ff_\gamma(X) \ar[d]_-{\varphi_\alpha(\omega_X^{\beta,\gamma})}\ar[r]^-{(\varphi_2)(\alpha,\beta)_{\ff_\gamma(X)}} &
\varphi_{\beta\alpha}\ff_\gamma(X)  \ar[d]^-{\omega_X^{ \beta\alpha,\gamma}}\\
\varphi_\alpha\ff_{\gamma\beta}(X) \ar[r]_-{\omega_X^{\alpha,\gamma \beta}} & \ff_{\beta\alpha}(X).
}
\end{split}
\end{equation*}
\end{thm}
\begin{proof}
Let $\alpha,\beta \in G$ and $X \in \cc$. Pick $V \in \ee_\alpha$,
where $\ee_\alpha$ denotes the class of objects of $\cc_\alpha$ with
invertible left dimension (recall that $\ee_\alpha \neq \emptyset$
since $\cc$ is non-singular). Set $d_V=\dim_l(V)$ and
$$
\psfrag{Z}[Bl][Bl]{\scalebox{.8}{$Z_{\beta\alpha}(X)$}}
\psfrag{T}[Br][Br]{\scalebox{.8}{$Z_\beta(X)$}}
\psfrag{Y}[Bl][Bl]{\scalebox{.8}{$V$}}
\psfrag{S}[Br][Br]{\scalebox{.8}{$V$}}
\psfrag{r}[Bc][Bc]{\scalebox{.9}{$Z_2(\alpha,\beta)_X$}}
\psfrag{a}[Bc][Bc]{\scalebox{.8}{$\alpha$}}
a^{V,\beta}_X=d_V^{-1}  \rsdraw{.45}{.9}{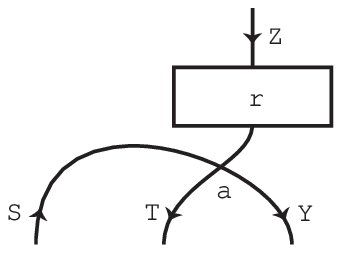} \; : V^* \otimes Z_\beta(X) \otimes V \to Z_{\beta\alpha}(X)
$$
and 
$$\psfrag{R}[Bl][Bl]{\scalebox{.8}{$Z_\beta(X)$}}
\psfrag{L}[Br][Br]{\scalebox{.8}{$Z_{\beta\alpha}(X)$}}
\psfrag{Y}[Bl][Bl]{\scalebox{.8}{$V$}}
\psfrag{S}[Br][Br]{\scalebox{.8}{$V$}}
\psfrag{c}[Bc][Bc]{\scalebox{.9}{$Z_2(\alpha^{-1},\beta\alpha)_X$}}
\psfrag{b}[Bc][Bc]{\scalebox{.8}{$\alpha^{-1}$}}
b^{V,\beta}_X=  \rsdraw{.45}{.9}{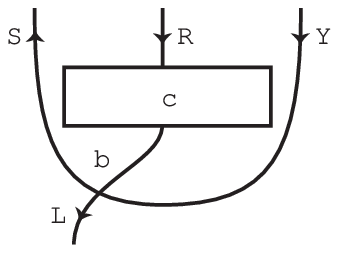} \; : Z_{\beta\alpha}(X) \to V^* \otimes Z_\beta(X) \otimes V.
$$
Observe that
\begin{center}
\vspace{.5em} $\ds \psfrag{Z}[Bl][Bl]{\scalebox{.8}{$Z_\alpha(X)$}}
\psfrag{T}[Br][Br]{\scalebox{.8}{$Z_\beta(X)$}}
\psfrag{X}[Br][Br]{\scalebox{.8}{$X$}}
\psfrag{Y}[Bl][Bl]{\scalebox{.8}{$V$}}
\psfrag{A}[Bl][Bl]{\scalebox{.8}{$Y$}}
\psfrag{Y}[Bl][Bl]{\scalebox{.8}{$V$}}
\psfrag{S}[Br][Br]{\scalebox{.8}{$V$}}
\psfrag{r}[Bc][Bc]{\scalebox{.9}{$Z_2(\alpha,\beta)_X$}}
\psfrag{a}[Bc][Bc]{\scalebox{.8}{$\alpha$}}
\psfrag{R}[Bl][Bl]{\scalebox{.8}{$Z_\beta(X)$}}
\psfrag{L}[Br][Br]{\scalebox{.8}{$Z_\alpha(X)$}}
\psfrag{Y}[Bl][Bl]{\scalebox{.8}{$V$}}
\psfrag{S}[Br][Br]{\scalebox{.8}{$V$}}
\psfrag{c}[Bc][Bc]{\scalebox{.9}{$Z_2(\alpha^{-1},\beta\alpha)_X$}}
\psfrag{b}[Bc][Bc]{\scalebox{.8}{$\alpha^{-1}$}}
b^{V,\beta}_Xa^{V,\beta}_X \, = \, d_V^{-1} \hspace{-1em}
\rsdraw{.45}{.9}{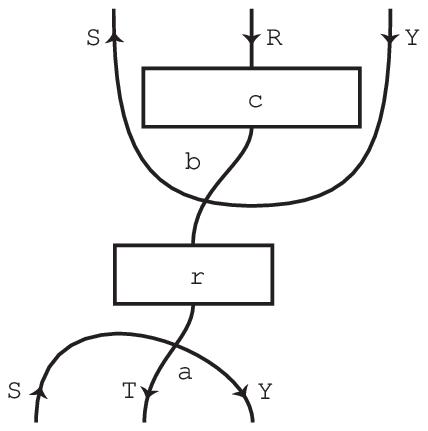}
 \psfrag{r}[Bc][Bc]{\scalebox{.9}{$Z_{\alpha^{-1}}\bigl(Z_2(\alpha,\beta)_X\bigr)$}}
\, \overset{(i)}{=} \, d_V^{-1} \hspace{-1em}
\rsdraw{.45}{.9}{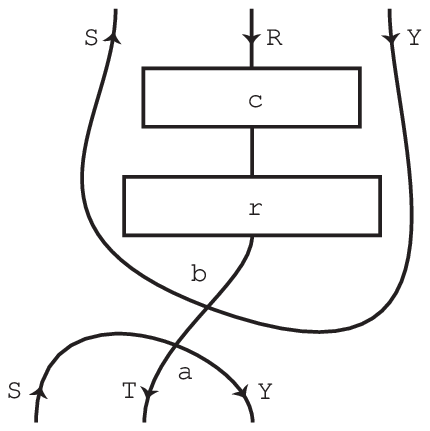}$
\\[1em]
$\ds \psfrag{Z}[Bl][Bl]{\scalebox{.8}{$Z_\alpha(X)$}}
\psfrag{T}[Br][Br]{\scalebox{.8}{$Z_\beta(X)$}}
\psfrag{X}[Br][Br]{\scalebox{.8}{$X$}}
\psfrag{Y}[Bl][Bl]{\scalebox{.8}{$V$}}
\psfrag{A}[Bl][Bl]{\scalebox{.8}{$Y$}}
\psfrag{Y}[Bl][Bl]{\scalebox{.8}{$V$}}
\psfrag{S}[Br][Br]{\scalebox{.8}{$V$}}
\psfrag{r}[Bc][Bc]{\scalebox{.9}{$Z_2(\alpha^{-1},\alpha)_{Z_\beta(X)}$}}
\psfrag{a}[Bc][Bc]{\scalebox{.8}{$\alpha$}}
\psfrag{R}[Bl][Bl]{\scalebox{.8}{$Z_\beta(X)$}}
\psfrag{L}[Br][Br]{\scalebox{.8}{$Z_\alpha(X)$}}
\psfrag{Y}[Bl][Bl]{\scalebox{.8}{$V$}}
\psfrag{t}[Bc][Bc]{\scalebox{.8}{$1$}}
\psfrag{S}[Br][Br]{\scalebox{.8}{$V$}}
\psfrag{c}[Bc][Bc]{\scalebox{.9}{$Z_2(1,\beta)_X$}}
\psfrag{b}[Bc][Bc]{\scalebox{.8}{$\alpha^{-1}$}}
 \psfrag{i}[Bc][Bc]{\scalebox{.8}{$\id_{V^* \otimes V}$}}
\, \overset{(ii)}{=}\, d_V^{-1} \hspace{-1em}
\rsdraw{.45}{.9}{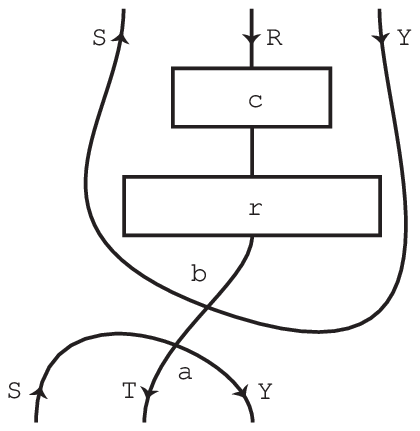}
\, \overset{(iii)}{=} \, d_V^{-1}   \rsdraw{.45}{.9}{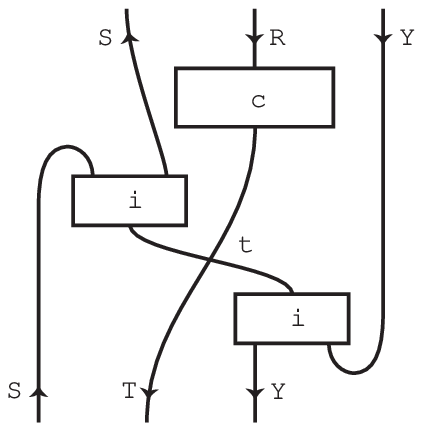}$
\\[1em]
$\ds \psfrag{Z}[Bl][Bl]{\scalebox{.8}{$Z_\alpha(X)$}}
\psfrag{T}[Br][Br]{\scalebox{.8}{$Z_\beta(X)$}}
\psfrag{X}[Br][Br]{\scalebox{.8}{$X$}}
\psfrag{Y}[Bl][Bl]{\scalebox{.8}{$V$}}
\psfrag{A}[Bl][Bl]{\scalebox{.8}{$Y$}}
\psfrag{Y}[Bl][Bl]{\scalebox{.8}{$V$}}
\psfrag{S}[Br][Br]{\scalebox{.8}{$V$}}
\psfrag{r}[Bc][Bc]{\scalebox{.9}{$Z_2(\alpha^{-1},\alpha)_{Z_\beta(X)}$}}
\psfrag{a}[Bc][Bc]{\scalebox{.8}{$\alpha$}}
\psfrag{R}[Bl][Bl]{\scalebox{.8}{$Z_\beta(X)$}}
\psfrag{L}[Br][Br]{\scalebox{.8}{$Z_\alpha(X)$}}
\psfrag{Y}[Bl][Bl]{\scalebox{.8}{$V$}}
\psfrag{t}[Bc][Bc]{\scalebox{.8}{$1$}}
\psfrag{S}[Br][Br]{\scalebox{.8}{$V$}}
\psfrag{c}[Bc][Bc]{\scalebox{.9}{$\sigma^\beta_{X,V \otimes V^*}$}}
\psfrag{b}[Bc][Bc]{\scalebox{.8}{$\alpha^{-1}$}}
 \psfrag{i}[Bc][Bc]{\scalebox{.8}{$\id_{V^* \otimes V}$}}
\, \overset{(iv)}{=}\, d_V^{-1} \rsdraw{.45}{.9}{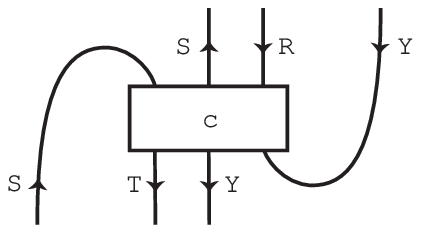} = \,
\pi^V_{\ff_\beta(X)}.$ \vspace{.5em}
\end{center}
Here,  $(i)$ follows from the naturality of
$\partial^{\alpha^{-1}}$,  $(ii)$ from the monoidality of $Z$,
 $(iii)$ from the definition of $Z_2(\alpha^{-1},\alpha)$, and $(iv)$ from the definition of $\sigma^\beta_{X,V \otimes V^*}$.

 We claim that
$ a^{V,\beta}_X b^{V,\beta}_X=\id_{Z_{\beta\alpha}(X)} $.  Indeed,
for any $Y \in \cc_{\beta\alpha}$,
\begin{center}
\vspace{.5em}
$\ds (\id_Y \otimes a^{V,\beta}_X b^{V,\beta}_X)\partial_{X,Y}^{\beta\alpha}
\psfrag{Z}[Bl][Bl]{\scalebox{.8}{$Z_{\beta\alpha}(X)$}}
\psfrag{T}[Br][Br]{\scalebox{.8}{$Z_\beta(X)$}}
\psfrag{X}[Br][Br]{\scalebox{.8}{$X$}}
\psfrag{Y}[Bl][Bl]{\scalebox{.8}{$V$}}
\psfrag{A}[Bl][Bl]{\scalebox{.8}{$Y$}}
\psfrag{Y}[Bl][Bl]{\scalebox{.8}{$V$}}
\psfrag{S}[Br][Br]{\scalebox{.8}{$V$}}
\psfrag{r}[Bc][Bc]{\scalebox{.9}{$Z_2(\alpha,\beta)_X$}}
\psfrag{a}[Bc][Bc]{\scalebox{.8}{$\alpha$}}
\psfrag{d}[Bc][Bc]{\scalebox{.8}{$\beta\alpha$}}
\psfrag{R}[Bl][Bl]{\scalebox{.8}{$Z_\beta(X)$}}
\psfrag{L}[Br][Br]{\scalebox{.8}{$Z_\alpha(X)$}}
\psfrag{Y}[Bl][Bl]{\scalebox{.8}{$V$}}
\psfrag{S}[Br][Br]{\scalebox{.8}{$V$}}
\psfrag{c}[Bc][Bc]{\scalebox{.9}{$Z_2(\alpha^{-1},\beta\alpha)_X$}}
\psfrag{b}[Bc][Bc]{\scalebox{.8}{$\alpha^{-1}$}}
 =  d_V^{-1}  \!\rsdraw{.45}{.9}{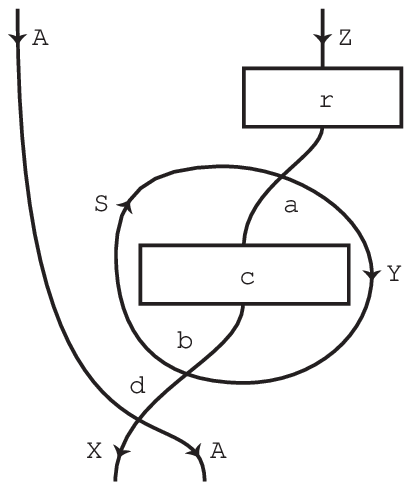}
 \psfrag{c}[Bc][Bc]{\scalebox{.9}{$Z_\alpha\bigl(Z_2(\alpha^{-1},\beta\alpha)_X\bigr)$}}
\, \overset{(i)}{=} \, d_V^{-1} \!\!\!\! \rsdraw{.45}{.9}{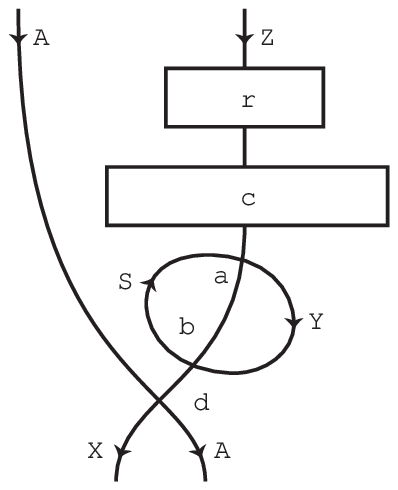}$
\\[1em]
$\ds \psfrag{Z}[Bl][Bl]{\scalebox{.8}{$Z_{\beta\alpha}(X)$}}
\psfrag{T}[Br][Br]{\scalebox{.8}{$Z_\beta(X)$}}
\psfrag{X}[Br][Br]{\scalebox{.8}{$X$}}
\psfrag{Y}[Bl][Bl]{\scalebox{.8}{$V$}}
\psfrag{A}[Bl][Bl]{\scalebox{.8}{$Y$}}
\psfrag{Y}[Bl][Bl]{\scalebox{.8}{$V$}}
\psfrag{S}[Br][Br]{\scalebox{.8}{$V$}}
\psfrag{r}[Bc][Bc]{\scalebox{.9}{$Z_2(1,\beta\alpha)_X$}}
\psfrag{a}[Bc][Bc]{\scalebox{.8}{$\alpha$}}
\psfrag{t}[Bc][Bc]{\scalebox{.8}{$1$}}
\psfrag{d}[Bc][Bc]{\scalebox{.8}{$\beta\alpha$}}
\psfrag{R}[Bl][Bl]{\scalebox{.8}{$Z_\beta(X)$}}
\psfrag{L}[Br][Br]{\scalebox{.8}{$Z_\alpha(X)$}}
\psfrag{Y}[Bl][Bl]{\scalebox{.8}{$V$}}
\psfrag{S}[Br][Br]{\scalebox{.8}{$V$}}
\psfrag{c}[Bc][Bc]{\scalebox{.9}{$Z_2(\alpha^{-1},\alpha)_X$}}
\psfrag{b}[Bc][Bc]{\scalebox{.8}{$\alpha^{-1}$}}
 \psfrag{c}[Bc][Bc]{\scalebox{.9}{$Z_2(\alpha,\alpha^{-1})_{Z_{\beta\alpha}(X)}$}}
 \psfrag{B}[Bl][Bl]{\scalebox{.8}{$V^*\otimes V$}}
 \psfrag{i}[Bc][Bc]{\scalebox{.8}{$\id_{V^* \otimes V}$}}
\,\overset{(ii)}{=}\, d_V^{-1}  \rsdraw{.45}{.9}{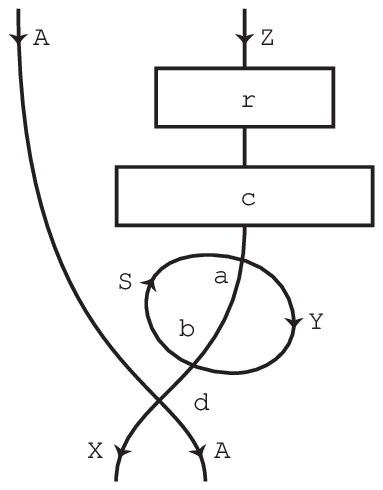}
\,\overset{(iii)}{=}\, d_V^{-1}  \rsdraw{.45}{.9}{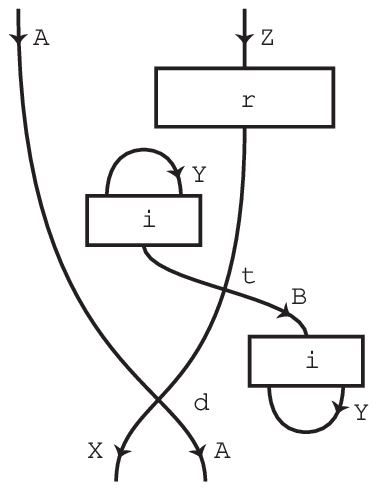}$
\\[1em]
$\ds \psfrag{Z}[Bl][Bl]{\scalebox{.8}{$Z_{\beta\alpha}(X)$}}
\psfrag{T}[Br][Br]{\scalebox{.8}{$Z_\beta(X)$}}
\psfrag{X}[Br][Br]{\scalebox{.8}{$X$}}
\psfrag{Y}[Bl][Bl]{\scalebox{.8}{$V$}}
\psfrag{A}[Bl][Bl]{\scalebox{.8}{$Y$}}
\psfrag{Y}[Bl][Bl]{\scalebox{.8}{$V$}}
\psfrag{S}[Br][Br]{\scalebox{.8}{$V$}}
\psfrag{r}[Bc][Bc]{\scalebox{.9}{$Z_2(1,\beta\alpha)_X$}}
\psfrag{a}[Bc][Bc]{\scalebox{.8}{$\alpha$}}
\psfrag{t}[Bc][Bc]{\scalebox{.8}{$1$}}
\psfrag{d}[Bc][Bc]{\scalebox{.8}{$\beta\alpha$}}
\psfrag{R}[Bl][Bl]{\scalebox{.8}{$Z_\beta(X)$}}
\psfrag{L}[Br][Br]{\scalebox{.8}{$Z_\alpha(X)$}}
\psfrag{Y}[Bl][Bl]{\scalebox{.8}{$V$}}
\psfrag{S}[Br][Br]{\scalebox{.8}{$V$}}
\psfrag{c}[Bc][Bc]{\scalebox{.9}{$Z_2(\alpha^{-1},\alpha)_X$}}
\psfrag{s}[Bc][Bc]{\scalebox{.9}{$(Z_0)_{Z_{\beta\alpha}(X)}$}}
\psfrag{b}[Bc][Bc]{\scalebox{.8}{$\alpha^{-1}$}}
 \psfrag{c}[Bc][Bc]{\scalebox{.9}{$Z_2(\alpha^{-1},\alpha)_{Z_\alpha(X)}$}}
 \psfrag{B}[Bl][Bl]{\scalebox{.8}{$V^*\otimes V$}}
 \psfrag{i}[Bc][Bc]{\scalebox{.8}{$\id_{V^* \otimes V}$}}
\,\overset{(iv)}{=} \,d_V^{-1}  \;\,\rsdraw{.45}{.9}{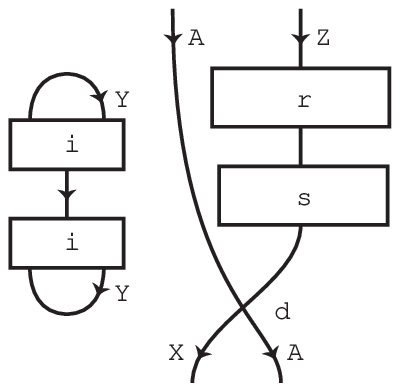}
\,\overset{(v)}{=}\; \partial^{\beta\alpha}_{X,Y}.$
\vspace{.5em}
\end{center}
Here the equality $(i)$ follows from the naturality of $\partial^{\beta\alpha}$,
 $(ii)$ and  $(v)$ are consequences of the monoidality of $Z$,
  $(iii)$ follows from the definition of $Z_2(\alpha,\alpha^{-1})$, and
   $(iv)$ is a consequence of the naturality of $\partial^1$ and the definition of $Z_0$. Now,
    the universal property of
    $\rho^{\beta\alpha}_{X,Y}=(\lev_Y \otimes \id_{Z_{\beta\alpha}(X)})(\id_{Y^*} \otimes \partial^{\beta\alpha}_{X,Y})$ implies
     that $a^{V,\beta}_X b^{V,\beta}_X=\id_{Z_{\beta\alpha}(X)}$.

     The latter equality and the
formula  $ b^{V,\beta}_X a^{V,\beta}_X=\pi^V_{\ff_\beta(X)}$
mean that $Z_{\beta\alpha}(X)$ is the image of the idempotent
$\pi^V_{\ff_\beta(X)}$. Set
\begin{equation}\label{eq-omegaVa}
\ds \psfrag{Z}[Bl][Bl]{\scalebox{.8}{$Z_{\beta\alpha}(X)$}}
\psfrag{E}[Bl][Bl]{\scalebox{.8}{$E^V_{\ff_\beta(X)}$}}
\psfrag{T}[Br][Br]{\scalebox{.8}{$V$}}
\psfrag{V}[Bl][Bl]{\scalebox{.8}{$V$}}
\psfrag{a}[Bc][Bc]{\scalebox{.8}{$\alpha$}}
\psfrag{r}[Bc][Bc]{\scalebox{.9}{$Z_2(\alpha,\beta)_X$}}
\omega^{V,\beta}_X= a^{V,\beta}_X \,q^V_{\ff_\beta(X)}= \, d_V^{-1}  \; \rsdraw{.45}{.9}{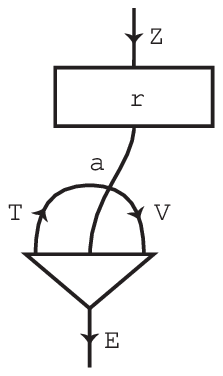} \co E^V_{\ff_\beta(X)} \to Z_{\beta\alpha}(X).
\end{equation}
By Section~\ref{sect-split-idem}, $\omega^{V,\beta}_X$ is the unique isomorphism $E^V_{\ff_\beta(X)} \to Z_{\beta\alpha}(X)$ such that
\begin{equation}\label{eq-ab-pq}
p^V_{\ff_\beta(X)}=\bigl(\omega^{V,\beta}_X\bigr)^{-1} a^{V,\beta}_X \quad \text{and} \quad q^V_{\ff_\beta(X)}=b^{V,\beta}_X \omega^{V,\beta}_X.
\end{equation}
Note that
$\bigl(\omega^{V,\beta}_X\bigr)^{-1}=p^V_{\ff_\beta(X)}\,b^{V,\beta}_X$.
We can now state the key lemma of the proof.

\begin{lem}\label{lem-pptes-omegaVb}
\begin{enumerate}
\labela
\item $\omega^{V,\beta}_X\co \varphi_V(\ff_\beta(X)) \to \ff_{\beta\alpha}(X)$ is an isomorphism in $\zz_G(\cc)$.
\item $\omega^U_X\delta^{U,V}_{\ff_1(X)}=\omega^V_X$ for all $U,V \in \ee_\alpha$,
\end{enumerate}
\end{lem}
\begin{proof}
We use the pictorial formalism of Section~\ref{sect-proof-lems-defdouble}. For $U,V \in \ee_\alpha$ and $Y\in \cc_1$, set
$$
 \psfrag{U}[Bl][Bl]{\scalebox{.8}{$U$}}
 \psfrag{V}[Br][Br]{\scalebox{.8}{$V$}}
 \psfrag{Y}[Bl][Bl]{\scalebox{.8}{$Y$}}
 \psfrag{E}[Bl][Bl]{\scalebox{.8}{$E_{\ff_\beta(X)}^V$}}
 \psfrag{F}[Bl][Bl]{\scalebox{.8}{$E_{\ff_\beta(X)}^U$}}
\lambda^{U,V,\beta}_{X,Y}= d_V^{-1}\;\rsdraw{.45}{.9}{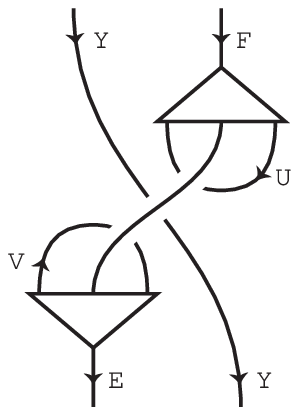}.
$$
Then
\begin{equation}\label{eq-omegaVb}
(\id_Y \otimes \omega^{U,\beta}_X)\lambda^{U,V,\beta}_{X,Y}=\sigma^{\beta\alpha}_{X,Y}(\omega^{V,\beta}_X \otimes \id_Y).
\end{equation}
Indeed,
\begin{center}
\vspace{.5em}
$\ds (\id_Y \otimes \omega^{U,\beta}_X)\lambda^{U,V,\beta}_{X,Y}
\psfrag{Z}[Bl][Bl]{\scalebox{.8}{$Z_{\beta\alpha}(X)$}}
\psfrag{Y}[Bl][Bl]{\scalebox{.8}{$Y$}}
\psfrag{E}[Bl][Bl]{\scalebox{.8}{$E^V_{\ff_\beta(X)}$}}
\psfrag{G}[Bl][Bl]{\scalebox{.8}{$E^U_{\ff_\beta(X)}$}}
\psfrag{T}[Br][Br]{\scalebox{.8}{$V$}}
\psfrag{S}[Br][Br]{\scalebox{.8}{$U$}}
\psfrag{V}[Bl][Bl]{\scalebox{.8}{$V$}}
\psfrag{U}[Bl][Bl]{\scalebox{.8}{$U$}}
\psfrag{a}[Bc][Bc]{\scalebox{.8}{$\alpha$}}
\psfrag{t}[Bc][Bc]{\scalebox{.8}{$1$}}
\psfrag{r}[Bc][Bc]{\scalebox{.9}{$Z_2(\alpha,\beta)_X$}}
\psfrag{c}[Bc][Bc]{\scalebox{.9}{$Z_2(1,1)_X$}}
\psfrag{i}[Bc][Bc]{\scalebox{.8}{$\id$}}
\; =\, d_U^{-1}d_V^{-1}  \!\!\! \!\!\!\rsdraw{.45}{.9}{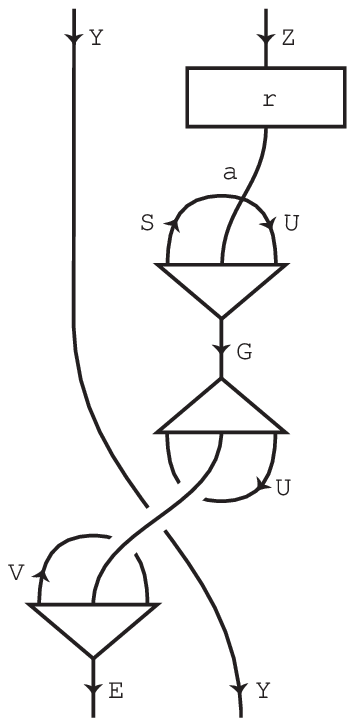}
\, \overset{(i)}{=}\, d_U^{-1}d_V^{-1}  \!\!\!\!\!\!\!\!\rsdraw{.45}{.9}{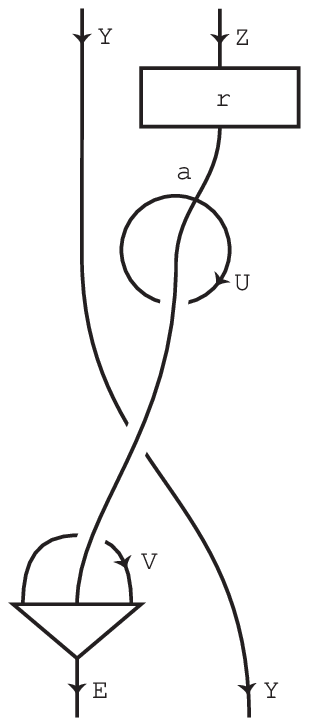}$
\\[1em]
$\ds \psfrag{Z}[Bl][Bl]{\scalebox{.8}{$Z_{\beta\alpha}(X)$}}
\psfrag{Y}[Bl][Bl]{\scalebox{.8}{$Y$}}
\psfrag{E}[Bl][Bl]{\scalebox{.8}{$E^V_{\ff_\beta(X)}$}}
\psfrag{T}[Br][Br]{\scalebox{.8}{$V$}}
\psfrag{S}[Br][Br]{\scalebox{.8}{$U$}}
\psfrag{V}[Bl][Bl]{\scalebox{.8}{$V$}}
\psfrag{U}[Bl][Bl]{\scalebox{.8}{$U$}}
\psfrag{a}[Bc][Bc]{\scalebox{.8}{$\alpha$}}
\psfrag{t}[Bc][Bc]{\scalebox{.8}{$1$}}
\psfrag{r}[Bc][Bc]{\scalebox{.9}{$Z_2(\alpha,\beta)_X$}}
\psfrag{c}[Bc][Bc]{\scalebox{.9}{$Z_2(1,\beta)_X$}}
\psfrag{i}[Bc][Bc]{\scalebox{.8}{$\id$}}
\, \overset{(ii)}{=}\, d_U^{-1}d_V^{-1}  \;\; \rsdraw{.45}{.9}{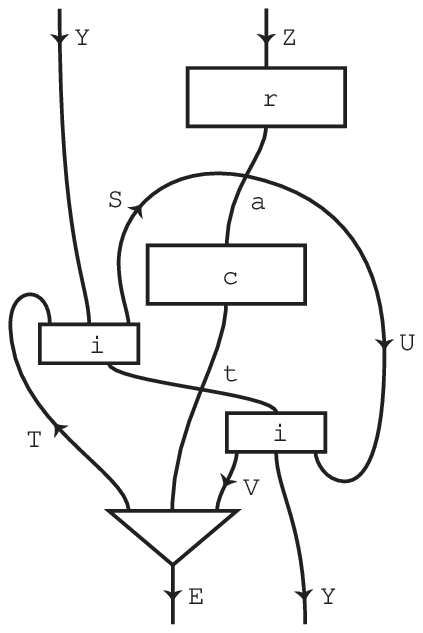}
\psfrag{c}[Bc][Bc]{\scalebox{.9}{$Z_\alpha(Z_2(1,\beta)_X)$}}
\, \overset{(iii)}{=}\, d_U^{-1}d_V^{-1}  \;\; \rsdraw{.45}{.9}{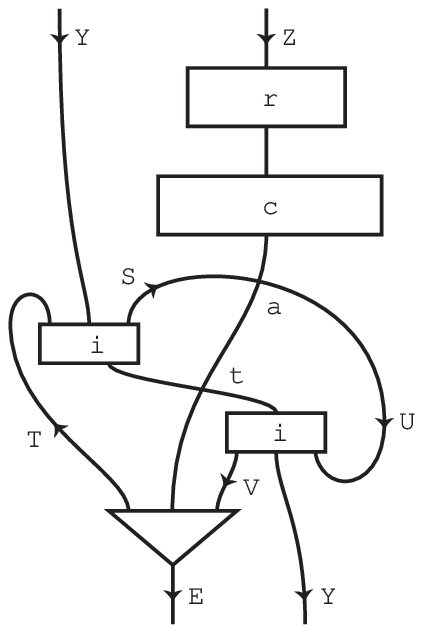}$
\\[1em]
$\ds \psfrag{Z}[Bl][Bl]{\scalebox{.8}{$Z_{\beta\alpha}(X)$}}
\psfrag{Y}[Bl][Bl]{\scalebox{.8}{$Y$}}
\psfrag{E}[Bl][Bl]{\scalebox{.8}{$E^V_{\ff_\beta(X)}$}}
\psfrag{T}[Br][Br]{\scalebox{.8}{$V$}}
\psfrag{S}[Br][Br]{\scalebox{.8}{$U$}}
\psfrag{V}[Bl][Bl]{\scalebox{.8}{$V$}}
\psfrag{U}[Bl][Bl]{\scalebox{.8}{$U$}}
\psfrag{a}[Bc][Bc]{\scalebox{.8}{$\alpha$}}
\psfrag{t}[Bc][Bc]{\scalebox{.8}{$1$}}
\psfrag{r}[Bc][Bc]{\scalebox{.9}{$Z_2(\alpha,\beta)_X$}}
\psfrag{c}[Bc][Bc]{\scalebox{.9}{$Z_2(\alpha,1)_{Z_\beta(X)}$}}
\psfrag{i}[Bc][Bc]{\scalebox{.8}{$\id$}}
\, \overset{(iv)}{=}\, d_U^{-1}d_V^{-1}  \;\; \rsdraw{.45}{.9}{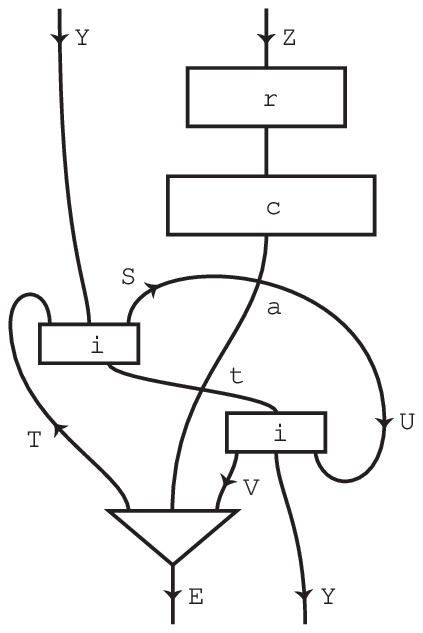}
\, \overset{(v)}{=}\, d_U^{-1}d_V^{-1}  \;\; \rsdraw{.45}{.9}{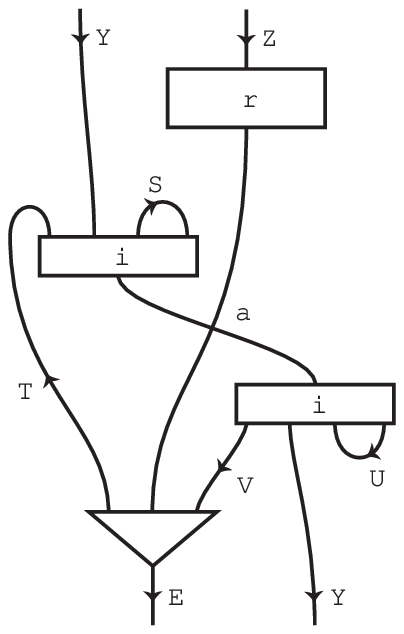}$
\\[1em]
$\ds \psfrag{Z}[Bl][Bl]{\scalebox{.8}{$Z_{\beta\alpha}(X)$}}
\psfrag{Y}[Bl][Bl]{\scalebox{.8}{$Y$}}
\psfrag{E}[Bl][Bl]{\scalebox{.8}{$E^V_{\ff_\beta(X)}$}}
\psfrag{T}[Br][Br]{\scalebox{.8}{$V$}}
\psfrag{S}[Br][Br]{\scalebox{.8}{$U$}}
\psfrag{V}[Bl][Bl]{\scalebox{.8}{$V$}}
\psfrag{U}[Bl][Bl]{\scalebox{.8}{$U$}}
\psfrag{a}[Bc][Bc]{\scalebox{.8}{$\alpha$}}
\psfrag{t}[Bc][Bc]{\scalebox{.8}{$1$}}
\psfrag{r}[Bc][Bc]{\scalebox{.9}{$Z_2(\alpha,\beta)_X$}}
\psfrag{c}[Bc][Bc]{\scalebox{.9}{$Z_2(1,\alpha)_{Z_\beta(X)}$}}
\psfrag{i}[Bc][Bc]{\scalebox{.8}{$\id$}}
\, \overset{(vi)}{=}\, d_V^{-1} \;\,  \rsdraw{.45}{.9}{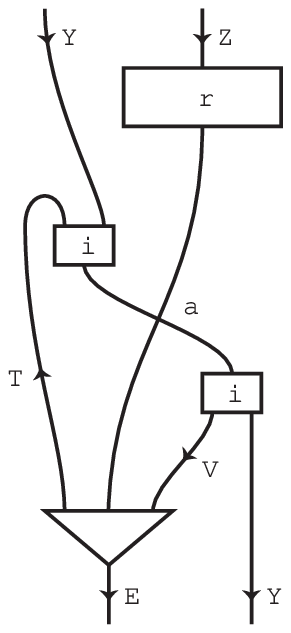}
\, \overset{(vii)}{=}\, d_V^{-1} \!\!\!\!\! \rsdraw{.45}{.9}{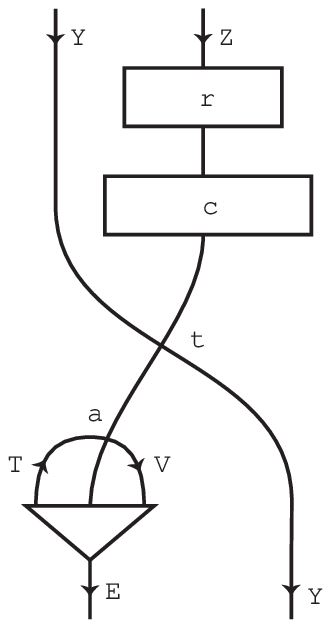}
\!\!\!=\,\; \sigma^{\beta\alpha}_{X,Y}(\omega^{V,\beta}_X \otimes \id_Y)$
\\[1em]
$\ds \psfrag{Z}[Bl][Bl]{\scalebox{.8}{$Z_{\beta\alpha}(X)$}}
\psfrag{Y}[Bl][Bl]{\scalebox{.8}{$Y$}}
\psfrag{E}[Bl][Bl]{\scalebox{.8}{$E^V_{\ff_\beta(X)}$}}
\psfrag{T}[Br][Br]{\scalebox{.8}{$V$}}
\psfrag{S}[Br][Br]{\scalebox{.8}{$U$}}
\psfrag{V}[Bl][Bl]{\scalebox{.8}{$V$}}
\psfrag{U}[Bl][Bl]{\scalebox{.8}{$U$}}
\psfrag{a}[Bc][Bc]{\scalebox{.8}{$\alpha$}}
\psfrag{t}[Bc][Bc]{\scalebox{.8}{$1$}}
\psfrag{r}[Bc][Bc]{\scalebox{.9}{$Z_2(1,\beta\alpha)_X$}}
\psfrag{c}[Bc][Bc]{\scalebox{.9}{$Z_1(Z_2(\alpha,\beta)_X)$}}
\psfrag{i}[Bc][Bc]{\scalebox{.8}{$\id$}}
\, \overset{(viii)}{=}\, d_V^{-1}  \rsdraw{.45}{.9}{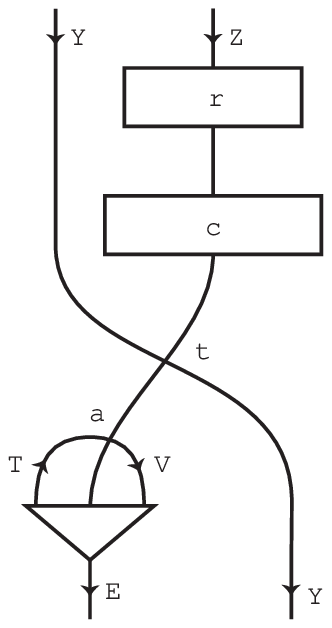}
\psfrag{c}[Bc][Bc]{\scalebox{.9}{$Z_2(\alpha,\beta)_X$}}
\, \overset{(ix)}{=}\, d_V^{-1}  \rsdraw{.45}{.9}{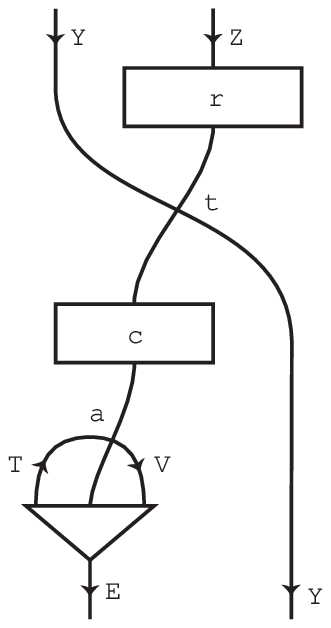}
\,=\,\; \sigma^{\beta\alpha}_{X,Y}(\omega^{V,\beta}_X \otimes \id_Y).$ \vspace{.5em}
\end{center}
Here, the equality $(i)$ is obtained by applying \eqref{eq-Drel1} and then \eqref{eq-dim-sigma},  $(ii)$ follows from the definition of $\sigma^\beta_X$, $(iii)$ and $(vi)$ from the naturality of $\partial^\alpha$,  $(iv)$ and $(viii)$ from the monoidality of $Z$, $(v)$ from the definition of $Z_2(\alpha,1)$,  $(vii)$ from the definition of $Z_2(1,\alpha)$, and $(ix)$ from the naturality of~$\partial^1$.

We now prove the   claims of the lemma. Remark first that
$$
\lambda^{V,V,\beta}_{X,Y}=\gamma^V_{\ff_\beta(X),Y}, \quad \lambda^{U,V,\beta}_{X,\un}=\delta^{U,V}_{\ff_\beta(X)}, \quad \text{and} \quad \sigma^{\beta\alpha}_{X,\un}=\id_{Z_{\beta\alpha}(X)}.
$$
Thus \eqref{eq-omegaVb} for $U=V$ gives that $\omega^{U,\beta}_X$ is a morphism in $\zz_G(\cc)$.
 Since it is an isomorphism in $\cc$ and the forgetful functor $\uu\co \zz_G(\cc) \to \cc$ is conservative,
 we obtain that $\omega^{U,\beta}_X$  is a isomorphism in $\zz_G(\cc)$, which is claim (a).
 Finally, Equation~\eqref{eq-omegaVb} for $Y=\un$ is nothing but claim
 (b).  This completes the proof of
 Lemma~\ref{lem-pptes-omegaVb}.
\end{proof}
By Lemma~\ref{lem-pptes-omegaVb}, the system $(\omega^V_X)_{V \in \ee_\alpha}$ induces an isomorphism in $\zz_G(\cc)$,
$$
\omega_X^{\alpha,\beta} \co
\varphi_\alpha\ff_\beta(X) \to \ff_{\beta\alpha}(X),
$$
related to the universal cone   $\iota^\alpha$ of \eqref{univ-cones} by
\begin{equation*}
\begin{split}
\xymatrix@R=1cm @C=.7cm { \varphi_\alpha\ff_\beta(X) \ar[rd]_-{\omega_X^{\alpha,\beta}}\ar[rr]^-{(\iota^\alpha_V)_{\ff_\beta(X)}} && \varphi_V\ff_\beta(X) \ar[ld]^-{\omega_X^{V,\beta}}\\
& \ff_{\beta\alpha}(X) &
}
\end{split}
\end{equation*}
It remains to check that the family
$$
\omega^{\alpha,\beta}=\{\omega_X^{\alpha,\beta}\co \varphi_\alpha\ff_\beta(X) \to \ff_{\beta\alpha}(X)\}_{X \in \cc}
$$
is a comonoidal natural transformation and that the family $\{\omega^{\alpha,\beta}\}_{\alpha,\beta \in G}$ is compatible with the monoidal structure of $\varphi$ . We only need to verify that for all $\alpha,\beta,\gamma \in G$, $U \in \ee_\alpha$,  $V \in \ee_\beta$, $W \in \ee_{\beta\alpha}$, $R\in \ee_1$,  and any morphism $f\co X \to Y$ in $\cc$, the following equalities hold:
\begin{gather*}
Z_{\beta\alpha}(f) \omega_X^{U,\beta}=\omega_Y^{U,\beta}\varphi_UV(Z_\beta(f)),\\
(Z_{\beta\alpha})_2(X,Y)\omega_{X\otimes Y}^{U,\beta}= (\omega_{X}^{U,\beta} \otimes \omega_{Y}^{U,\beta})(\varphi_U)_2(Z_\beta(X),Z_\beta(Y))^{-1} \varphi_U((Z_\beta)_2(X,Y)),\\
(Z_{\beta\alpha})_0\omega_{\un}^{U,\beta}=(\varphi_U)_0^{-1} \varphi_U((Z_\beta)_0),\\
\omega^{W,\gamma}_X\xi^{U,V,W}_{\ff_\gamma(X)}=\omega^{U,\gamma\beta}_X \varphi_U(\omega^{V,\gamma}_X),\\
\omega^{R,1}_X=\eta_{\ff_\alpha(X)}^R.
\end{gather*}
These equalities are verified  via graphical computations
similar to those  above using the definitions of
$\varphi_V(f)$, $(\varphi_V)_2^{-1}$, $(\varphi_V)_0^{-1}$,
$\xi^{U,V,W}$,   $\eta^R$ given in Sections~\ref{sect-action of-G}
and~\ref{sect-proof-lems-defdouble}.
 \end{proof}

We next    compute   the enhanced $G$-braiding $\{\tau_{(A,\sigma),Y}\}_{(A,\sigma) \in \zz_G(\cc), Y \in \cch}$ of $\zz_G(\cc)$ on the $G$-free objects.

\begin{thm}\label{cor-calc-tau}
Let $\alpha,\beta \in G$,  $X \in \cc$, and $Y \in \cc_\beta$. Then
$$
\tau_{\ff_\alpha(X),Y}=\bigl(\id_Y \otimes \bigl(\omega^{\beta,\alpha}_X\bigr)^{-1} Z_2(\beta,\alpha)_X\bigr)\partial^\beta_{Z_\alpha(X),Y},
$$
where $\partial^\beta$ is defined in \eqref{eq-def-partial}, that is,
$$
\psfrag{Z}[Bl][Bl]{\scalebox{.8}{$Z_{\alpha\beta}(X)$}}
\psfrag{T}[Br][Br]{\scalebox{.8}{$Z_\alpha(X)$}}
\psfrag{Y}[Bl][Bl]{\scalebox{.8}{$Y$}}
\psfrag{S}[Br][Br]{\scalebox{.8}{$Y$}}
\psfrag{r}[Bc][Bc]{\scalebox{.9}{$Z_2(\beta,\alpha)_X$}}
\psfrag{a}[Bc][Bc]{\scalebox{.8}{$\beta$}}
\bigl(\id_Y \otimes  \omega^{\beta,\alpha}_X\bigr)\tau_{\ff_\alpha(X),Y}=\rsdraw{.45}{.9}{sigmaXY}  \;.
$$
\end{thm}

\begin{proof}\label{sect-dem-cor-calc-tau}
For $V \in \ee_\beta$, recall notation $\Gamma^V$ from Section~\ref{sect-braiding in-G}. We have
\begin{center}
\vspace{.5em}
$\ds
\psfrag{Y}[Bl][Bl]{\scalebox{.8}{$Y$}}
\psfrag{R}[Bl][Bl]{\scalebox{.8}{$E^V_{\ff_\alpha(X)}$}}
\psfrag{T}[Br][Br]{\scalebox{.8}{$Z_\alpha(X)$}}
\psfrag{S}[Br][Br]{\scalebox{.8}{$V$}}
\psfrag{V}[Bl][Bl]{\scalebox{.8}{$V$}}
\psfrag{b}[Bc][Bc]{\scalebox{.8}{$\beta^{-1}$}}
\psfrag{l}[Bc][Bc]{\scalebox{.8}{$\beta$}}
\psfrag{t}[Bc][Bc]{\scalebox{.8}{$1$}}
\psfrag{r}[Bc][Bc]{\scalebox{.9}{$Z_2(\beta^{-1},\beta)_{Z_\alpha(X)}$}}
\psfrag{e}[Bc][Bc]{\scalebox{.9}{$Z_{\beta^{-1}}(Z_2(\beta,\alpha)_X)$}}
\psfrag{c}[Bc][Bc]{\scalebox{.9}{$Z_2(1,\alpha)_X$}}
\psfrag{i}[Bc][Bc]{\scalebox{.8}{$\id$}}
\Gamma^V_{\ff_\alpha(X),Y}\, \overset{(i)}{=}\, \rsdraw{.45}{.9}{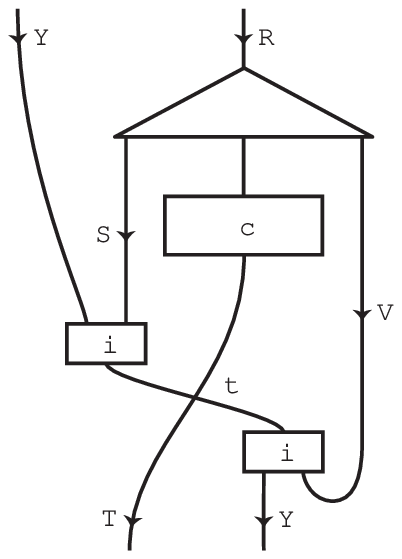} \, \overset{(ii)}{=}\,\rsdraw{.45}{.9}{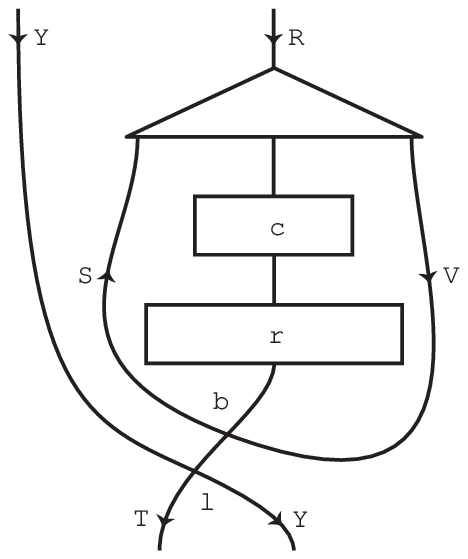}$
\\[1em]
$\ds
\psfrag{Y}[Bl][Bl]{\scalebox{.8}{$Y$}}
\psfrag{R}[Bl][Bl]{\scalebox{.8}{$E^V_{\ff_\alpha(X)}$}}
\psfrag{T}[Br][Br]{\scalebox{.8}{$Z_\alpha(X)$}}
\psfrag{S}[Br][Br]{\scalebox{.8}{$V$}}
\psfrag{V}[Bl][Bl]{\scalebox{.8}{$V$}}
\psfrag{b}[Bc][Bc]{\scalebox{.8}{$\beta^{-1}$}}
\psfrag{l}[Bc][Bc]{\scalebox{.8}{$\beta$}}
\psfrag{t}[Bc][Bc]{\scalebox{.8}{$1$}}
\psfrag{r}[Bc][Bc]{\scalebox{.9}{$Z_{\beta^{-1}}(Z_2(\beta,\alpha)_X)$}}
\psfrag{c}[Bc][Bc]{\scalebox{.9}{$Z_2(\beta^{-1},\alpha\beta)_X$}}
\psfrag{i}[Bc][Bc]{\scalebox{.8}{$\id$}}
\, \overset{(iii)}{=}\,\rsdraw{.45}{.9}{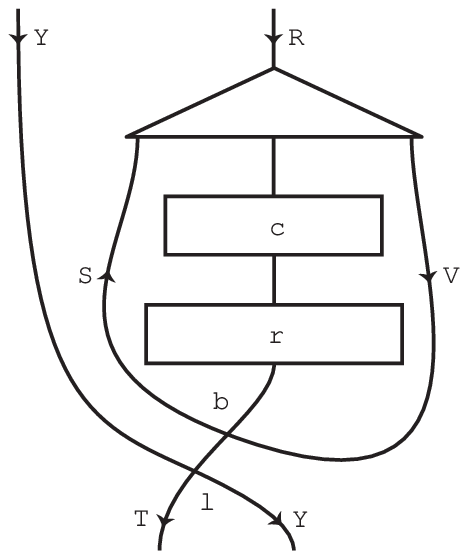}
\psfrag{r}[Bc][Bc]{\scalebox{.9}{$Z_2(\beta,\alpha)_X$}}
\, \overset{(iv)}{=}\,\rsdraw{.45}{.9}{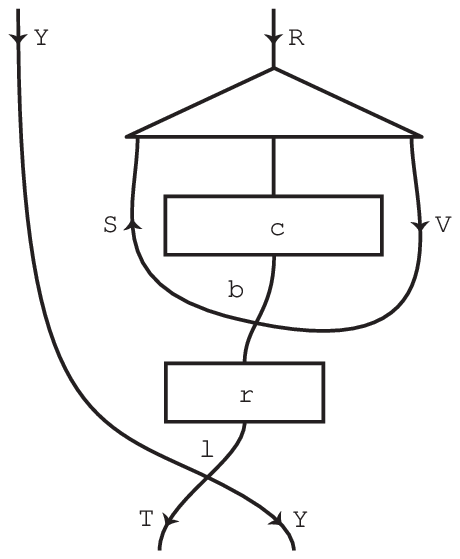}$
\\[1em]
$\ds \,=\, \bigl(\id_Y \otimes \bigl(\omega^{V,\alpha}_X\bigr)^{-1} Z_2(\beta,\alpha)_X\bigr)\partial^\beta_{Z_\alpha(X),Y}.$
\vspace{.5em}
\end{center}
Here, the equality $(i)$  follows from the definition of $\Gamma^V$, $(ii)$  from the definition of $Z_2(\beta^{-1},\beta)$,  $(iii)$   from the monoidality of $Z$, and  $(iv)$  from the  naturality of $\partial^{\beta^{-1}}$. Composing the above equality on the left with $(\id_Y \otimes (\iota^\beta_V)_{\ff_\alpha(X)})$, where $\iota^\beta$ is the universal cone \eqref{univ-cones}, we obtain the claim of the theorem.
\end{proof}

\subsection{Remark}\label{rem-nabla}
Let $\nabla$ be the following composition of monoidal
functors:
$$
\xymatrix@R=1cm @C=1.15cm
 { \Aut(\zz_G(\cc)) \ar[r]^-{\simeq} & \Aut(\cc^{Z_1}) \ar[r] & \Endcm(\cc^{Z_1})
\ar[r]^-{?\rtimes {Z_1}} & \Endcm(\cc).
}
$$
Here    the first arrow is the strict monoidal isomorphism induced by  $\Psi\co
\cc^{Z_1} \to \zz_G(\cc)$,
  the second arrow   is the strict monoidal functor
 acting as $(F,F_2,F_0) \mapsto (F,F^{-1}_2,F^{-1}_0)$ on the objects and as the identity on the
 morphisms, and the third arrow is
  the   functor \eqref{rhorho+}  with $T=Z_1$. Both  $\nabla \varphi$ and $Z$ are   monoidal functor $ \overline{G} \to
\Endcm(\cc)$, and   $\nabla\varphi(\alpha)=\uu\varphi_\alpha \ff_1$ for any $\alpha \in G$. The comonoidal isomorphism $\omega^{\alpha,1}\co \varphi_\alpha \ff_1 \to \ff_\alpha$ of Theorem~\ref{lemu} induces a  comonoidal isomorphism $ \uu(\omega^{\alpha,1})\co \nabla\varphi(\alpha)=\uu\varphi_\alpha \ff_1 \to \uu \ff_\alpha = Z_\alpha$. The  second statement of Theorem~\ref{lemu} implies that the family $ \{\uu(\omega^{\alpha,1})\}_{\alpha \in G}$ is  a monoidal natural isomorphism
$\nabla {\varphi} \simeq Z$.

\subsection{The case of $G$-fusion categories}\label{sect-ovni}
Let $\cc$ be a $G$-fusion category (over   $\kk$). It is clear that $\cc$ satisfies all the assumptions of this section. We give here explicit  formulas for the functor $Z$ of
Lemma~\ref{lem-Z-def-global}.

Fix a   $G$-representative set $I$ of simple
objects of $\cc$ such that $\un \in I$.  Note that $I=\amalg_{\alpha\in G} I_\alpha$ where $I_\alpha$ is the set of all elements of $I$ belonging to $\cc_\alpha$.
  By Lemma~\ref{lem-finite-semi--},
$$
Z_\alpha(X)=\bigoplus_{i \in I_\alpha} i^* \otimes X \otimes i,
$$
with   universal dinatural transformation
$$
\rho^\alpha_{X,Y}=\sum_{i \in I_\alpha} \,
 \psfrag{i}[Br][Bl]{\scalebox{.85}{$i$}}
 \psfrag{X}[Bc][Bl]{\scalebox{.85}{$X$}}
 \psfrag{Y}[Bc][Bc]{\scalebox{.85}{$Y$}}
\rsdraw{.45}{.9}{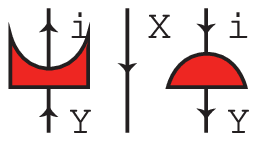} \;
\co Y^* \otimes X \otimes Y \to  Z_\alpha(X) \quad \text{for} \quad Y \in
\cc_\alpha.
$$
For any morphism $f$ in $\cc$, we have $Z_\alpha(f)=\sum_{i \in I_\alpha} \id_{i^*} \otimes f \otimes \id_i$.

The natural transformations $\partial^\alpha$ of \eqref{eq-def-partial} are given, for $X \in \cc$ and $Y \in \cc_\alpha$, by
$$
 \psfrag{A}[Bc][Bc]{\scalebox{.9}{$X$}}
 \psfrag{n}[Bc][Bc]{\scalebox{.9}{$i$}}
 \psfrag{X}[Bc][Bc]{\scalebox{.9}{$Y$}}
\partial^\alpha_{X,Y}= \sum_{i  \in I_\alpha} \; \rsdraw{.50}{.9}{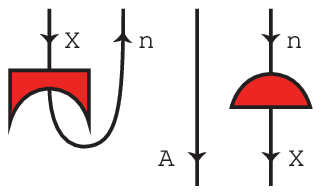} \;.
$$
Then we deduce from Figure~\ref{fig-def-Zgen} that the comonoidal structure of $ Z_\alpha$ and the monoidal structure of $ Z$ are given, for $\alpha,\beta \in G$ and $X,Y \in \cc$, by
\begin{align*}
&\displaystyle ( Z_\alpha)_2(X,Y)=\sum_{i \in I_\alpha}\,
 \psfrag{i}[Br][Bc]{\scalebox{.85}{$i$}}
 \psfrag{X}[Bc][Bc]{\scalebox{.85}{$X$}}
 \psfrag{Y}[Bc][Bc]{\scalebox{.85}{$Y$}}
\rsdraw{.45}{.9}{Z-coprod}  :  Z_\alpha(X\otimes Y) \to  Z_\alpha(X)\otimes  Z_\alpha(Y), \\[1em]
&\displaystyle ( Z_\alpha)_0=\sum_{i \in I_\alpha}\,
 \psfrag{i}[Br][Bc]{\scalebox{.85}{$i$}}
\rsdraw{.25}{.9}{Z-counit} \, :  Z_\alpha(\un) \to \un, \\[.3em]
& \displaystyle  Z_2(\alpha,\beta)_X=\!\!\sum_{\substack{i \in I_\alpha\\ j\in I_\beta\\k \in I_{\beta\alpha}}}\;
 \psfrag{i}[Br][Bc]{\scalebox{.85}{$j$}}
 \psfrag{j}[Br][Bc]{\scalebox{.85}{$i$}}
 \psfrag{k}[Bc][Bc]{\scalebox{.85}{$k$}}
 \psfrag{X}[Bc][Bc]{\scalebox{.85}{$X$}}
\rsdraw{.6}{.9}{Z-prod}   :  Z_\alpha Z_\beta(X ) \to  Z_{\beta\alpha}(X) ,\\[.5em]
& \displaystyle ( Z_0)_X=\id_X \co X \to \un^* \otimes X \otimes \un \hookrightarrow  Z_1(X).
\end{align*}
From these formulas, we deduce that for $\alpha \in G$, the $\alpha$-free functor $\ff_\alpha \co \cc  \to \zz_G(\cc)$ carries any $X \in \cc$ to
$\ff_\alpha(X)=(Z_\alpha(X),\sigma^\alpha_X)$ where for $Y \in \cc_1$,
$$
\sigma^\alpha_{X,Y} = \sum_{i,j \in I_\alpha}\;
 \psfrag{i}[Br][Bc]{\scalebox{.85}{$i$}}
 \psfrag{j}[Br][Bc]{\scalebox{.85}{$j$}}
 \psfrag{k}[Bc][Bc]{\scalebox{.85}{$k$}}
 \psfrag{X}[Bc][Bc]{\scalebox{.85}{$X$}}
  \psfrag{Y}[Bc][Bc]{\scalebox{.85}{$Y$}}
 \rsdraw{.6}{.9}{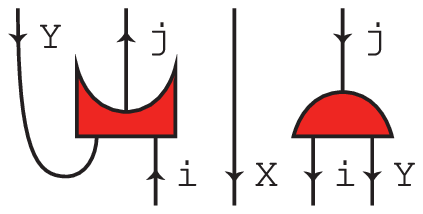}   :  Z_\alpha(X) \otimes Y \to Y \otimes Z_\alpha(X).
$$
The  functor $\ff_\alpha$ carries any morphism $f$ in $\cc$ to $\ff_\alpha(f)=Z_\alpha(f)$.

Let $\alpha,\beta\in G$ and $X \in \cc$. For any $V \in \cc_\alpha$ with invertible left dimension, we always may choose $E_{\ff_\alpha(X)}^V$ to be the object $Z_{\alpha\beta}(X)$ and the morphisms $p_{\ff_\alpha(X)}^V$, $q_{\ff_\alpha(X)}^V$ of \eqref{eq-idemp-split} to be
\begin{gather*}
 \psfrag{i}[Br][Bc]{\scalebox{.85}{$i$}}
 \psfrag{j}[Br][Bc]{\scalebox{.85}{$j$}}
 \psfrag{k}[Bc][Bc]{\scalebox{.85}{$k$}}
 \psfrag{X}[Bc][Bc]{\scalebox{.85}{$X$}}
  \psfrag{Y}[Bc][Bc]{\scalebox{.85}{$V$}}
p_{\ff_\alpha(X)}^V=d_V^{-1}\sum_{\substack{i \in I_\alpha\\ j \in I_{\alpha\beta}}}\; \rsdraw{.6}{.9}{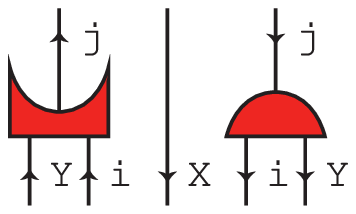} \;,\\[.2em]
 \psfrag{i}[Br][Bc]{\scalebox{.85}{$j$}}
 \psfrag{j}[Br][Bc]{\scalebox{.85}{$i$}}
 \psfrag{k}[Bc][Bc]{\scalebox{.85}{$k$}}
 \psfrag{X}[Bc][Bc]{\scalebox{.85}{$X$}}
  \psfrag{Y}[Bc][Bc]{\scalebox{.85}{$V$}}
q_{\ff_\alpha(X)}^V=\sum_{\substack{i \in I_\alpha\\ j \in I_{\alpha\beta}}}\; \rsdraw{.6}{.9}{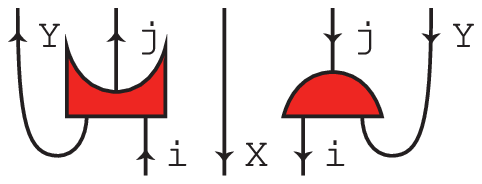}\;.
\end{gather*}
This can be verified   using \eqref{eq-triangle-coev} and the
following equality: for all $U,V \in \cc$ and $i \in I$,
\begin{equation}\label{eq-assoc-tri}
\sum_{j \in I} \;\, \psfrag{X}[Br][Br]{\scalebox{.9}{$U$}}
\psfrag{Y}[Bl][Bl]{\scalebox{.9}{$V$}}
\psfrag{i}[Br][Br]{\scalebox{.9}{$j\,$}}
\psfrag{j}[Bl][Bl]{\scalebox{.9}{$i$}} \rsdraw{.45}{.9}{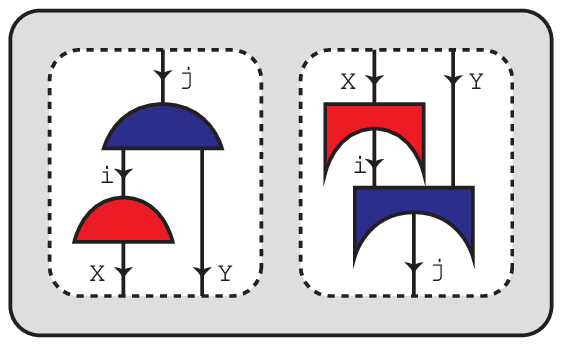} \; = \;
\psfrag{X}[Br][Br]{\scalebox{.9}{$U$}}
\psfrag{Y}[Bl][Bl]{\scalebox{.9}{$V$}}
\psfrag{i}[Br][Br]{\scalebox{.9}{$j$}}
\psfrag{j}[Bl][Bl]{\scalebox{.9}{$i$}}
\rsdraw{.45}{.9}{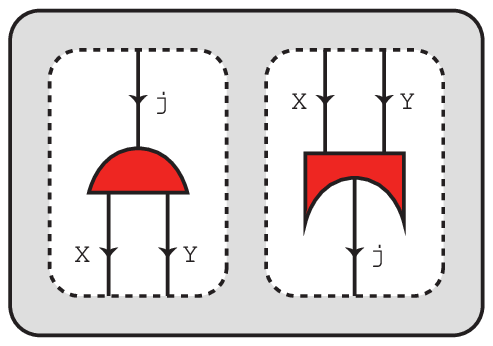} \;,
\end{equation}
which graphically reflects  the fact that
$$
\Hom_\cc(U \otimes V,i)=\bigoplus_{j \in I} \Hom_\cc(X,j) \otimes_\kk \Hom_\cc(j \otimes
V,i).
$$
With the above choices, we obtain that
$\varphi_\beta\ff_\alpha=\ff_{\alpha\beta}$ as comonoidal functors.
Also, the isomorphism $\omega^{\beta,\alpha}_{X}$ of
Theorem~\ref{lemu} is the identity, and for any $Y \in \cc_\beta$,
$$
\tau_{\ff_\alpha(X),Y} = \sum_{\substack{i \in I_\alpha\\ j \in I_{\alpha\beta}}}\;
 \psfrag{i}[Br][Bc]{\scalebox{.85}{$i$}}
 \psfrag{j}[Br][Bc]{\scalebox{.85}{$j$}}
 \psfrag{k}[Bc][Bc]{\scalebox{.85}{$k$}}
 \psfrag{X}[Bc][Bc]{\scalebox{.85}{$X$}}
  \psfrag{Y}[Bc][Bc]{\scalebox{.85}{$Y$}}
 \rsdraw{.6}{.9}{a-free}   \,.
$$

\appendix

\section{The object $(C_{\alpha,\beta},\sigma^{\alpha,\beta})$}\label{sect-objects-Cab}

Let $\cc$ be a non-singular $G$-centralizable
$G$-graded pivotal category. Let $\varphi \co
\overline{G} \to
\Aut(\zz_G(\cc))$ be the crossing provided by
Theorem~\ref{thm-G-center} and $Z\co \overline{G} \to
\Endcm(\cc)$ be the monoidal
 functor of
Lemma~\ref{lem-Z-def-global}.  Consider
  $\alpha, \beta \in G$ such that the following coend exists in $\cc$:
$$
C_{\alpha,\beta}=\int^{Y \in \cc_\beta} Z_\alpha(Y)^* \otimes Y.
$$
Here we   lift   $C_{\alpha,\beta}$ to an object of $\zz_G(\cc)$ in a canonical way. The latter object     plays an important role in   3-dimensional HQFT, as will be discussed elsewhere.

Let $\varrho^{\alpha,\beta}=\{\varrho^{\alpha,\beta}_Y
\co Z_\alpha(Y)^* \otimes Y \to C_{\alpha,\beta}\}_{Y \in
\cc_\beta}$ be the universal
dinatural transformation associated with   $C_{\alpha,\beta}$. By \cite[Corollary 3.8]{BV3}, since $Z_1$ is a Hopf monad,   the dinatural transformation $Z_1(\varrho^{\alpha,\beta})$ is universal. Therefore there is a unique morphism
 $r_{\alpha,\beta}
\co Z_1(C_{\alpha,\beta}) \to C_{\alpha,\beta}$ such that for all
$Y\in\cc_\beta$,
\begin{equation}\label{eq-def-rab}
\begin{gathered}
r_{\alpha,\beta}Z_1(\varrho^{\alpha,\beta}_Y)= \\
 \varrho^{\alpha,\beta}_{Z_1(Y)}
 \bigl ( Z_2(\alpha,1)_Y^*s^l_{Z_\alpha(Y)} Z_1(Z_2(1,\alpha)^*_Y) \otimes \id_{Z_1(Y)} \bigr)(Z_1)_2(Z_\alpha(Y)^* ,Y),
\end{gathered}
\end{equation}
where $s^l$ is the left antipode of $Z_1$ (given by Lemma~\ref{lem-central-HM}(a) for $\dd=\cc_1$).
Define $\sigma^{\alpha,\beta}=\{\sigma^{\alpha,\beta}_X\co C_{\alpha,\beta} \otimes X \to X \otimes C_{\alpha,\beta}\}_{X \in \cc_1}$ by
$$
\psfrag{R}[Br][Br]{\scalebox{.8}{$C_{\alpha,\beta}$}}
\psfrag{C}[Bl][Bl]{\scalebox{.8}{$C_{\alpha,\beta}$}}
\psfrag{X}[Bl][Bl]{\scalebox{.8}{$X$}}
\psfrag{S}[Br][Br]{\scalebox{.8}{$X$}}
\psfrag{r}[Bc][Bc]{\scalebox{.9}{$r_{\alpha,\beta}$}}
\psfrag{t}[Bc][Bc]{\scalebox{.8}{$1$}}
\sigma^{\alpha,\beta}_{X}= (\id_X \otimes r_{\alpha,\beta})\partial^1_{C_{\alpha,\beta},X}\, =\rsdraw{.45}{.9}{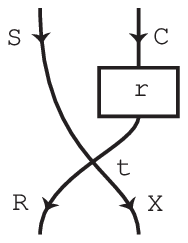} \;,
$$
where $\partial^1$ is defined in \eqref{eq-def-partial}.

 \begin{thm}\label{cor-coend-Cab}
$(C_{\alpha,\beta},\sigma^{\alpha,\beta})$ is an object of
$\zz_G(\cc)$ lying in $\zz_{\alpha^{-1}\beta^{-1}\alpha\beta} (\cc)$
and
$$(C_{\alpha,\beta},\sigma^{\alpha,\beta})=\int^{(A,\sigma) \in
\zz_\beta(\cc)} (\varphi_\alpha(A,\sigma))^* \otimes (A,\sigma).$$
\end{thm}

\begin{proof}
Let $F_\alpha$, $\Psi$, and $\ff_\alpha=\Psi F_\alpha$ be as in Section~\ref{sect-def-free-hb}. Denote by $U_1\co \cc^{Z_1} \to \cc$ and $\uu\co \zz_G(\cc) \to \cc$ the forgetful functors. Set $$Q=\Psi^{-1}\varphi_\alpha\Psi\co \cc^{Z_1} \to \cc^{Z_1}.$$ Observe that
$
Q\rtimes Z_1=U_1\Psi^{-1}\varphi_\alpha\Psi F_1=\uu\varphi_\alpha\ff_1
$. Theorem~\ref{lemu} provides a comonoidal natural isomorphism
$
\omega^{\alpha,1}=\{\omega^{\alpha,1}_Y \co \varphi_\alpha \ff_1(Y) \to \ff_{\alpha}(Y)\}_{Y  \in \cc}.
$
Since $$\{\omega^{\alpha,1}_Y=\uu(\omega^{\alpha,1}_Y)\co (Q\rtimes Z_1)(Y) \to Z_\alpha(Y)\}_{Y \in \cc_\beta}$$ is then a natural isomorphism,
$$
C_{\alpha,\beta}=\int^{Y \in \cc_\beta} Z_\alpha(Y)^* \otimes Y=\int^{Y \in \cc_\beta} (Q\rtimes Z_1)(Y)^* \otimes Y,
$$
with associated dinatural transformation $j=\{j_Y\co (Q\rtimes Z_1)(Y)^* \otimes Y \to C_{\alpha,\beta}\}_{Y \in \cc_\beta}$ given by
$$
j_Y=\varrho^{\alpha,\beta}_Y\bigl( \bigl((\omega^{\alpha,1}_Y)^{-1}\bigr)^* \otimes \id_Y \bigr).
$$
Therefore, by Lemma~\ref{lem-HM-coend} applied to $\dd=\cc_\beta$ and $T=Z_1$, we deduce that
$$
(C_{\alpha,\beta},c)=\int^{(A,r) \in\cc_\beta^{Z_1}} (Q(A,r))^* \otimes (A,r),
$$
where $c$ is some $Z_1$-action on $ C_{\alpha,\beta}$.
Since $\Psi\co \cc^{Z_1}
\to  \zz_G(\cc)$ is a grading-preserving strict monoidal isomorphism of $G$-graded categories,\begin{gather*}
\int^{(A,\sigma) \in
\zz_\beta(\cc)} (\varphi_\alpha(A,\sigma))^* \otimes (A,\sigma)=
\Psi\left ( \int^{(A,r) \in
\cc_\beta^{Z_1}} (\Psi^{-1}\varphi_\alpha\Psi(A,r))^* \otimes (A,r)\right )\\
=\Psi\left ( \int^{(A,r) \in
\cc_\beta^{Z_1}} (Q(A,r))^* \otimes (A,r)\right )
= \Psi(C_{\alpha,\beta},c).
\end{gather*}
Since $\Psi(C_{\alpha,\beta},r_{\alpha,\beta})=(C_{\alpha,\beta},\sigma^{\alpha,\beta})$, we only need to prove that $c=r_{\alpha,\beta}$. The dinatural transformation $Z_1(\varrho^{\alpha,\beta})$ being universal, this is equivalent
to proving that  $cZ_1(\varrho^{\alpha,\beta}_Y)=r_{\alpha,\beta}Z_1(\varrho^{\alpha,\beta}_Y)$
for all $Y \in\cc_\beta$.  Fix $Y\in \cc_\beta$. Let $a_Y$ is the $Z_1$-action of $QF_1(Y)=\Psi^{-1}\varphi_\alpha\Psi F_1(Y)=\Psi^{-1}\varphi_\alpha\ff_1(Y)$. Since
$\omega^{\alpha,1}_Y$ is a morphism is $\zz_G(\cc)$, the morphism $\Psi^{-1}(\omega^{\alpha,1}_Y)=\omega^{\alpha,1}_Y \co QF_1(Y) \to F_\alpha(Y)$ is a morphism in $\cc^{Z_1}$, that is,
\begin{equation*}
 \omega^{\alpha,1}_Y a_Y= Z_2(1,\alpha)_Y Z_1(\omega^{\alpha,1}_Y), \quad \text{and so} \quad a_Y = \bigl(\omega^{\alpha,1}_Y \bigr)^{-1}Z_2(1,\alpha)_Y Z_1(\omega^{\alpha,1}_Y) .
\end{equation*}
By Lemma~\ref{lem-HM-coend},  the $Z_1$-action $c\co Z_1(C_{\alpha,\beta}) \to C_{\alpha,\beta}$ satisfies
$$
cZ_1(j_Y)=
\iota_{Z_1(Y)}\bigr(Q(Z_2(1,1)_Y)^* s^l_{Q\rtimes Z_1(Y)}
Z_1(a_Y^*) \otimes \id_{Z_1(Y)}\bigl)\, Z_2((Q \rtimes T)(Y)^*,Y).
$$
Composing this equality on the right with $Z_1((\omega^{\alpha,1}_Y)^* \otimes \id_Y)$, using the above expression of $a_Y$  and   the naturality of $s^l$, we obtain that
$$
cZ_1(\varrho^{\alpha,\beta}_Y)=\varrho^{\alpha,\beta}_{Z_1(Y)} \bigr(\Delta^*s^l_{Z_\alpha(Y)}
Z_1(Z_2(1,\alpha)^*_Y) \otimes \id_{Z_1(Y)} \bigr)(Z_1)_2(Z_\alpha(Y)^* ,Y)
$$
where $\Delta= \omega^{\alpha,1}_Y  Q(Z_2(1,1)_Y) \bigl(\omega^{\alpha,1}_{Z_1(Y)} \bigr)^{-1}$. Pick $V \in \ee_\alpha$. Since $\Psi$ acts as the identity on morphisms and using that $ \omega^{\alpha,1}= (\iota^\alpha)^{-1}\omega^{V,1}$, and $\varphi_\alpha=(\iota^\alpha)^{-1} \varphi_V\,\iota^\alpha$, where
$\iota^\alpha$ is the universal cone \eqref{univ-cones}, we obtain
$$
\Delta= \omega^{\alpha,1}_Y  \varphi_\alpha(Z_2(1,1)_Y) \bigl(\omega^{\alpha,1}_{Z_1(Y)} \bigr)^{-1}
=\omega^{V,1}_Y  \varphi_V(Z_2(1,1)_Y) \bigl(\omega^{V,1}_{Z_1(Y)} \bigr)^{-1}.
$$
In view of the definition of $r_{\alpha,\beta}$ given in \eqref{eq-def-rab}, to prove that $cZ_1(\varrho^{\alpha,\beta}_Y)=r_{\alpha,\beta}Z_1(\varrho^{\alpha,\beta}_Y)$, it is enough to prove that $\Delta=Z_2(\alpha,1)_Y$, i.e., that $ \omega^{V,1}_Y  \varphi_V(Z_2(1,1)_Y)=Z_2(\alpha,1)_Y\omega^{V,1}_{Z_1(Y)}$. Denote by $\rho^\alpha$  the universal dinatural transformation \eqref{eq-def-rho-of-Z} and let $\pi^V$, $p^V$, $q^V$ be as in \eqref{eq-idemp-split}.
Recall from \eqref{eq-omegaVa} that $\omega^{V,1}_X=d_V^{-1} Z_2(\alpha,1)_X\rho^\alpha_{Z_1(X),V}q^V_{\ff_1(X)}$  for any $X \in \cc$. Now, by Lemma~\ref{lem-Z2-comono}, $Z_2(1,1)_Y$ is a morphism in $\zz_G(\cc)$ from $\ff_1(Z_1(Y))$ to $\ff_1(Y)$. This implies that
\begin{equation}\label{eq-piVZ}
\pi^V_{\ff_1(Y)}(\id_{V^*} \otimes Z_2(1,1)_Y \id_Y)=(\id_{V^*} \otimes Z_2(1,1)_Y \id_Y) \pi^V_{\ff_1(Z_1(Y))}.
\end{equation}
Then
\begin{align*}
 \omega^{V,1}_Y &  \varphi_V(Z_2(1,1)_Y) \\
 &\overset{(i)}{=}  d_V^{-1} Z_2(\alpha,1)_Y\rho^\alpha_{Z_1(Y),V}q^V_{\ff_1(Y)}  p^V_{\ff_1(Y)} (\id_{V^*} \otimes Z_2(1,1)_Y \id_Y) q^V_{\ff_1(Z_1(Y))} \\
  &\overset{(ii)}{=}  d_V^{-1} Z_2(\alpha,1)_Y\rho^\alpha_{Z_1(Y),V}\pi^V_{\ff_1(Y)}  (\id_{V^*} \otimes Z_2(1,1)_Y \id_Y) q^V_{\ff_1(Z_1(Y))}\\
  &\overset{(iii)}{=}  d_V^{-1} Z_2(\alpha,1)_Y\rho^\alpha_{Z_1(Y),V}  (\id_{V^*} \otimes Z_2(1,1)_Y \id_Y)
      \pi^V_{\ff_1(Z_1(Y))}q^V_{\ff_1(Z_1(Y))}\\
  &\overset{(iv)}{=}  d_V^{-1} Z_2(\alpha,1)_Y Z_\alpha(Z_2(1,1)_Y)\rho^\alpha_{Z_1^2(Y),V} q^V_{\ff_1(Z_1(Y))}\\
    &\overset{(v)}{=}  d_V^{-1} Z_2(\alpha,1)_Y Z_2(\alpha,1)_{Z_1(Y)}\rho^\alpha_{Z_1^2(Y),V} q^V_{\ff_1(Z_1(Y))}\\
    &= Z_2(\alpha,1)_Y \omega^{V,1}_{Z_1(Y)}.
\end{align*}
Here, the equality $(i)$  follows from the definition of $\varphi_V$ (see section~\ref{sect-action of-G}), $(ii)$  from \eqref{eq-idemp-split},   $(iii)$  from \eqref{eq-piVZ}, $(iv)$  from \eqref{eq-idemp-split} and the naturality of $\rho^\alpha$, and $(v)$ from the monoidality of $Z$.
This completes the proof of the theorem.
\end{proof}

We can explicitly  compute    $(C_{\alpha,\beta},\sigma^{\alpha,\beta}) $ in the case where $\cc$ is a $G$-fusion category. Fix  a   $G$-representative set $I=\amalg_{\alpha\in G} I_\alpha$ of simple
objects of $\cc$.    By Lemma~\ref{lem-finite-semi--},
\begin{equation}\label{eq-Cab}
C_{\alpha,\beta}=\bigoplus_{\substack{i \in I_\alpha\\ j \in I_\beta}} i^* \otimes j^* \otimes i \otimes j
\end{equation}
with universal dinatural transformation given, for $Y \in
\cc_\beta$, by
$$
\varrho^{\alpha,\beta}_Y=\sum_{\substack{i \in I_\alpha\\ j \in I_\beta}} \,
 \psfrag{i}[Br][Bc]{\scalebox{.85}{$i$}}
 \psfrag{a}[Bc][Bc]{\scalebox{.85}{$i^*$}}
 \psfrag{x}[Bc][Bc]{\scalebox{.85}{$\phi_i^{-1}\,$}}
 \psfrag{j}[Br][Bc]{\scalebox{.85}{$j$}}
 \psfrag{Y}[Bc][Bc]{\scalebox{.85}{$Y$}}
\rsdraw{.45}{.9}{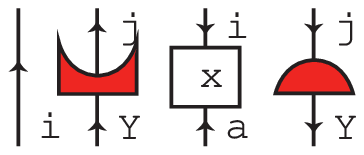} \;
\co  Z_\alpha(Y)^* \otimes Y \to C_{\alpha,\beta},
$$
where $\phi=\{\phi_X\co X \to X^{**}\}_{x \in \cc}$ is the pivotal structure of $\cc$, see \eqref{pivotal-struct}. For $X \in \cc$, set
$$
\zeta_X^\alpha=\sum_{i \in I_\alpha} \id_{i^* \otimes X^*} \otimes \phi_i \co Z_\alpha(X^*)  \to Z_\alpha(X)^*,
$$
so that $$\id_{C_{\alpha,\beta}}=\sum_{j \in
I_\beta}\varrho^{\alpha,\beta}_j(\zeta^\alpha_j \otimes \id_j).$$
Then using \eqref{eq-def-rab}, the above description of
$Z$, and the description of $s^l$ given in
Figure~\ref{fig-strucZ-fusion}, we obtain that
\begin{center}
\vspace{.5em}
$\ds  r_{\alpha,\beta}=r_{\alpha,\beta}Z_1(\id_{C_{\alpha,\beta}}) =\sum_{j \in I_\beta} r_{\alpha,\beta}Z_1(\varrho^{\alpha,\beta}_j)Z_1(\zeta^\alpha_j \otimes \id_j) $
\\[.15em]
$\ds = \sum_{j \in I_\beta} \varrho^{\alpha,\beta}_{Z_1(j)}
 \bigl ( Z_2(\alpha,1)_j^*s^l_{Z_\alpha(j)} Z_1(Z_2(1,\alpha)^*_j) \otimes \id_{Z_1(j)} \bigr)(Z_1)_2(Z_\alpha(j)^* ,j)Z_1(\zeta^\alpha_j \otimes \id_j)$
\\[.2em]
$\ds
\psfrag{i}[Bc][Bc]{\scalebox{.9}{$i$}}
 \psfrag{j}[Bc][Bc]{\scalebox{.9}{$j$}}
 \psfrag{k}[Bc][Bc]{\scalebox{.9}{$k$}}
 \psfrag{l}[Bc][Bc]{\scalebox{.9}{$l$}}
 \psfrag{n}[Bc][Bc]{\scalebox{.9}{$z$}}
  \psfrag{a}[Bc][Bc]{\scalebox{.9}{$a$}}
   \psfrag{b}[Bc][Bc]{\scalebox{.9}{$b$}}
 =\!\!\!\sum_{\substack{i,k,b\in I_\alpha \\ j,l \in I_\beta\\ z,a \in I_1}} \;\; \rsdraw{.50}{.9}{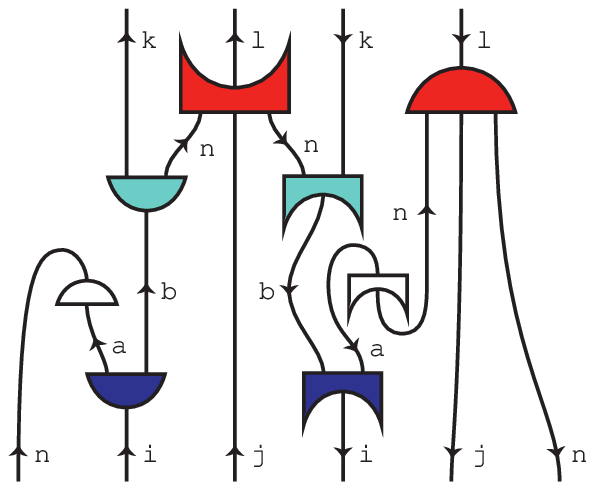} .
$
\vspace{.5em}
\end{center}
Finally, using \eqref{eq-derijd} and \eqref{eq-assoc-tri}, we obtain
$$
\psfrag{i}[Bc][Bc]{\scalebox{.9}{$i$}}
 \psfrag{j}[Bc][Bc]{\scalebox{.9}{$j$}}
 \psfrag{k}[Bc][Bc]{\scalebox{.9}{$k$}}
 \psfrag{l}[Bc][Bc]{\scalebox{.9}{$l$}}
 \psfrag{n}[Bc][Bc]{\scalebox{.9}{$z$}}
r_{\alpha,\beta}=\!\!\sum_{\substack{i,k,\in I_\alpha \\ j,l \in I_\beta\\ z \in I_1}} \;\; \rsdraw{.50}{.9}{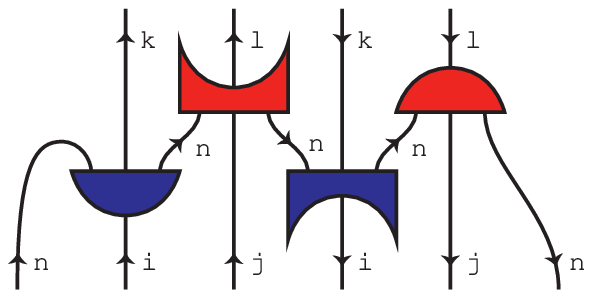}.
$$
Hence $
\sigma^{\alpha,\beta}=\{\sigma^{\alpha,\beta}_X=(\id_X \otimes r_{\alpha,\beta})\partial^1_{C_{\alpha,\beta},X}\co
C_{\alpha,\beta} \otimes X \to X \otimes C_{\alpha,\beta}\}_{X \in
\cc_1}$
is given by
$$
\psfrag{i}[Bc][Bc]{\scalebox{.9}{$i$}}
 \psfrag{j}[Bc][Bc]{\scalebox{.9}{$j$}}
 \psfrag{k}[Bc][Bc]{\scalebox{.9}{$k$}}
 \psfrag{l}[Bc][Bc]{\scalebox{.9}{$l$}}
 \psfrag{n}[Bc][Bc]{\scalebox{.9}{$z$}}
 \psfrag{X}[Bc][Bc]{\scalebox{.9}{$X$}}
 \psfrag{a}[cc][cc]{\scalebox{.9}{$p^\beta_{z^* \otimes j \otimes z}$}}
 \psfrag{s}[cc][cc]{\scalebox{.9}{$q^\beta_{z^* \otimes j \otimes z}$}}
 \psfrag{c}[cc][cc]{\scalebox{.9}{$q_{z \otimes k \otimes z^*}^\alpha$}}
 \psfrag{r}[cc][cc]{\scalebox{.9}{$p_{z \otimes k \otimes z^*}^\alpha$}}
  \psfrag{v}[cc][cc]{\scalebox{.9}{$p_X^\gamma$}}
   \psfrag{u}[cc][cc]{\scalebox{.9}{$q_X^\gamma$}}
\sigma^{\alpha,\beta}_X= \!\!\!\sum_{\substack{i,k\in I_\alpha \\ j,l \in I_\beta\\ z \in I_1}} \; \rsdraw{.50}{.9}{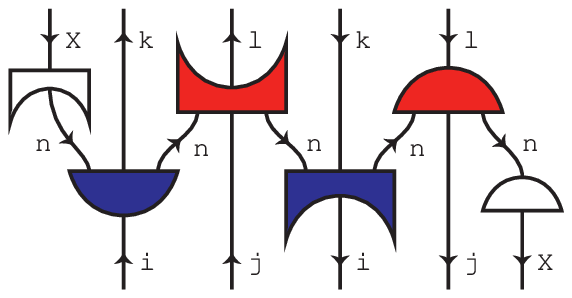}.
$$

\end{document}